\documentclass[article, reqno,10pt]{amsart}
\usepackage{amssymb} 
\usepackage{latexsym}
\usepackage{amsmath}
\usepackage{amsthm}
\usepackage{mathrsfs}
\usepackage{verbatim}
\usepackage{amsfonts}
\usepackage[left=1.25in,top=1in,right=1.25in,bottom=1in]{geometry} 
\usepackage{graphicx}
\usepackage[T1]{fontenc}
\usepackage[toc]{appendix}
\usepackage{url,color}
\usepackage{amsbsy}
\usepackage{mathrsfs}
\usepackage{bbm}
\usepackage{latexsym,verbatim}

\usepackage{tikz}

\catcode`@=11
\def\section{\@startsection{section}{1}%
  \z@{1.5\linespacing\@plus\linespacing}{.5\linespacing}%
  {\normalfont\bfseries\large\centering}}
\catcode`@=12

\numberwithin{equation}{section}

\begin{document}

\sloppy
\renewcommand{\theequation}{\arabic{section}.\arabic{equation}}
\thinmuskip = 0.5\thinmuskip
\medmuskip = 0.5\medmuskip
\thickmuskip = 0.5\thickmuskip
\arraycolsep = 0.3\arraycolsep

\newtheorem{theorem}{Theorem}[section]
\newtheorem{lemma}[theorem]{Lemma}
\newtheorem{prop}[theorem]{Proposition}
\newtheorem{definition}{Definition}
\newtheorem{remark}{Remark}
\renewcommand{\thetheorem}{\arabic{section}.\arabic{theorem}}
\newcommand{\prf}{\noindent{\bf Proof.}\ }
\def\prfe{\hspace*{\fill} $\Box$

\smallskip \noindent}

\def\be{\begin{equation}}
\def\ee{\end{equation}}
\def\bea{\begin{eqnarray}}
\def\eea{\end{eqnarray}}
\def\beas{\begin{eqnarray*}}
\def\eeas{\end{eqnarray*}}

\newcommand{\R}{\mathbb R}
\newcommand{\N}{\mathbb N}
\newcommand{\T}{\mathbb T}
\newcommand{\K}{\mathbb S}
\newcommand{\leftexp}[2]{{\vphantom{#2}}^{#1}{#2}}

\def\g{\partial}
\def\E{\mathcal{ E}}
\def\D{\mathcal{D}}
\def\F{\mathcal{F}}
\def\P{\mathcal{P}}
\def\t{\bar{\partial}} 
\def\div{\mbox{div}}
\def\curl{\mbox{curl}}
\def\open#1{\setbox0=\hbox{$#1$}
\baselineskip = 0pt
\vbox{\hbox{\hspace*{0.4 \wd0}\tiny $\circ$}\hbox{$#1$}}
\baselineskip = 11pt\!}
\def\fn{\open{f}}
\def\n{\bar{\nu}}
\def\A{\leftexp{\kappa}{\!\!\!\!\!\!A}}
\def\a{\leftexp{\kappa}{\!\!\!a}}
\def\kA{\leftexp{\kappa}{\!\!\!\tilde{A}}}
\def\w{\leftexp{\kappa}{\!\!\!w}}
\def\kw{\leftexp{\kappa}{\!\!\!\tilde{w}}}

\title{Well-posedness for the classical Stefan problem and the zero surface tension limit}
\author{Mahir Had\v{z}i\'{c} and Steve Shkoller}
\address{Department of Mathematics,
King's College London,
London, United Kingdom}
\address{Department of Mathematics,
University of California, Davis, CA, USA}
\maketitle

\begin{abstract} We develop a framework for a {\em unified} treatment of well-posedness for the Stefan problem
with or without surface tension.   In the absence of surface tension,  we establish well-posedness in Sobolev spaces
for the classical Stefan problem.  We introduce a new velocity variable which extends the velocity of the moving free-boundary into
the interior domain.  The equation satisfied by this velocity is used for the analysis in place of the heat equation satisfied by the temperature.
Solutions to the classical Stefan problem are then  constructed as the limit of  solutions to a carefully
chosen sequence of approximations to the velocity equation,  in which the 
moving free-boundary is regularized and the boundary condition is modified in a such a way as to preserve the basic
nonlinear structure of the original problem.   With our methodology,  we simultaneously find the required stability condition
for well-posedness  and obtain  new estimates for the regularity of the moving free-boundary.    Finally, we prove that solutions of the Stefan problem 
with positive surface tension $\sigma$ converge to solutions of the classical Stefan problem as $\sigma \to 0$.
\end{abstract}

\section{Introduction}\label{S:INTRODUCTION}
\subsection{The problem formulation}\label{se:form}
We consider the local well-posedness and boundary regularity of solutions to the
classical one-phase Stefan problem,  describing the evolving phase boundary of 
a freezing liquid.  We also establish the limit of zero surface tension.

The temperature $p(t,x)$  of a liquid inside of a time-dependent domain $\Omega(t)$
and an {\it a priori unknown}  moving boundary $\Gamma(t)$ satisfies the
following system of equations:
\begin{subequations}
\label{eq:stefan}
\begin{alignat}{2}
p_t-\Delta p&=0&&\ \text{ in } \ \Omega(t)\,,\label{eq:heat}\\
\g_np&=V_{\Gamma(t)}&& \ \text{ on } \ \Gamma(t)\,,\label{eq:neumann}\\
p&=\sigma\kappa_{\Gamma}&& \ \text{ on } \ \Gamma(t)\,,\label{eq:dirichlet}\\
p(0,\cdot)=p_0\,, & \ \Gamma(0)=\Gamma_0&&\,.\label{eq:initial}
\end{alignat}
\end{subequations}
The domain
 $\Omega(t)$ is an evolving open subset of $\R^d$ with $d\ge 2$.  The set $\Gamma(t)$ denotes
the moving boundary (which may be a connected subset of $\partial \Omega(t)$ if
a part of the boundary of $\Omega(t)$ is fixed).  See Figure \ref{fig1}.
\begin{figure}
	[h] \centering 
		\begin{tikzpicture}[scale =1.2]
			\draw [black, thick,fill=gray!20!white] plot [smooth cycle, tension=1] coordinates {(0,0) (1,1) (2.8,0.9) (2.5,0) (1.5,-1)};
			\node [black, above] at (1.15,.3) {$\Omega$};
			\draw [->, thick] (.1,-.1)--(-1,-0.7);
			\node at (0.035,-0.1) {\textbullet};
			\node [right] at (.2,0) {$x'$};
			\node [below] at (-1.2,-0.7) {$N(x')$};
			\node [above] at (1,1.1) {$\Gamma $};
			\draw [black, thick, fill=gray!30!white, xshift=5cm]  plot [smooth cycle, tension=0.8] coordinates {(0,0) (1,-1) (2,-1) (3,1) (1,1)};
			\draw [->, thick] (3,-0.5) arc (180:360:.8cm and 0.5cm);
			\node [below] at (4,-1.) {$\Psi(t, \cdot )$};
			\node [xshift=5cm] at (2.85,-1) {\textbullet};
			\node [xshift=5cm, right] at (3,-1) {$\Psi(t,x')$};
			\node [xshift=5cm, below] at (4,-1.6) {$n(\Psi(t,x'))$};
			\node [above, xshift=5cm] at (1.7,.2) {$\Omega(t)$};
			\node [above, xshift=5cm] at (1.5,1) {$\Gamma(t) $};
			\draw [->, thick, xshift=5cm] (2,-1)--(3,-1.6);
		\end{tikzpicture}
		\caption{{\footnotesize The one-phase Stefan problem.  Displayed on the left side of the figure is the reference domain $\Omega$ and reference
boundary $\Gamma$.  The time-dependent domain $\Omega(t)$ and the moving free-boundary $\Gamma(t)$ is shown on the right side
of the figure.}}\label{fig1}

\end{figure}
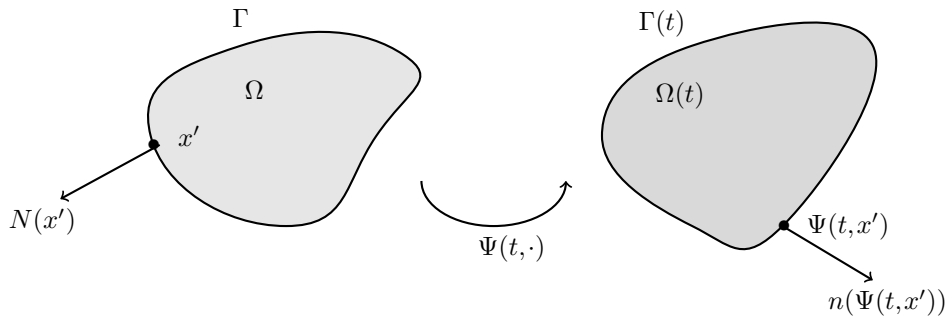

Equation (\ref{eq:heat})  expresses the fact that heat 
heat diffuses in the bulk $\Omega(t)$,  while the boundary condition  (\ref{eq:neumann})
states that the heat flux across the boundary governs the boundary evolution; that is, 
 $\g_np=\nabla p\cdot n$ is the normal derivative of $p$ on $\Gamma(t)$ where
 $n$ stands for the outward pointing unit normal, and 
$V_{\Gamma(t)}$ denotes the speed or the normal velocity of the hypersurface $\Gamma(t)$.
In the case  that $\sigma=0$,~(\ref{eq:dirichlet}) is termed the {\em classical Stefan condition}
and problem~(\ref{eq:stefan}) is called the {\em classical Stefan problem}.  In this case,
freezing of the liquid occurs at a constant temperature $p=0$.   On the other hand, 
if $\sigma>0$ in~(\ref{eq:dirichlet}) then the boundary condition is called 
the {\em Gibbs-Thomson correction} to the classical Stefan condition,
and the system (\ref{eq:stefan})
is then termed the {\em Stefan problem with surface tension}, whereby $\sigma>0$ is a given coefficient 
of surface tension
and $\kappa_{\Gamma(t)}$ stands for the mean curvature of the moving boundary $\Gamma(t)$.
Finally, we equip the problem with suitable initial conditions~(\ref{eq:initial}):
$p_0:\Omega(0)\to\R$ and $\Gamma_0$ are the prescribed initial temperature and 
boundary, respectively. 

Problem~(\ref{eq:stefan}) is an example of free-boundary partial differential equation which requires the initial
data to satisfy a {\it stability condition} in order to ensure well-posedness in Sobolev spaces; specifically, we shall require that
\[
-\g_np_0>0\quad\text{on}\,\,\,\,\Gamma(0)\,,
\]
which,  by analogy to fluid dynamics,  we shall refer to  as
{\em Taylor sign condition} or the Rayleigh-Taylor sign condition.
Below, we will  explain how this Taylor sign condition naturally appears from our analysis.
\subsection{The reference domain  $\Omega$ and the initial domain $\Omega_0$}  We will begin the analysis
with motion in $ \mathbb{R}^2$, and then describe the minor modifications needed to study motion
in $ \mathbb{R}^3  $.
To simplify our presentation, we will parameterize our initial free-boundary $ \Gamma _0$ as a graph over the one-dimensional torus $\T^1$ which we
identify with $[0,2 \pi]$; we define
\be\label{E:INITIALGAMMA}
\Gamma_0 = \{{\bf x}\in\T^1\times\R, \ {\bf x} = (x', h_0(x'))\}, \ \ h_0\in H^4(\T^1).
\ee
Without loss of generality we  shall further assume that $\Gamma_0$ is a small perturbation of the
manifold $\T^1\times\{x^2=0\}$ in the sense that
\be\label{E:INITIAL}
\|h_0\|_{H^4( \mathbb{T}  ^1)} \le \epsilon_0 \ll 1\,,
\ee
for some sufficiently small $\epsilon_0.$ In Appendix~\ref{A:SMOOTHREFERENCE},  we shall explain how to remove the assumption~\eqref{E:INITIAL}.
The only reason for making this smallness assumption is that~\eqref{E:INITIAL} and~\eqref{E:INITIALGAMMA} allow us to 
to use {\it one} global Cartesian coordinate system (rather than a collection of local
coordinate charts). This is ideal for 
describing  new identities that provide very natural estimates for the second-fundamental
form of the evolving free-boundary $\Gamma(t)$.    
All of our results apply to general domains; however,  in a general setting,  we  must employ a finite covering of $\Omega$ by local coordinate charts,
 together with
a partition-of-unity subordinate to that cover. In particular, the Stefan problem {\it localizes} to each chart and effectively
reduces to the analysis on 
\[
\Omega=\T^1\times (0,1).
\]
Again, we emphasize that
the assumption~\eqref{E:INITIAL} is not essential to our proof, and in Appendix~\ref{A:SMOOTHREFERENCE},  we explain
how to treat general $H^4$ initial geometries. 

We define the {\em initial domain}
\[
\Omega_0 = \{ (x_1,x_2) \in\mathbb \T^1\times\R\,\big| \  \, h_0(x^1)<x^2<1 \}\,,
\]
while the {\em reference domain} $\Omega =\T^1\times (0,1) $.
The set 
\[
\Gamma=\T^1\times\{x^2=0\}
\] 
is the {\it reference boundary} on which our parameterization $(x', h(t, x'))$ will be defined.  
The {\it top  boundary} $\partial\Omega_{\text{top}}=\T^1\times\{x^2=1\}$ 
is fixed in time, and
\be\label{eq:fixedbdry}
\g_np=0\quad\mbox{on}\,\,\,\,\partial\Omega_{\text{top}}.
\ee 
\subsection{Notation}
%
%
For any $s\geq0$ and given functions $f:\Omega\to\R$, $\varphi:\Gamma\to\R$ we set
\[
\|f\|_s:=\|f\|_{H^s(\Omega)};\quad
|\varphi|_s:=\|\varphi\|_{H^s(\Gamma)}.
\]
In particular, when $s$ is not an integer, the corresponding fractional Sobolev space is defined by
interpolation in a standard way.
If $f\colon[0,T]\times\Omega\to\R$, $\varphi\colon[0,T]\times\Gamma\to\R$ are given time-dependent functions,
then for any $1\leq p\leq\infty$ we set
\[
\|f\|_{L^p_tH^s_x}:=\|f\|_{L^p([0,T];H^s(\Omega))}\|; \ 
|\varphi|_{L^p_tH^s_x}:=\|\varphi\|_{L^p([0,T];H^s(\Gamma))}.
\]
If $i=1,2$ then 
$f,_i:=\g_{x^i}f$ is the partial derivative of $f$ with respect to $x^i$ coordinate. Similarly,
$f,_{ij}:=\g_{x^i}\g_{x^j}f$ and so on. When differentiating
with respect to the time variable $t$, we set
$f_t=f,_t=\g_tf$. For horizontal derivatives, we write
\[
\t f:=f,_1,\quad\t^kf:=\t_{x^1}^kf.
\]
We use $C$ to denote a universal constant that may vary from line to line.
In numerous estimates the sign $\lesssim$ is used; by definition,  $X\lesssim Y$ if and only if there exists a universal constant $C$ such that $X\leq CY$. 
We use $P$ to denote a generic real polynomial with positive coefficients that can similarly vary from line to line. 
We always sum over  repeated indices.
\subsection{Fixing the domain}   
In order to obtain a priori estimates, and to facilitate the  construction of solutions, we transform the Stefan problem to an
equivalent problem on a fixed domain.  
To this end, we  shall view $\Gamma(t)$ as a graph over $\mathbb T^1$ given by the {\em height function} $h(t,\cdot):\Gamma\to\R$  
$$\Gamma(t) : = \Psi(t,\Gamma).$$
%
In other words, the moving surface $\Gamma(t)$ is parameterized as the graph of a signed height function $h(t,x)$, so that
$\Gamma(t)=\{{\bf x}\in\T^1\times\R|\,\,\,\,{\bf x}=(x',h(t,x'))\}$.
With this parameterization, 
the outward unit normal $n(t,x')$ to $\Gamma(t)$  at the point $(x',h(t,x'))$ is given by
\be\label{E:UNITNORMAL}
n(t,x')=\frac{(\t h,-1)}{\sqrt{1+|\t h|^2}}.
\ee
Assuming  that $h(t, \cdot )$ is sufficiently regular and remains a graph, we can define a diffeomorphism $\Psi(t,\cdot):\Omega\to\Omega(t)$
as an harmonic extension of the  boundary diffeomorphism $(x',h)$, by solving the elliptic equation 
\begin{subequations}
\label{eq:gauge}
\begin{alignat}{2}
\Delta \Psi(t,\cdot)&=0&& \ \text{ in } \ \Omega,\label{eq:diff}\\
\Psi(t,x',0)&=(x',h(t,x'))&& \ \,\, x'\in\Gamma,\label{eq:damo1}\\
\Psi(t,\cdot)&=\text{Id}&& \ \text{ on } \ \partial\Omega_{\text{top}},
\end{alignat}
\end{subequations} 
where $\text{Id}$ denotes the identity map.
The mapping $\Psi(t, \cdot )$ is indeed a diffeomorphism;
note that the map $\Phi: = \Psi - \text{Id}$ solves the problem
\begin{subequations}
\label{eq:gauge2}
\begin{alignat}{2}
\Delta \Phi(t,\cdot)&=0&& \ \text{ in } \ \Omega,\label{eq:diff}\\
\Phi(t,x',0)&=(0,h(t,x'))&& \ \,\, x'\in\Gamma,\label{eq:damo1}\\
\Phi(t,\cdot)&={\bf 0}&& \ \text{ on } \ \partial\Omega_{\text{top}}\,,
\end{alignat}
\end{subequations} 
so that
by elliptic estimates,  we may conclude that $\|\Psi(t, \cdot )-\text{Id}\|_{H^{4.5}(\Omega)} \lesssim |h(t,\cdot )|_{H^4(\Gamma)} \lesssim \epsilon_0$
using the
assumption~\eqref{E:INITIAL}, and the continuity
of the map $t \mapsto h(t, \cdot )$ in $H^4(\Gamma)$ which will be proved below.
 By the inverse function theorem we infer that $\Psi(t, \cdot )$ is a diffeomorphism.

As a consequence of~\eqref{eq:gauge},
\begin{equation}\label{DM}
\|\Psi\|_{H^s(\Omega)}\leq C\|\Psi\|_{H^{s-0.5}(\Gamma)}\,,
\end{equation} 
and thus $\Psi(t, \cdot )$ gains a half-derivative of regularity in $\Omega$ with respect to the height function $h(t, \cdot )$ on $\Gamma$.

\subsection{Reference unit normal, unit tangent, line element, and the Jacobian}
We let 
\[
N=(0,-1), \ \ T=(1,0)
\]
denote the outward pointing  unit normal and  tangent vectors to  $\Gamma=\T^1\times\{x_2=0\}$, respectively.
The time-dependent unit normal $n(t, \cdot )$ and  tangent $\tau(t,\cdot )$ vectors  to $\Gamma(t)$ are given by
\be\label{E:UNITNORMAL}
n = J g^{-1} A^{\text T} N, \ \ \tau = J g^{-1} A^{\text T} T,
\ee
where 
\begin{equation} 
J(t,{\bf x}) : = \det \nabla\Psi(t,{\bf x}), \ \ {\bf x}\in\Omega, \label{E:J}
\end{equation} 
denotes the Jacobian determinant of $\nabla \Psi$,
and
\begin{equation} 
 g(t,x) : = \sqrt{1+(\t h(t,x))^2}, \ \ x\in\mathbb T^1 \label{E:G} \,,
 \end{equation} 
where $g^2\,dx$ is the line element associated with the metric induced on $\Gamma.$
Together with~\eqref{E:UNITNORMAL} we obtain the relationship
\be\label{eq:bullet}
A^2_{\bullet}:=J^{-1}(\t h, -1), \ \ |A^2_\bullet| = J^{-1}g.
\ee
The vector $A^2_{\bullet} (t, \cdot )$ will play an important role in the derivation of energy identities as it is parallel to $n(t,\cdot )$.
\subsubsection{The change of variables} On the reference domain $\Omega$, 
we set
\be\label{eq:ALE0}
q:=p\circ\Psi,\quad\,A:=[\nabla\Psi]^{-1},
\quad\,w:=\Psi_t,\quad\,v:=-\nabla p\circ\Psi.
\ee
Note that, by the chain-rule,  the relation $v=-\nabla p\circ\Psi$ can be written as
\be\label{eq:ALEv}
v^i+A^k_iq,_k=0\,\quad\mbox{in}\quad\Omega.
\ee
We also express $p_t\circ\Psi$ in terms of $q,v,w$.
Again, by the chain rule,
$p_t=q_t\circ\Psi^{-1}+\nabla q\circ\Psi^{-1}\cdot\Psi^{-1}_t$.
Since $\Psi^{-1}_t=-A\circ\Psi^{-1}\Psi_t\circ\Psi^{-1}$ and $w=\Psi_t$, 
using~(\ref{eq:ALEv}) we obtain that
$$
p_t\circ\Psi=q_t-q,_kA^k_rw^r=q_t+v\cdot w.
$$
The transformed Laplacian $\Delta_{\Psi} q := \Delta p\circ\Psi$ is defined as
\begin{equation}\label{ms1}
\Delta_{\Psi} q=A^j_i(A^k_iq,_k),_j \,,
\end{equation} 
and we define
\begin{align}\label{E:PSIGRADIENT}
\nabla_\Psi q: = A^k_iq,_k = \nabla p \circ \Psi.
\end{align}
\begin{remark}[Differentiation rules]
When differentiating the matrix $A=[\nabla\Psi]^{-1}$,
for a given $i,k\in\{1,2\}$, 
\[
\g_tA^k_i=-A^k_rw^r,_sA^s_i;\qquad\t A^k_i=-A^k_r\t\Psi^r,_sA^s_i.
\]
In particular,  a simple application of the above identities, together with the product rule, show that
for any given $a,b\in\N$:
\be\label{eq:comm}
\t^a\g_t^bA^k_i=-A^k_r\t^a\g_t^b\Psi^r,_sA^s_i+\{\t^a\g_t^b,\,A^k_i\};
\quad
\{\t^m\g_t^n,\,A^k_i\}:=\sum_{l+l'\geq1}a_{l,l'}\t^l\g_t^{l'}(A^k_rA^s_i)\t^{m-l}\g_t^{n-l'}\Psi^r,_s,
\ee
where the term $\{\cdot,\,\cdot\}$  is the {\em commutator} error.
Here the constants $a_{l,l'}$ are some universal constants, depending only on $m$, $n$, $l$ and $l'$
(where $0\leq l\leq m$, $0\leq l'\leq n$).
\end{remark}
\subsubsection{Classical Stefan problem in the new variables}
Using the family of diffeomorphisms $\Psi(t, \cdot)$, the classical Stefan problem (i.e. problem~(\ref{eq:stefan}) with $\sigma=0$) on the fixed reference domain $\Omega$
is given by
\begin{subequations}
\label{eq:ALE}
\begin{alignat}{2}
q_t- \Delta _\Psi q&=-v\cdot w&& \ \text{ in } \ \Omega \times (0,T] \,,\label{eq:ALEheat}\\
v^i+A^k_iq,_k&=0&&\ \text{ in } \ \Omega  \times (0,T] \,,\\
q&=0&& \ \text{ on } \ \Gamma \times [0,T] \,,\label{eq:ALEdirichlet}\\
\Delta \Psi&=0&& \ \text{ on } \ \Omega  \times [0,T] \,,\\
\Psi&= \text{Id}+ h\, N&& \ \text{ on } \ \Gamma  \times [0,T] \,,\\
\Psi&= \text{Id}&& \ \text{ on } \ \partial\Omega_{top}  \times [0,T] \,,\\
\Psi_t\cdot n(t)&=-v\cdot n(t)&& \ \text{ on } \ \Gamma  \times (0,T] \,,\label{eq:ALEneumann}\\
v\cdot N&=0&& \ \text{ on } \ \g\Omega_{\text{\text{top}}} \times [0,T] \,,\label{eq:ALEtop}\\
\Psi(0,\cdot)=\Psi_0& \ \   q(0,\cdot)=q_0=p_0\circ\Psi_0\,,&&\,\label{eq:ALEinitial}
\end{alignat}
\end{subequations}
where $ \Delta _ \Psi q$ is defined in (\ref{ms1}) and $N=(0,1)$ is the outward-pointing unit normal to $\partial\Omega_{\text{top}}$.
Problem~(\ref{eq:ALE}) is a reformulation of 
problem~(\ref{eq:stefan}). 
Condition~\eqref{eq:ALEneumann} is equivalent to the  evolution equation for the height function $h(t, \cdot )$ which is given by
\begin{align}\label{E:HEVOLUTION}
h_t(t,x)  = -g(t,x) \nabla_{\Psi}q(t,x) \cdot n(t,x), \ \ x\in \mathbb T^1, 
\end{align}
where the quantity $g$ is defined in~\eqref{E:G} and $ \nabla _\Psi q$ is defined in \eqref{E:PSIGRADIENT}.
The time-evolution of the map $\Psi(t,\cdot )$ is, in turn, coupled to the evolution of $q(t,\cdot )$ via~(\ref{eq:ALEheat}).
\subsubsection{The higher-order energy function $ \mathcal{E} (t)$} We define the higher-order energy function as
\be
\label{eq:ALEenergy}
\begin{array}{l}
\displaystyle
\E(t)=\E(q,h)(t):=\sum_{a+2b\leq4}\|\t^a\partial_t^bv\|_{L^2_tL^2_x}^2
+\sum_{a+2b\leq3}\|\t^a\partial_t^bv\|_{C^0_tL^2_x}^2\\
\displaystyle
+\sum_{b=0}^2\|\g_t^bq\|_{C^0_tH^{4-2b}_x}^2+\sum_{b=0}^2\|\g_t^bq\|_{L^{2}_tH^{5-2b}_x}^2
+\sum_{b=0}^2|\g_t^bh|_{C^0_tH^{4-2b}_x}^2+\sum_{b=0}^1|\g_t^bh_t|_{L^2_tH^{3-2b}_x}^2\,,
\end{array}
\ee
where the time integrals in the $L^2$-norms  above are over the time-interval $[0,t]$.   We will show that $ \mathcal{E} (t)$ remains bounded on $[0,T]$.

\subsubsection{The Taylor sign condition}
In order to obtain a locally well-posed problem  for arbitrarily large initial data, we must impose the Taylor sign condition on the initial data as follows: 
\be
\label{eq:taylor00}
-\g_np_0>0\quad\text{on}\,\,\,\,\Gamma(0) \,.
\ee
Expressed in terms of $q(0, \cdot )$, \eqref{eq:taylor00} is written as
 \be\label{eq:taylor0}
\left.q_0,_2\right|_{t=0} > 0 \text{ on } \Gamma \,.
\ee
The condition~(\ref{eq:taylor0}) ensures that
\be\label{eq:taylor}
\inf_{x'\in\Gamma}q,_2(t,x',0)>0, \ \  t\in [0,T]
\ee
if $T>0$ is taken sufficiently small.
As mentioned in Section~\ref{se:form}, we shall refer to~(\ref{eq:taylor0}) as the {\em Taylor sign condition} in analogy 
to the terminology used in the well-posedness theory in fluid mechanics~\cite{Tay, Ray}. The Taylor sign condition
will provide positivity  of the natural energy functional. 

\begin{remark}
Note that $q,_2=gJ^{-1} v\cdot n$ on $\Gamma$,  with $v$ defined by~(\ref{eq:ALE0}). By~\eqref{E:HEVOLUTION},  we conclude that
 $h_t = J q,_2$ at time $t=0$. Since the Jacobian remains positive on a short interval of time
the Taylor sign condition~\eqref{eq:taylor00}  shows that 
$h_t(0,x)<0$ for all $x\in\mathbb T^1.$ Thus, the domain $\Omega(t)$ expands on a short interval of time.
\end{remark}

\subsubsection{Compatibility conditions}\label{S:CC}
To ensure that the solution is continuously differentiable with respect to $t$, at $t=0$, we must impose  
compatibility conditions on the initial data. In particular, restricting~(\ref{eq:ALEheat}) to $\Gamma$ and 
evaluating at time $t=0$, for $H^4$ initial data,  we find that
\begin{subequations}
\label{eq:comp}
\begin{alignat}{2}
q_0&=0&& \ \text{ on } \ \Gamma\,,\label{E:COMP1}\\
\Delta_{\Psi_0}q_0+J_0^{-2}g_0^2(q_0),_2^2&=0&&\ \text{ on } \ \Gamma. \label{E:COMP2}
\end{alignat}
\end{subequations}

In the derivation of~\eqref{E:COMP2} we have crucially used~\eqref{E:UNITNORMAL} and the identity 
\begin{align*}
v(0,\cdot)\cdot w(0,\cdot)\big|_\Gamma & = - A^k_iq_{0,k}\g_t\Psi^i_t(0,\cdot) = -q_{0,2}A^k_i N,_k \Psi^i_t (0,\cdot)\\
& = J_0^{-1}g_0^{-2}q_{0,2} \Psi_t(0,\cdot)\cdot n_0 = J_0^{-1}g_0^{-2}q_{0,2} v(0,\cdot)\cdot n_0 \\
& = J_0^{-2}g_0^2(q_0),_2^2.
\end{align*}

Here $J_0, g_0$ are the initial values of $J,g$ defined in~\eqref{E:J},~\eqref{E:G}  respectively.
Conditions~\eqref{eq:comp} are satisfied for a large class of functions. 
Consider simply a function $q_0$ independent of $x_1$ in the slab $ T_\epsilon = \T^1\times[0,\epsilon]$ for some $\epsilon>0$, and of the  form 
\be\label{E:Q0}
q_0(x_1,x_2) = -\alpha^2\frac{x_2^2}{2} + \alpha x_2
\ee 
in $T_\epsilon$ for some $\alpha>0.$ 
Condition~\eqref{E:COMP1} is obviously satisfied, while~\eqref{E:COMP2} reduces to the requirement that
\[
0 = -\Delta_{\Psi_0} q_0 + J^{-2}g^2(q_0),_2^2 =  -|A^2_\bullet|^2 (q_0),_{22} + J^{-2}g^2(q_0),_2^2 =  J^{-2}g^2\left((q_0),_{22} + (q_0),_2^2 \right),
\]
where we have used~\eqref{eq:bullet}. It is easily checked that  for such $q_0$, 
\[
(q_0),_{22} + (q_0),_2^2 = 0 \ \ \text{ on } \ \Gamma \,,
\]
 and therefore the condition~\eqref{E:COMP2} is satisfied.
Note that the assumption $\alpha>0$ ensures the validity of the Taylor sign condition $(q_0),_2>0.$

Since we imposed the homogeneous Neumann condition~\eqref{eq:ALEtop} on the top boundary $\g\Omega_{\text{top}}$, we impose 
the compatibility condition
\be\label{eq:cctop}
(q_0),_2|_{t=0}=0 \ \ \text{on}\quad\g\Omega_{\text{top}}.
\ee
By employing a partition-of-unity of $\Omega$, we can now easily construct a $q_0\in H^4(\Omega)$ such that the compatibility conditions~\eqref{eq:comp} and~\eqref{eq:cctop} are simultaneously satisfied.
\begin{remark}
The quadratic function $q_0$ defined in~\eqref{E:Q0} satisfies the compatibility conditions.  This is one of many possible constructions of initial data satisfying the 
corresponding regularity and compatibility conditions.
\end{remark}

\subsection{Local-in-time well-posedness for the classical Stefan problem}
We define
\be\label{eq:solspace}
\mathcal{S}(t):=\{(q,h):\,\,\,\,\E(q,h)(t)<\infty\}.
\ee
Our first result is a well-posedness statement for the classical Stefan problem.
\begin{theorem}[Well-posedness of the classical Stefan problem]\label{th:main}
Given  initial conditions $(q_0,h_0)\in\mathcal{S}(0)$  with $q_0$ satisfying the Taylor sign condition~(\ref{eq:taylor0}) and the compatibility conditions~\eqref{eq:comp}--\eqref{eq:cctop},
the
problem~(\ref{eq:ALE}) is locally-in-time well-posed, i.e.
there is a $T>0$ such that 
and a unique solution $(q,h)$  on  
the time interval $[0,T]$ with  initial data $(q_0,h_0)$,  such that
$$
\sup_{t \in [0,T]} 
\E(q,h)\leq 2\E(q_0,h_0).
$$
\end{theorem}
\begin{remark}\label{re:mainth}   
The definition of our higher-order energy function $\E$ restricted to time $t=0$ 
requires an explanation of time-derivates of $q$ and $h$ at $t=0$.   Specifically,
the values $q_t|_{t=0}$, 
$q_{tt}|_{t=0}$, $h_t|_{t=0}$ and $h_{tt}|_{t=0}$ are defined via
space-derivatives using equations~(\ref{eq:ALEheat}) and~(\ref{eq:ALEneumann}).
\end{remark}

\subsection{The vanishing surface tension limit}
Our second main result establishes the vanishing surface tension limit. 
Denoting by $\mathcal H$ the mean curvature of the free-boundary, in the $\Psi$-parametrization, the 
 boundary condition~(\ref{eq:ALEdirichlet}) is replaced with 
\be\label{eq:ALEdirsigma}
q=\mathcal H = \sigma\frac{\t^2\Psi\cdot n}{|\t^2\Psi|^2}=-\sigma\frac{\t^2h}{(1+|\t h|^2)^{3/2}}\,;
\ee
then,  the problem~(\ref{eq:ALE}), with the boundary condition (\ref{eq:ALEdirsigma}) replacing (\ref{eq:ALEdirichlet}), 
is the {\it Stefan problem with surface tension}  formulated in harmonic coordinates.
The high-order energy function adapted 
for the presence of surface tension is given by
\be\label{eq:energysigma}
\E^{\sigma}=\E^{\sigma}(q,h)=\E(q,h)
+\sigma\sum_{b=0}^2|\partial_t^bh|_{C^0_tH^{5-2b}_x}^2
+\sigma\sum_{b=0}^1|\partial_t^bh_{t}|_{L^2_tH^{4-2b}_x}^2
+\sigma^2\sum_{b=0}^2|\partial_t^b h|_{L^\infty_tH^{4-2b}}^2, \ \ \sigma>0.
\ee

\subsubsection{Compatibility conditions for the Stefan problem with surface tension}

To ensure the spatial continuity of the temperature function $q$ and its first time derivative $q_t$ at time $t=0$, we must impose two sets 
of compatibility conditions. 
The first condition is
\be\label{E:CCSIGMA1}
q_0 = \sigma\mathcal H_0 = - \sigma g_0^{-3}\t^2h_0 \ \ \text{ on } \ \Gamma,
\ee
where $\mathcal H_0$ denotes the mean curvature of the initial free surface $\Gamma_0$, and $g_0 = \sqrt{1+(\t h_0)^2}.$
To obtain the second compatibility condition, we note that 
$
q_t\big|_{\Gamma} = - \sigma\partial_t\mathcal H\big|_{t=0} . 
$
From the boundary condition~\eqref{eq:ALEneumann} we can evaluate $h_t$ at time $t=0$ as
\be\label{E:CCSIGMA2}
h_t\big|_{t=0} = - g_0 \nabla_{\Psi_0}q_0\cdot n_0,
\ee
where the subscript $0$ refers to the initial values of the quantities $g,\Psi,q,$ and $n$ defined above.
Therefore, 
\begin{align}
\partial_t\mathcal H\big|_{t=0} 
& = - g_0^{-3}\t^2h_{t}\big|_{t=0} + 3g_0^{-5}\t^2h_0\t h_{t}\big|_{t=0}\t h_0 \notag\\ 
& =  g_0^{-3} \t^2(g_0 \nabla_{\Psi_0}q_0\cdot n_0) + 3q_0g_0^{-2}\t(g_0 \nabla_{\Psi_0}q_0\cdot n_0) \t h_0, \label{E:CCSIGMA3}
\end{align}
where we have used~\eqref{E:CCSIGMA1} and~\eqref{E:CCSIGMA2} in the last line.
After restricting~\eqref{eq:ALEheat} to $\Gamma$ at time $t=0$ and using~\eqref{E:CCSIGMA3},  we find that
\[
- \sigma\left(g_0^{-3} \t^2(g_0 \nabla_{\Psi_0}q_0\cdot n_0) + 3q_0g_0^{-2}\t(g_0 \nabla_{\Psi_0}q_0\cdot n_0) \t h_0\right)-\Delta_{\Psi_0}q_0 
 = -g_0 \left(\nabla_{\Psi_0}q_0\cdot n_0 \right) (A_0)^k_2q_{0,k}.
\]
In particular, the right-hand side can be separated into the $\sigma$-depenendent and $\sigma$-independent contributions, so that 
\[
-g_0 \left(\nabla_{\Psi_0}q_0\cdot n_0 \right) A^k_2q_{0,k} = J_0^{-2} g_0^2 (q_0),_2^2 + \sigma\t\mathcal H_0 \left(g_0 \left(\nabla_{\Psi_0}q_0\cdot n_0 \right) (A_0)^1_2 
+ g_0 \left((A_0)^1_\bullet \cdot n_0 \right) A^2_2(q_0),_2\right) \,.
\]
Combining the two previous identities, we find the second compatibility condition to be
\be\label{E:CCSIGMA4}
\Delta_{\Psi_0} q_0 + J_0^{-2} g_0^2 (q_0),_2^2 = \sigma \mathcal C(q_0,h_0),
\ee
where
\begin{align}
\mathcal C(q_0,h_0) : = & -\left[g_0^{-3} \t^2(g_0 \nabla_{\Psi_0}q_0\cdot n_0) + 3q_0g_0^{-2}\t(g_0 \nabla_{\Psi_0}q_0\cdot n_0) \t h_0\right] \notag \\
& -\t\mathcal H_0 \left[g_0 \left(\nabla_{\Psi_0}q_0\cdot n_0 \right) (A_0)^1_2 
+ g_0 \left((A_0)^1_\bullet \cdot n_0 \right) (A_0)^2_2(q_0),_2\right] \label{E:CCSIGMA5}
\end{align}
\subsubsection{Initial data satisfying compatibility conditions}
When $\Psi = \text{Id}$ (and therefore $h_0(x)=0,$ $g_0=J_0=1$) the compatibility conditions~\eqref{E:CCSIGMA1} and~\eqref{E:CCSIGMA4}--\eqref{E:CCID2} simplify significantly and take the form
\begin{align}
q_0 = 0 \ \text{ and } \   (q_0),_{22} + (q_0),_2^2 & =  \sigma (q_0),_{211} \ \ \text{ on } \ \Gamma \label{E:CCID1}\\
 (q_0),_{222}, \ (q_0),_{211} & \in C^0(\Gamma).\label{E:CCID2}
\end{align}
It is easy to check that the function $q_0$ constructed in Section~\ref{S:CC} satisfies~\eqref{E:CCID1}--\eqref{E:CCID2}.
For general $h$ satisfying $|h|_{4}\ll 1$ we can construct the initial temperature $q_0$ 
satisfying~\eqref{E:CCSIGMA1} and~\eqref{E:CCSIGMA4}--\eqref{E:CCID2} by perturbative methods, using for instance the implicit function theorem.

\subsubsection{Well-prepared initial data}

To obtain the vanishing surface tension limit, we need to define a suitable class of  initial data $(q_0^\sigma,h_0^\sigma),$ $\sigma\ge0.$

\begin{definition}[Well-prepared data]\label{def:wellprep}
A family of initial data $(q^{\sigma}_0, h^{\sigma}_0)_{\sigma\geq0}$ such that $\E(q_0^\sigma,h_0^\sigma)<\infty$
is well-prepared if it satisfies {\bf 1)} compatibility
conditions~\eqref{E:CCSIGMA1},~\eqref{E:CCSIGMA4}--\eqref{E:CCID2} associated to the Stefan problem with surface tension,
{\bf 2)} the Taylor sign condition~\eqref{eq:taylor0}, and  
{\bf 3)} $\E(q^{\sigma}_0-q_0,h^{\sigma}_0-h_0)\to0$ as $\sigma\to0$.
\end{definition}

We now demonstrate that the class of well-prepared initial data is non-empty. Let us assume for simplicity that $\Psi_0 = \text{Id}$ and therefore the initial hypersurface $\Gamma_0$ is flat. 
For $\sigma\ge0$ we have $h^\sigma(x')=0,$ $x'\in\mathbb T^1.$ Let $b:\mathbb T^1\to\mathbb R$ be a given smooth function and $\alpha>0$ a given positive real number.
Consider a function $q_0^\sigma$ independent of $x_1$ in the slab $ T_\epsilon = \T^1\times[0,\epsilon]$ for some $0<\epsilon<1$ and of the  form 
\[
q_0^\sigma(x_1,x_2) = -\alpha^2\frac{x_2^2}{2} + \alpha x_2 + \sigma b(x_1)x_2^3, \ \ (x_1,x_2)\in T_\epsilon.
\] 
It is straightforward to check that conditions~\eqref{E:CCID1}--\eqref{E:CCID2} are both satisfied with this choice of $q_0^\sigma.$ Moreover, 
The Taylor sign condition holds since $q^\sigma_{0,2}=\alpha>0$ for any $\sigma\ge0$ and the convergence requirement {\bf 3)} in Definition~\ref{def:wellprep} is clearly satisfied.
Outside the slab $T_\epsilon$ we can extend the function $q^\sigma_0$ smoothly so that the Neumann boundary condition $\g_Nq^\sigma_0$ is satisfied on $\partial\Omega_{\text{top}}$.

\subsubsection{The vanishing surface tension limit}

For a given $T>0$ let
\begin{align}
C^1_tC^0_x\cap C^0_tC^2_x:=& \Big\{(q,h):\,\,\,q\in C^1([0,T];C^0(\Omega))\cap C^0([0,T];C^2(\Omega)), \notag \\
& \ \ h\in C^1([0,T];C^0(\Gamma))\cap C^0([0,T];C^2(\Gamma))\Big\} \label{eq:c12}
\end{align}
with the associated norm:
\begin{align}
\|(q,h)\|_{C^1_tC^0_x\cap C^0_tC^2_x} =& \max_{t\in[0,T],x\in\Omega}\left(|q(t,x)|+|\g_tq(t,x)|+|\nabla q(t,x)|+|\nabla^2q(t,x)|\right) \notag \\
&+\max_{t\in[0,T],x'\in\Gamma}\left(|h(t,x')|+|\g_th(t,x')|+|\t h(t,x')|+|\t^2h(t,x')|\right). \label{E:QHNORM}
\end{align}


\begin{theorem}[The limit of zero surface tension]\label{th:main2}
Let $(q_0^{\sigma}, h_0^{\sigma})_{\sigma\geq0}$ be a sequence of well-prepared initial conditions
in the sense of Definition~\ref{def:wellprep} such that
$$
\E(q^{\sigma}_0-q_0,h^{\sigma}_0-h_0)\to0 \text{  as } \sigma\to 0 \,.
$$

Let
$(q^{\sigma}(t, \cdot ),h^{\sigma}(t, \cdot ))_{\sigma\geq0}$  denote the corresponding sequence of solutions to the 
Stefan problem with surface tension, such that $( q^\sigma(0, \cdot ), h^\sigma(0, \cdot )) = (q_0^{\sigma}, h_0^{\sigma})$.
Then, there exists a $\sigma$-independent time $T>0$ and a constant $C$ depending only on $(q_0,h_0)$ such that
\[
\E^{\sigma}(q^{\sigma},h^{\sigma})(T)\leq C\quad\sigma\geq0.
\]
for all $\sigma\geq0$.  

Furthermore,  the sequence $(q^{\sigma},h^{\sigma})$ converges in the ${C^1_tC^0_x\cap C^0_tC^2_x}$-norm
to the unique solution $(q,h)$ of the classical Stefan problem~(\ref{eq:ALE}) with 
$\sigma=0$ and the initial data $(q(0),h(0))=(q_0,h_0)$. 
\end{theorem}




%
\subsection{Prior results and a motivation for the current treatment}
There is a large literature on the classical one-phase Stefan problem. For a comprehensive 
overview, 
we refer the reader to \textsc{Meirmanov}~\cite{Me} and \textsc{Visintin}~\cite{Vi}. The first  {\em weak} solutions were defined by \textsc{Kamenomostskaya}~\cite{Ka},  \textsc{Ladyzhenskaya, Solonnikov}
and \textsc{Uralceva}~\cite{LaSoUr}.  These weak solutions  were analyzed  by \textsc{Friedman, Kinderlehrer}~\cite{Fr82},~\cite{Fr89},~\cite{FrKi},
\textsc{Cafarelli, Evans}~\cite{Ca78},~\cite{CaEv}, wherein the
 regularity of weak solutions was established.   Since the problem satisfies a maximum principle,
it is ideally suited to the so-called {\it viscosity solutions} approach. Existence and regularity of
viscosity solutions was established by \textsc{Athanasopoulos, Caffarelli, and Salsa} in~\cite{AtCaSa1},~\cite{AtCaSa2}.
Existence of viscosity solutions in the one-phase case was proven by 
\textsc{Kim}~\cite{Ki} and in the two-phase case by \textsc{Kim and Po\v zar}~\cite{KiPo}. 
A local-in time regularity theorem was proven in~\cite{sCiK2010} which in particular shows that initially
Lipschitz free-boundaries become $C^1$ over a possibly smaller spatial region.
For  an exhaustive 
overview and introduction to the regularity theory of such solutions we refer the reader to \textsc{Caffarelli and Salsa}~\cite{CaSa}, see also more recent results~\cite{sCiK2010,sCiK2012}. 

Local existence of classical solutions for the classical Stefan problem was shown by 
\textsc{Meirmanov} (see~\cite{Me} and references therein) and \textsc{Hanzawa}~\cite{Ha}. 
In the first approach, the author regularizes the problem by adding artificial viscosity to~(\ref{eq:neumann}) and fixes 
the moving domain by switching to so-called von Mises variables.  The obtained solutions however,
lose derivatives with respect to the assumed regularity on the initial data. Similarly, in~\cite{Ha}
the author uses Nash-Moser iteration to obtain a local-in-time solution, however again with a significant
derivative loss with respect to the initial data. 
A local existence result for the one-phase n-dimensional Stefan problem is proved in~\cite{FrSo}, 
where the required 
regularity class for the temperature function is $W^{2,1}_p$ with $p>n+2$. For the two-phase
Stefan problem a local existence result is presented in~\cite{PrSaSi} in the framework
of $L^p$ maximal regularity, where the corresponding functional spaces of Sobolev-type require
$p>n+3$, where $n$ is the dimension of the ambient space.

In  related work, local  and global existence for the one-phase and two-phase Muskat problems  has been established in 
\cite{CoCoGa,CoCoGa10,CoCoGaSt,ChGrSh2016}.   For the local and global well-posedness of the
one-phase Hele-Shaw problem and optimal decay rates of the solutions, see \cite{ChCoSh2014} and the references therein.

As to the Stefan problem with surface tension (also known as the Stefan problem with Gibbs-Thomson correction), a
global weak existence theory (without uniqueness) is given in~\cite{AlWa,Lu,Roe}.
In \cite{FriRe} the authors consider the
Stefan problem with small surface tension i.e. $\sigma\ll1$ whereby~(\ref{eq:dirichlet})
is replaced by $v=\sigma\kappa$.
Local existence of classical solutions is studied in \cite{Ra}.
In \cite{EsPrSi} the authors
prove a local existence and uniqueness result for classical solutions 
under a smallness assumption on the initial datum
close to flat hypersurfaces. Global existence close to flat hyper-surfaces
is proved in~\cite{HaGu1} and close to stationary spheres for the
two-phase problem in~\cite{Ha1} and later in~\cite{PrSiZa}.

With the Gibbs-Thomson correction, problem~(\ref{eq:stefan}) 
can account for phenomena such as
the phase nucleation, undercooling (superheating) and it is also used in modeling crystal growth~\cite{Vi}. It is
a small-scale model as opposed to the macro-scale classical Stefan problem.
In this sense, there is a fundamental importance in rigorously understanding the link between the two models. 
As explained in~\cite{Vi},~\cite{Vi96}, one can associate a free energy to the Stefan problem with surface tension defined by
\[
F_{\sigma}(\tilde{p},\tilde{\Gamma})=\int_{\Omega}\tilde{p}\,dx+\sigma|\tilde{\Gamma}|,
\] 
where $\tilde{p}$, $\tilde{\Gamma}$ are time-independent. 
Then in the the sense of $\Gamma$-convergence of De Giorgi~ \cite{De}, the free energy
$F_{\sigma}(\tilde{p},\tilde{\Gamma})$ converges to the free energy for the classical Stefan problem, see \cite{Vi}.
This is, however, a completely time-independent consideration and does not address the vanishing $\sigma$-limit of 
time-dependent solutions to the full non-linear problem~(\ref{eq:stefan}).
In the context of the water wave problem, the vanishing surface tension limit in two and three dimensions
has been studied in~\cite{AmMa09, AmMa05}; for the full Euler equations, see \cite{CoHoSh2014}.
\par
Turning our attention to the Stefan problem, we can observe that 
there are two parallel developments in the existence theory for weak solutions briefly mentioned above.
The first one applies to the classical Stefan problem and it is motivated by the validity of maximum principle; suitable notions of weak and viscosity
solutions have been established~\cite{AtCaSa1, AtCaSa2, LaSoUr, FrKi, Ca78}. 
The second development refers to the problem with surface tension, wherein the weak solution
existence results are in BV-type spaces, and rely upon the gradient-flow structure of the problem. From the point of view of the
vanishing surface tension, it is natural to ask whether the two concepts are 
compatible in any rigorous mathematical manner. The answer is inconclusive due to a  lack compactness.
While the control of solutions constructed in~\cite{Lu, AlWa} is strong enough to pass to some limit
as $\sigma\to0$, it is too weak to guarantee a sharp interface in the limit. In other words, it is not clear how to 
preclude the formation of so-called mushy 
regions~\cite{Vi}.
\par
We develop a {\em new} energy method for the Stefan problem with and without 
surface tension and prove the vanishing surface tension limit.
The {\em well-posedness} is established in $H^k$ Sobolev spaces using a combination of energy estimates for
tangential derivatives and elliptic-type estimates for added parabolic-type regularity.
Our framework is motivated by the analysis of the free-surface
incompressible Euler equations of \textsc{Coutand and Shkoller}~\cite{CoSh07,CoSh10}.

Precise statements of our results are given in Theorems~\ref{th:main} and~\ref{th:main2}.
The estimates that we use are nonlinear in nature and
they fundamentally exploit the intricate energy structure of the problem.
In particular, no derivative loss occurs with respect to the regularity of the initial data.
This framework is particularly convenient, as it allows us to rigorously establish the vanishing surface tension 
limit locally-in-time, as formulated in Theorem~\ref{th:main2}. In this way, we link
two fundamental models of phase transitions that are valid on different spatial scales, 
thus answering the open question explained above. 
In forthcoming work,
we shall extend our results to the two-phase Stefan problem,  providing the analog of Theorem~\ref{th:main}~\cite{HaNaSh},
while the question of global-in-time stability of steady states using this functional-analytic framework has been addressed in~\cite{HaSh3, HaSh4}.
\subsection{Methodology and outline of the paper}
There are three main ingredients in our approach to the Stefan problem.   First, we replace the study of the heat equation for temperature, with the equation
for a velocity field $u(t, x)$ which satisfies the equation $u + \nabla p =0$.
Second, we  introduce the so-called
Arbitrary Lagrange-Eulerian  (ALE) variables, in which we introduce a family of diffeomorphisms $\Psi(t, \cdot ) : \Omega \to \Omega(t)$ which
fix the moving domain.
 With respect to this change-of-variables, we define, respectively,  the new velocity  and temperature fields
$v= u \circ \Psi$ and  $q = p \circ \Psi$; in these variables,  the velocity equation becomes $v + \nabla p \circ \Psi =0$.   This equation contains the geometry of the evolving
free-boundary, and by the use of energy estimates for tangential derivatives, 
we are able to naturally estimate the second-fundamental form as
\be\label{eq:trans}
\int_{\Gamma}q,_2|\t^k\Psi\cdot A^2_{\bullet}|^2\,dx'\approx\int_{\Gamma}q,_2|\t^kh|^2\,dx',
\ee
for $k$ some positive integer. 
In the original Eulerian framework~(\ref{eq:stefan}), 
the energy dissipation law is given by
\[
\frac{1}{2}\frac{d}{dt}\int_{\Omega(t)}p(t,x)^2\,dx+\int_{\Omega(t)}|\nabla p(t,x)|^2\,dx=0.
\]
This basic energy law is to weak to control the regularity  of the evolving free-boundary.
Observe that our higher-order control of the free-boundary given by \eqref{eq:trans}  naturally produces the stability condition; in particular,  the Taylor sign condition~(\ref{eq:taylor0}) arises as coefficient to the second-fundamental form, and it sign determines either the control or growth of the curvature and its
derivatives via a Gronwall-type inequality.

A further subtlety consists in the discovery of another coercive energy term which is defined on
the whole domain $\Omega$ (phase), displayed in the fourth line of~(\ref{eq:coercive}).
It contains terms of the general form
\[
\|\t^a\partial_t^bq+\t^a\partial_t^b\Psi\cdot v\|_{L^{\infty}_tL^2_x}^2\,\,\,\,\text{and}
\,\,\,\,\|\t^a\partial_t^bq_t+\t^a\partial_t^b\Psi_t\cdot v\|_{L^2_tL^2_x}^2
\]
for $a,b$ as in~(\ref{eq:coercive}). They are {\em intrinsically} linked to the problem and contain information about
the regularity of the divergence of  the velocity $v$.
Taking $a=0$ and $b=1$,  the first term above becomes the norm of the ALE-divergence of $v$, as it is easily seen
from~(\ref{eq:ALEheat}):
\[
\|q_t+v\cdot w\|_{L_t^{\infty}L_x^2}=\|\div_{\Psi}v\|_{L_t^{\infty}L_x^2}.
\]
The gauge condition~(\ref{eq:gauge}) allows us to get 
optimal Sobolev regularity for $\Psi$ and hence for  the temperature function $q$. 
This allows us to prove that the energy $\E$ defined in~(\ref{eq:energy}) is in fact
bounded by the coercive quadratic form (the ``natural energy") $\F$~(\ref{eq:coercive}) dictated by the
Stefan problem.
\par
Condition~(\ref{eq:taylor00}) is the exact equivalent of the Taylor sign condition, necessary for well-posedness of free-surface incompressible 
Euler equations without surface tension~\cite{CoSh07} or the water wave problem~\cite{Wu97}. 
If the initial temperature $q_0$ is nonnegative, it is implied by the Hopf's lemma, at least over a short period of time. 
In a short time regime, we prove a uniform lower bound on $\lambda$ (cf.~(\ref{eq:taylor})), thus enabling us to close the estimates.
\par
In many free-boundary problems, constructing the solution is in general
a challenging problem despite the (possible) availability of good a-priori estimates. 
Our main technical idea to make the construction as straightforward as possible, 
is to regularize the problem via {\em horizontal convolution by layers}
as introduced 
in~\cite{CoSh07} in the study of well-posedness of the incompressible Euler equation on a 
moving domain. In addition to that, we also regularize the Stefan condition $p=0$ on $\Gamma(t)$ 
by modifying it into a Robin-type condition. If $\kappa>0$ is a suitable regularization parameter, to each $\kappa$ we shall associate an
energy functional $\E_\kappa$ which will be shown to satisfy the following energy inequality:
\[
\E_{\kappa}(t)\leq C\E_{\kappa}(0)+C(t+\sqrt{t})P(\sqrt{\E_{\kappa}})
\]
where $P$ is some with the leading order cubic contribution.
Such a polynomial inequality, through a continuity argument
leads to uniform-in-$\kappa$ time of existence $[0,T]$ and the bound
\[
\E_{\kappa}(t)\leq 2C\E_{\kappa}(0).
\]
Passing to the limit as $\kappa\to0$, we recover the solution of the Stefan problem~(\ref{eq:ALE}). Our 
regularization is intrinsic to the problem and it does not rely on formulating a sequence of iterated linear problems.
\par
The second part of this work focuses on the problem of the vanishing surface tension limit. Once 
the well-posedness framework of Theorem~\ref{th:main} is set-up, the idea is rather straightforward.
Namely, at the level of energy, the presence of surface tension simply augments the high-order
energy functional by a $\sigma$-dependent contribution coming from the 
boundary $\Gamma$, so to obtain~(\ref{eq:energysigma}).
The goal is to prove a uniform-in-$\sigma$ upper bound on $\E^{\sigma}$ on a $\sigma$-independent time 
interval $[0,T]$. This is made possible by one fundamental property of $\E^{\sigma}$:
it distinguishes between two boundary energy contributions of general forms
\[
|\sqrt{q,_2}\t^a\g_t^bh|_0^2 \ \text{ and } \ \sigma|\t^{a+1}\g_t^bh|_0^2
\]
for suitable $a,b\in\N_0$. Since the error terms are at least of cubic order, we can afford 
to estimate {\em all} lower order terms in terms of the 
{\em $\sigma$-independent}
energy term, while the two terms with highest number of derivatives get bounded via the {\em $\sigma$-dependent} 
energy contribution. With uniform estimates in hand, we can pass to the limit as $\sigma\to0$.
\par
The plan of the paper is as follows. In Section~\ref{se:reg} we introduce the $\kappa$-regularized problem and the 
associated high-order energy $\E_{\kappa}$. We then state the energy identities (Lemma~\ref{lm:i1}), 
prove that $\E_{\kappa}$ is controlled by
the natural energy $\F_{\kappa}$ (Proposition~\ref{pr:basic}) and finally prove Lemma~\ref{lm:i1}.
In Section~\ref{se:est1}, we provide the energy estimates for the error terms. Passage to the vanishing surface tension limit is explained
in Section~\ref{se:vanish}. In Section~\ref{se:3D} we explain how to extend our results to the three dimensional setting.
\section{Local well-posedness for the classical Stefan problem}
\subsection{A  nonlinear regularization of the Stefan problem:  the $\kappa$-problem}\label{se:reg}
We regularize the problem by using the {\em horizontal convolution by layers}, introduced 
in~\cite{CoSh07} in the study of well-posedness of the incompressible Euler equation on a 
moving domain. 
\begin{definition}[Horizontal convolution-by-layers]\label{D:CONVOLUTION}
Let $\rho_{\kappa}$ be a $C^{\infty}(\R)$-bump function supported in a ball of radius $\kappa$ defined through:
$
\rho_{\kappa}(x):=\frac{1}{\kappa}\rho(\frac{x}{\kappa}),
$
where 
\be
\label{eq:change1}
\rho(x)=\left\{
\begin{array}{l}
c_*e^{-1/(1-|x|^2)},\quad|x|<1,\\
0\quad |x|\geq1
\end{array}
\right.
\ee
and constant $c_*$ is such that $\int_{\R}\rho(x')\,dx'=1$.
For any given $g:\Omega\to\R$ we define the horizontal convolution by layers of $g$
via
\[
\Lambda_{\kappa}g(x^1,x^2):=\int_{\Gamma}g(x^1,x^2)\rho_{\kappa}(x^1-x')\,dx'.
\]
\end{definition}

We also define the standard 2-D sequence of mollifiers: 
$\eta_ \kappa (x) = \kappa^{-2} \eta(x/\kappa)$ 
where 
$\eta(x) = c_*e^{-1/(1-|x|^2)}$ 
for 
$|x|<1$ 
and 
$\eta(x)=0$ 
for 
$|x|\ge 1$, 
and $c_*$ is chosen so that $\int_{ \mathbb{R}  ^2} \eta(x) dx =0$.
To formulate the regularized problem, we introduce the following quantities:
\[
\Psi_{\kappa}(t,x')=(x',h_{\kappa}(t,x'))\,\,\,\,\text{with}\,\,\,\,
h_{\kappa}(t,x'):=\Lambda_{\kappa}\Lambda_{\kappa}h(t,x'),\,\,\,\, x'\in\Gamma
\]
and we define $\Psi_{\kappa}$ on $\Omega$ as a harmonic extension of its boundary value on $\Gamma$ 
as in~(\ref{eq:gauge}).
Analogously to~\eqref{DM} the following trace estimate is true:
\be\label{E:DMKAPPA}
\|\g_t^a\Psi_\kappa\|_{H^s(\Omega)}\lesssim \|\g_t^a\Psi_\kappa\|_{H^{s-0.5}(\Gamma)}, \ \ s>0.5, \ a\in\mathbb N.
\ee
We also denote $J_{\kappa}:=\det \nabla\Psi_{\kappa}.$
Furthermore,
\[
\A:=[\nabla \Psi_{\kappa}]^{-1};\quad\w:=\partial_t\Psi_{\kappa }.
\]
In analogy to~(\ref{eq:bullet}), we introduce 
\be\label{eq:bulletk}
\A^2_{\bullet}:=\frac{(\t h_{\kappa},-1)}{J_{\kappa}};\quad \a:=J_{\kappa}^{-1}\A.
\ee

For $\kappa >0$, 
we now define a nonlinear regularization of the Stefan problem, which we call the $ \kappa $-problem (\ref{eq:ALE}), in which the coefficients are smoothed by use of the
horizontal convolution operator $ \Lambda _ \kappa $.   On a time interval  $[0,T_ \kappa ]$, the $ \kappa $-problem is given as
\begin{subequations}
\label{eq:reg}
\begin{alignat}{2}
q_t- \Delta _{ \Psi_\kappa} q&=-v\cdot \w + \alpha  && \ \text{ in } \  [0, T_ \kappa ]\times \Omega  \,, \label{eq:heatreg} \\
v^i+\A^k_iq,_k&=0&& \ \text{ in } \ [0, T_ \kappa ]\times \Omega \,,\label{eq:veq}\\
q&=-\kappa^2 v\cdot \a^2_{\bullet}+\kappa^2\beta(t,x')&& \ \text{ on } \ [0, T_ \kappa ]\times \Gamma \,,\label{eq:bdryreg}\\
\Psi_t\cdot n_{\kappa}&=-v\cdot n_{\kappa}&& \ \text{ on } \ [0, T_ \kappa ]\times \Gamma\,,\label{eq:evolreg}\\
v\cdot N&=0&& \ \text{ on } \   [0, T_ \kappa ] \times\g\Omega_{\text{top}}\,,\label{eq:topreg}\\
\Psi(0,\cdot)=\Psi_0& \ \   q(0,\cdot)=Q^\kappa_0,&& \,\label{eq:initialreg}
\end{alignat}
\end{subequations}
where 
\begin{subequations}
\label{E:GAUGEKAPPA}
\begin{alignat}{2}
\Delta \Psi_\kappa(t,\cdot)&=0&& \ \text{ in } \ \Omega,\label{eq:diff}\\
\Psi_\kappa(t,x',0)&=(x',h_\kappa(t,x'))&& \ \,\, x'\in\Gamma,\label{eq:damo1}\\
\Psi_\kappa(t,\cdot)&=\text{Id}&& \ \text{ on } \ \partial\Omega_{\text{top}},
\end{alignat}
\end{subequations} 
$ \Delta _{ \Psi_\kappa} q :=  \A_i^j\big(\A^k_iq,_k\big),_j$, and the time-independent forcing function $ \alpha (x)$ is given by\
\be\label{E:ALPHA}
\alpha = J_0^{-2}g_0^2[{q_0},_2]^2 -J_{\kappa }(0,\cdot )^{-2}g_{\kappa}(0, \cdot )^2 [{q_0},_2]^2.
\ee
Here 
\[
g_\kappa(t,x): = \sqrt{1+(\t h_\kappa(t,x))^2}\,,
\]
and $\beta(t,x')$ is defined as
\be\label{eq:beta}
\beta(t,x'):=\sum_{k=0}^2\frac{t^k}{k!}\g_t^k(v\cdot\a^2_{\bullet})|_{t=0} \,.
\ee
Note that we use the subscript and superscript $ \kappa $ on dependent variables in which  there is  explicit use of  the horizontal convolution operator 
$ \Lambda _ \kappa $; of course, all of the $q$, $h$, and $\Psi$ all implicitly depend on $ \kappa $ as well, but for notational convenience, we do not indicate this
implicit dependence on $\kappa $.

The presence of the horizontal mollification operator $ \Lambda _ \kappa $ in the approximate $\kappa$-problem changes the compatibility conditions on the
the initial data.    The addition of the 
 forcing functions $ \alpha (x)$ and $\beta(t,x')$ ensure that the compatibility conditions \eqref{eq:comp} are modified to be
\begin{subequations}
\label{eq:comp-kappa}
\begin{alignat}{2}
Q^ \kappa _0&=0&& \ \text{ on } \ \Gamma\,,\\
\Delta_{\Psi^\kappa_0} Q^\kappa_0&=- J_0^{-2}g_0^2[{q_0},_2]^2&&\ \text{ on } \ \Gamma \,,
\end{alignat}
\end{subequations}
where $\Psi^ \kappa_0 = \Psi_ \kappa (0, \cdot )$.
The approximated initial temperature function $Q^ \kappa _0$ is then defined as the solution of the fourth-order elliptic equation
\begin{subequations}
\label{Q-kappa}
\begin{alignat}{2}
\Delta _{\Psi^\kappa_0}  \Delta _{\Psi^\kappa_0} Q^ \kappa _0&=  \eta_ \kappa * E(\Delta _{\Psi_0}  \Delta _{\Psi_0} q_0) && \ \text{ in } \ \Omega\,,\\
Q^\kappa_0&=0&&\ \text{ on } \ \Gamma \,, \\
\Delta_{\Psi^\kappa_0} Q^\kappa_0&=- J_0^{-2}g_0^2[{q_0},_2]^2&&\ \text{ on } \ \Gamma \,,
\end{alignat}
\end{subequations}
where $E$ continuously maps $H^k(\Omega)$ to $H^k(\mathbb R^2)$ for all $k\ge0$. 
The fourth-order elliptic equation \eqref{Q-kappa} can be written as a  system of second-order equations given by
\begin{subequations}
\label{Q-kappa2}
\begin{alignat}{2}
 \Delta _{\Psi^\kappa_0} Q^ \kappa _0&=  R_0^\kappa && \ \text{ in } \ \Omega\,,\\
 \Delta _{\Psi^\kappa_0} R^ \kappa _0&=  \eta_ \kappa * E( \Delta _{\Psi_0}  \Delta _{\Psi_0} q_0) && \ \text{ in } \ \Omega\,,\\
Q^\kappa_0&=0&&\ \text{ on } \ \Gamma \,, \\
R^\kappa_0&=- J_0^{-2}g_0^2[{q_0},_2]^2&&\ \text{ on } \ \Gamma \,.
\end{alignat}
\end{subequations}
According to the basic elliptic regularity theorem with Sobolev class coefficients, Theorem 3.6 in \cite{ChSh2014},  we obtain estimates for $R^\kappa_0$ and then
$Q^\kappa_0$ which show that
$$
\| Q^\kappa_0\|_4^2 \le C \E(q_0,h_0) \,,
$$
the constant $C$ being independent of $ \kappa $.   Thus we see that $Q^\kappa_0 \to q_0$ in $H^s(\Omega)$, $s\in[0,4)$, and so we have an approximated
 initial temperature
function $Q^\kappa_0 \in H^4(\Omega)$ which satisfies the compatibility conditions \eqref{eq:comp-kappa}.
 Again, from the elliptic system \eqref{Q-kappa2} and the Sobolev embedding theorem, $Q^\kappa_0 \to q_0$ in 
$C^1(\overline{\Omega})$, and hence the Taylor sign condition~\eqref{eq:taylor0} remains valid for $Q^\kappa_0$, so that
\begin{align}\label{E:TAYLORKAPPA}
(Q^\kappa_0),_2\Big|_{t=0}>0 \ \ \text{ for sufficiently small $\kappa>0.$}
\end{align}

In equation~(\ref{eq:evolreg}), $n_{\kappa}$ denotes the outer unit normal with respect to the regularized
surface $\Gamma_{\kappa}$, i.e. in the coordinate representation 
\[
n_{\kappa}=\frac{(\t h_{\kappa},-1)}{\sqrt{1+|\t h_{\kappa}|^2}} = g_ \kappa ^{-1} (\t h_{\kappa},-1) \,.
\]  
Note that the corresponding unit tangent to $\Gamma_{\kappa}$ is given via
\[
\tau_{\kappa}=\frac{\t\Psi_{\kappa}}{|\t\Psi_\kappa|}=\frac{(1,\t h_{\kappa})}{\sqrt{1+|\t h_{\kappa}|^2}} = g_ \kappa ^{-1} (1,\t h_{\kappa})\,.
\]
In analogy to~\eqref{E:HEVOLUTION},  equation~\eqref{eq:evolreg} can be reformulated as an evolution equation for $h$, given by
\begin{align}\label{E:HEVOLUTION}
h_t(t,x)  = g_\kappa(t,x) v \cdot n_\kappa(t,x), \ \ x\in \mathbb T^1 \,.
\end{align}


%
\begin{remark}[The regularization~\eqref{eq:bdryreg}]
The approximate $\kappa$-problem uses
horizontal convolution by layers together with carefully chosen artificial viscosity terms.
This approximation scheme provides a simple existence theory for the $\kappa$-problem while maintaining the nonlinear energy structure.
\end{remark}
\begin{remark}
We introduce the regularization~(\ref{eq:bdryreg}) to circumvent a technical difficulty of closing the energy estimates
at the level of highest-in-time differentiated problem. The problem arises from the commutation of the horizontal convolution operator
appearing in the terms of the following schematic form:
\[
\int_{\Gamma}\Lambda_{\kappa}\Lambda_{\kappa}\Psi_{tt}\cdot\Psi_{ttt}\,\mathcal{T}\,dx',
\] 
where $\mathcal{T}$ is a lower order term. 
Of course, when performing a-priori estimates (i.e. assuming that the solutions to the {\em original} problem
are smooth enough to justify all the integrations by parts), such an issue does not arise.
\end{remark}
\subsubsection{Solutions to the $\kappa$-problem}
\begin{theorem}\label{kthm}
Let $ \kappa >0$ be fixed. 
Let  $(Q_0^\kappa,h_0)\in H^4(\Omega)\times H^4(\Gamma)$ be given initial data satisfying the compatibility conditions~\eqref{eq:comp-kappa}.
Then there is a time $T_ \kappa $ depending on $ \kappa $, such that there exists a unique
solution $(q,h)= (q( \kappa ), h(\kappa))$ to~(\ref{eq:reg}) on  the time interval $[0,T_ \kappa ]$.   
The solution satisfies
\be\label{E:REGKAPPA}
\sum_{a=0}^2\left(\|\partial_t^aq\|_{C^0_tH^{4-2a}_x} + \|\partial_t^a q\|_{L^2_tH^{5-2a}_x} 
+ \|q_{ttt}\|_{L^2_t(H^{1}_x)'}
+|\partial_t^{a+1}h|_{L^2_tH^{4-2a}_x} \right) +\sum_{a=0}^1|\partial_t^{a+1}h|_{C^0_tH^{3-2a}_x}< \infty \,,
\ee
where $H^1(\Omega)'$ denotes the dual space of $ H^1(\Omega) $.
\end{theorem} 
\begin{proof} 
\def\g{\partial}
We briefly sketch the proof. For  $T_ \kappa$ fixed (and taken sufficiently small) and for  $K>0$, we  define the closed set
\begin{align}\label{E:ZK}
Z_K :=  \Big\{ &h: [0,T]\times\Gamma\to\mathbb R, \, \big| \g_t^a h \in C([0,T_\kappa ],H^{4-{2a}}(\Gamma))\cap L^2([0,T_ \kappa ],H^{5-{2a}}(\Gamma)), \ a=0,1,2, \notag \\
& \sum_{a=0}^2\left(|\g_t^a h|_{C^0_tH^{4-2a}}^2 + |\g_t^{a}h|_{L^2_tH^{5-2a}}^2\right) \le K, \
 \text{ $h_0$ and $Q_0^\kappa$ satisfy compatibility conditions~\eqref{eq:comp-kappa}}\Big\} \,.
\end{align}

Given  $h\in Z_K$, we define $h_ \kappa = \Lambda _\kappa ^2 h$, and then we define its harmonic extension 
$\Psi_ \kappa $ by solving \eqref{E:GAUGEKAPPA}.   We then define the corresponding $\A$, $\a$, and $J_\kappa$, and consider the weak formulation
of the parabolic problem~\eqref{eq:heatreg}--\eqref{eq:bdryreg}:
for all test functions $\phi \in H^1(\Omega)$ and a.e. $t\in [0,T]$,
\begin{align} 
\langle q_t J_ \kappa \,, \phi \rangle + \int_\Omega q,_k \A^k_i \, \A^j_i \phi,_j \, J_\kappa dx + {\frac{1}{\kappa ^2}} \int_\Gamma q\, \phi dx_1
= \int_\Omega q,_k \a^k_i \Phi_t^i \, \phi dx + \int_\Omega  \alpha \, \phi \, J_ \kappa  dx  + \int_\Gamma \beta \phi dx_1 \,,
\label{weakform}
\end{align} 
together with the initial condition
$$
q(0, x)  = Q_0^\kappa (x) \,.
$$
Since $\A$, $\a$, and $J_ \kappa $ are  in $C^ \infty $, and since $\A^k_i \A^j_i \ge  \lambda$ for $ \lambda>0$, and the compatibility conditions are satisfied,
 standard parabolic theory provides the existence
of a unique solution on a short time-interval $[0,T_\kappa ]$ with the desired regularity properties.
In particular  it is a standard argument to establish existence of a unique solution in  $q \in L^2(0,T_\kappa ; H^5(\Omega))$
which satisfies the estimate \eqref{E:REGKAPPA}.  

Using a Galerkin scheme on \eqref{weakform}, we obtain unique solutions in $L^2(0,T_\kappa; H^1(\Omega))$ for $q$, $q_t$, and $q_{tt}$
and also find that $q_{ttt} \in L^2(0,T_ \kappa; H^1(\Omega)')$, where $H^1(\Omega)'$ denotes the dual space of $ H^1(\Omega) $.   Standard parabolic
regularity theory, as in \cite{taylor}, shows that $q \in L^2(0,T_\kappa; H^5(\Omega))$ and that $q_t \in L^2(0,T_\kappa; H^3(\Omega))$.

With this solution  $q$, we define the associated velocity field $v$ using \eqref{eq:veq}. We  then update the height function $h$ as
\[
\Phi(h)(t): = h_0 + \int_0^t  g_\kappa(t) v \cdot n_\kappa(\tau)\, d\tau, \ t\in[0,T].
\]
Choosing $T_\kappa$ sufficiently small, it can be shown that  $\Phi$ maps $Z_K$ into itself, and that $\Phi$  is a contraction map.   The fixed-point 
of $\Phi$ is then a 
 solution to the $ \kappa $-problem \eqref{eq:reg}.
\end{proof} 

\begin{remark} 
A priori, the time of existence $T_ \kappa $ may converge to $0$ as $ \kappa \to 0$.   By obtaining $ \kappa $-independent bounds on 
solutions to  (\ref{eq:reg}), we will prove that, in fact, the time of existence is independent of $ \kappa $ and given by
$T>0$.    
\end{remark} 
\subsection{The higher-order energy function compatible with the $ \kappa \to 0$ asymptotics}
The asymptotically consistent
 higher-order energy function  associated to our sequence of regularized $ \kappa $-problems is given by
\begin{align}
\E_{\kappa} & =\E_{\kappa}(q,h):=\sum_{a+2b\leq4}\|\t^a\partial_t^bv\|_{L^2_tL^2_x}^2
+\sum_{a+2b\leq3}\|\t^a\partial_t^bv\|_{C^0_tL^2_x}^2 \notag \\
&+\kappa^2\sum_{a+2b\leq4}|\t^a\partial_t^bh_t|_{L^2_tL^2_x}^2+\kappa^2\sum_{a+2b\leq3}|\t^a\partial_t^bh_t|_{C^0_tL^2_x}^2\notag \\
&+\sum_{b=0}^2\|\g_t^bq\|_{C^0_tH^{4-2b}_x}^2+\sum_{b=0}^2\|\g_t^{b}q\|_{L^{2}_tH^{5-2b}_x}^2  \notag \\
&+\sum_{b=0}^2|\Lambda_{\kappa}h|_{C^0_tH^{4-2b}_x}^2+\sum_{b=0}^1|\g_t\Lambda_{\kappa}h|_{L^2_tH^{3-2b}_x}^2 \label{eq:energy}
\end{align}


As a consequence of Theorem~\ref{kthm} the map $t\mapsto\E_\kappa(t)$ is continuous on $[0,T_\kappa]$.

\subsection{Bounds on lower-order norms}

 Let 
\begin{align*}
\mathcal A_\kappa(t)=& \sum_{a+2b\le2}\|\t^a\partial_t^bv\|_{L^2_tL^2_x}^2
+\sum_{a+2b\leq1}\|\t^a\partial_t^bv\|_{L^{\infty}_tL^2_x}^2\\
& +\kappa^2\sum_{a+2b\leq2}|\t^a\partial_t^bh_t|_{L^2_tL^2_x}^2+\kappa^2\sum_{a+2b\leq1}|\t^a\partial_t^bh_t|_{L^{\infty}_tL^2_x}^2\\
&+\sum_{b=0}^1\|\g_t^bq\|_{L^{\infty}_tH^{2-2b}_x}^2+\sum_{b=0}^1\|\g_t^bq\|_{L^{2}_tH^{3-2b}_x}^2
+\sum_{b=0}^1|\g_t^b\Lambda_{\kappa}h|_{L^{\infty}_tH^{2-2b}_x}^2+|\Lambda_{\kappa}h|_{L^2_tH^{3-2b}_x}^2.
\end{align*}
We then assume that 
\be\label{E:APRIORI}
\mathcal A_\kappa(t) \le \mathcal E_\kappa(0)+1, \ \ t\in[0,T_\kappa].
\ee
By the fundamental theorem of calculus it is easy to see that 
\[
\mathcal A_\kappa(t) \le \mathcal A_\kappa(0) + t \sup_{0\le s\le t}\E_\kappa(t) \le \E_\kappa(0) + t \sup_{0\le s\le t}\E_\kappa(t).
\]
In Section~\ref{se:proof1} we will prove an a priori bound for $\E_\kappa$ independent of $\kappa$ and show that the time of existence $T$ is independent of $\kappa.$ 
The bound~\eqref{E:APRIORI} will then be justified a posteriori using the fundamental theorem of calculus, smallness of $T_\kappa$, and the definition of $\E_\kappa.$ 
By choosing $T_{\kappa}$ possibly smaller we assume that for certain $\delta>0$
\be\label{eq:assumptions}
 \min_{x'\in \Gamma} q,_2(t,x')>\delta\,\,\,\,\,\text{and}\,\,\,\,\,|\t h_\kappa(t,\cdot)|_{\infty}^2\leq1/2, \ \ t\in[0,T_\kappa],
\ee
where $(q,h)$ is the solution of the $\kappa$-problem~\eqref{eq:reg}. 
The first inequality is true by continuity-in-time of the energy $\E_\kappa$ and the Taylor sign condition~\eqref{E:TAYLORKAPPA}.
The second inequality follows from the from the continuity-in-time and smallness of $|\t h_0|_3$~\eqref{E:INITIAL}.

\subsection{The energy identities}\label{su:identities}

In this section we collect the high-order energy identities in two lemmas stated below. 
We use the notation $\mathcal{T}$ for those error terms which in an easy straightforward way are seen to satisfy 
the energy bound of the form:
\[
\int_0^t|\mathcal{T}(s)|\,ds\lesssim tP(\E_{\kappa});
\]
this bound will then always follow from
the standard $L^{\infty}-L^2-L^2$ type estimates.
Here and in the rest of the paper $P(\cdot)$ stands for a generic polynomial satisfying $P(0)=0.$

\begin{lemma}\label{lm:i1}
Assume that $(q,h)$ is a solution to the regularized Stefan problem~(\ref{eq:reg}) given by Theorem~\ref{kthm}. 
Then the following identities hold:
\begin{enumerate}
\item[(i)]
\begin{align}
&\int_{\Omega}|\t^4v|^2
+\frac{1}{2}\frac{d}{dt}\int_{\Gamma}(-q_{,2})\big|\t^4\Lambda_{\kappa}\Psi\cdot \A^2_{\bullet}\big|^2
+\frac{1}{2}\frac{d}{dt}\int_{\Omega}\big(\t^4q+\t^4\Psi_{\kappa}\cdot v\big)^2\notag \\
&+\kappa^2\int_{\Gamma}J_{\kappa}^{-1}|\t^4h_t|^2
=\int_{\Omega}\mathcal{R}_1+\int_{\Gamma}\mathcal{R}_2+\mathcal{T}; \label{eq:enidentity}
\end{align}

\begin{align}
&\int_{\Omega}|\t^2\partial_tv|^2
+\frac{1}{2}\frac{d}{dt}\int_{\Gamma}(-q_{,2})\big|\t^2\partial_t\Lambda_{\kappa}\Psi\cdot \A^2_{\bullet}\big|^2
+\frac{1}{2}\frac{d}{dt}\int_{\Omega}\big(\t^2\partial_tq+\t^2\partial_t\Psi_{\kappa}\cdot v\big)^2 \notag \\
&+\kappa^2\int_{\Gamma}J_{\kappa}^{-1}|\t^2\g_th_t|^2
=\int_{\Omega}\mathcal{R}_3+\int_{\Gamma}\mathcal{R}_4+\mathcal{T}; \label{eq:en34}
\end{align}

\begin{align}
&\int_{\Omega}|\partial_t^2v|^2
+\frac{1}{2}\frac{d}{dt}\int_{\Gamma}(-q_{,2})\big|\partial_t^2\Lambda_{\kappa}\Psi\cdot\A^2_{\bullet}\big|^2
+\frac{1}{2}\frac{d}{dt}\int_{\Omega}\big(\partial_t^2q+\partial_t^2\Psi_{\kappa}\cdot v\big)^2  \notag \\
&+\kappa^2\int_{\Gamma}J_{\kappa}^{-1}|\g_{tt}h_t|^2
=\int_{\Omega}\mathcal{R}_5+\int_{\Gamma}\mathcal{R}_6+\mathcal{T},  \label{eq:en56}
\end{align}
where $\mathcal{R}_i$, $i=1,\dots 6$, are error terms given below respectively by
~(\ref{eq:remainder}), ~(\ref{eq:remaindergamma}), ~(\ref{eq:r3}), ~(\ref{eq:r4}), ~(\ref{eq:r5}) and~(\ref{eq:r6}).
\item[(ii)]
\begin{align}
&\frac{1}{2}\frac{d}{dt}\int_{\Omega}|\t^3v|^2
+\int_{\Gamma}(-q,_2)|\t^3\Lambda_{\kappa}\Psi_t\cdot\A^2_{\bullet}|^2
+\int_{\Omega}(\t^3q_t+\t^3\w\cdot v)^2 \notag \\
&+\frac{\kappa^2}{2}\frac{d}{dt}\int_{\Gamma}J_{\kappa}^{-1}|\t^3h_t|^2
=\int_{\Omega}\mathcal{S}_1+\int_{\Gamma}\mathcal{S}_2+\mathcal{T}; \label{eq:s12}
\end{align}

\begin{align}
&\frac{1}{2}\frac{d}{dt}\int_{\Omega}|\t v_t|^2
+\int_{\Gamma}(-q,_2)|\t\Lambda_{\kappa}\Psi_{tt}\cdot \A^2_{\bullet}|^2+\int_{\Omega}(\t q_{tt}+\t\w_t\cdot v)^2  \notag \\
&+\frac{\kappa^2}{2}\frac{d}{dt}\int_{\Gamma}J_{\kappa}^{-1}|\t h_{tt}|^2 
=\int_{\Omega}\mathcal{S}_3+\int_{\Gamma}\mathcal{S}_4+\mathcal{T}, \label{eq:s34}
\end{align}
where $\mathcal{S}_i$, $i=1,\dots 4$, are error terms given below respectively 
by~(\ref{eq:sremainder}), ~(\ref{eq:sremaindergamma}), ~(\ref{eq:s3}), ~(\ref{eq:s4}).
\end{enumerate}
\end{lemma}

We postpone the proof of Lemma~\ref{lm:i1} to Section~\ref{su:proofs}. 


\subsection{Equivalence of the higher-order norm $\E_{\kappa}$ and the natural energy function $\F_\kappa$}
By summing the left-hand sides of the identities~(\ref{eq:enidentity})--(\ref{eq:s34}) from Lemma~\ref{lm:i1}, 
the {\em natural} coercive quadratic form $\mathcal{F}_{\kappa}$ that arises as the energy
takes the form
\be
\label{eq:coercive}
\begin{array}{l}
\displaystyle
\F_{\kappa}=\sum_{a+2b\leq4}\|\t^a\partial_t^bv\|_{L^2_tL^2_x}^2
+\frac{1}{2}\sum_{a+2b\leq3}\|\t^a\partial_t^bv\|_{L^{\infty}_tL^2_x}^2\\
\displaystyle
\quad
+\kappa^2\sum_{a+2b\leq4}|J_{\kappa}^{-1/2}\t^a\partial_t^bh_t|_{L^2_tL^2_x}^2+\frac{\kappa^2}{2}\sum_{a+2b\leq3}|J_{\kappa}^{-1}\t^a\partial_t^bh_t|_{L^{\infty}_tL^2_x}^2\\
\displaystyle
\quad+\frac{1}{2}\sum_{a+2b\leq4;}|\sqrt{q,_2}J_{\kappa}^{-1/2}\t^a\partial_t^b\Lambda_{\kappa}h|_{L^{\infty}_tL^2_x}^2
+\sum_{a+2b\leq3}|\sqrt{q,_2}J_{\kappa}^{-1/2}\t^a\partial_t^b\Lambda_{\kappa}h_{t}|_{L^2_tL^2_x}^2\\
\displaystyle
\quad+\frac{1}{2}\sum_{a+2b\leq4;}\|\t^a\partial_t^bq+\t^a\partial_t^b\Psi_{\kappa}\cdot v\|_{L^{\infty}_tL^2_x}^2
+\sum_{a+2b\leq3;}\|\t^a\partial_t^bq_t+\t^a\partial_t^b\Psi_{\kappa t}\cdot v\|_{L^2_tL^2_x}^2.
\end{array}
\ee
The mathematical reason for imposing the Taylor sign condition~(\ref{eq:taylor}) now becomes apparent.
In order for the second line in the definition of $\F_{\kappa}$~(\ref{eq:coercive}) above to make sense
we must have 
\[
\min_{x'\in\Gamma}(q,_2)(t,x',0)>0,
\]
as it was assumed in~(\ref{eq:taylor}) for the (unregularized) classical Stefan problem.
In order to perform the estimates in the next section, it 
is crucial to show that the energy $\E_{\kappa}$ is bounded by $\F_\kappa$.
To prove this statement we first establish the following temperature estimate.


\begin{lemma}\label{L:TEMPERATURE}
Let $(q,h)$ be a solution of the regularized problem~\eqref{eq:reg} given by Theorem~\ref{kthm}.
Assume that the a priori assumption~\eqref{E:APRIORI} holds on $[0,T_\kappa].$ Then
\begin{enumerate}
\item[(i)]
\begin{align}\label{E:ELLIPTIC1}
\sum_{a=0}^2\|\partial_t^aq\|_{L^{\infty}_tH^{4-2a}_x}^2
\lesssim \F_{\kappa} \ \ \text{ on } \ [0,T_\kappa].
\end{align}
\item[(ii)]
\begin{align}\label{E:ELLIPTIC2}
\|q\|_{L^2_tH^{4.5}_x}^2 + \sum_{a=1}^2\|\g_t^aq\|_{L^2_tH^{5-2a}}^2
\lesssim \F_{\kappa} \ \ \text{ on } \ [0,T_\kappa].
\end{align}
\end{enumerate}
\end{lemma}



\begin{proof}
We use elliptic regularity theory and the a priori assumption~\eqref{E:APRIORI} to show that
$\|q\|_{L^\infty_tH^2_x}  \lesssim \F_{\kappa}$  since 
$\Delta_{\Psi_\kappa}q=q_t+v\cdot \w+\alpha,$ $\|q_t\|_{L^{\infty}L^2}\lesssim \F_{\kappa},$
$\|v\|_{L^{\infty}_tL^2_x}\|w\|_{L^{\infty}_tL^2_x}\leq\|v\|_{L^{\infty}_tL^2_x}|h_t|_{L^{\infty}_tH^1_x}\lesssim \F_{\kappa}$, and 
$
\|\alpha\|_{L^\infty_tL^2_x} \lesssim \F_\kappa(0) \lesssim \F_\kappa.
$
Differentiating~(\ref{eq:heatreg}) with respect to $x^j$ ($j=1,2$), we obtain that
$\Delta_{\Psi_\kappa}q,_j=(\A^m_i\A^n_i),_jq,_{mn}+(q_t+v\cdot \w),_j + \alpha,_j$.
Furthermore, since $q,_j=\Psi_{\kappa,j}\cdot v$ we have that
\begin{align*}
\|q,_{jt}\|&\lesssim\|\Psi_{\kappa,jt}\|_{L^{\infty}_tL^{\infty}_x}\|v\|_{L^{\infty}_tL^2_x}+\|\Psi_{\kappa,j}\|_{L^{\infty}_tL^{\infty}_x}\|v_t\|_{L^{\infty}_tL^2_x}\\
&\lesssim|h_t|_{L^{\infty}_tH^2_x}\|v\|_{L^{\infty}_tL^2_x}+\|\nabla\Psi_\kappa\|_{L^{\infty}_tH^{1.5}}\|v_t\|_{L^{\infty}_tL^2_x}
\lesssim\F_{\kappa},
\end{align*}
where we have used the trace bound~\eqref{E:DMKAPPA} and the a priori assumption~\eqref{E:APRIORI}.
Note that
\begin{align*}
\|(v\cdot \w),_j\|_{L^{\infty}_tH^1_x}& \lesssim\|v\|_{L_t^{\infty}H^1_x}\|\w\|_{L^{\infty}_tL^{\infty}_x}+\|v\|_{L^{\infty}_tL^2_x}\|\w,_j\|_{L^{\infty}_tL^{\infty}_x}\\
&\lesssim \|q\|_{L^{\infty}_tH^2_x}|h_{\kappa,t}|_{L^{\infty}_tH^1_x}+\|q\|_1|h_{\kappa,t}|_{L^{\infty}_tH^2_x}\lesssim \F_{\kappa},
\end{align*}
where we have used the bound~\eqref{E:APRIORI} in the last estimate.
It is easy to see that $\|(\A^m_i\A^n_i),_jq,_{mn}\|_{L^{\infty}_tL^2_x}\lesssim \F_{\kappa}\left(1+P(\mathcal A_\kappa\right))\lesssim \F_\kappa,$ where $P$ stands for a generic polynomial.
Finally, $\|\nabla\alpha\|_{L^\infty_t L^2_x}\lesssim \F_\kappa.$
Thus, by the elliptic theory again, we conclude 
\[
\|q\|_{L^{\infty}_tH^3_x}^2\lesssim \F_{\kappa}.
\]
Differentiating~(\ref{eq:heatreg}) with respect to $t$, we obtain
$\Delta_{\Psi_\kappa}q_t=-(\A^m_i\A^n_i),_tq,_{mn}+v_{tt}+(v\cdot \w)_t$
(since $\alpha$ is independent of $t.$). Again, using $\|v_{tt}\|^2_{L^{\infty}_tL^2_x}\lesssim\F_{\kappa}$,
the previous estimates and the bound $|h_{\kappa,t}|_{L^{\infty}_tH^2_x}^2\lesssim \F_{\kappa}$, elliptic regularity
implies $\|q_t\|_{L^{\infty}_tH^2_x}\lesssim\F_{\kappa}$. Furthermore 
$\|q_{tt}\|_{L^{\infty}_tL^2_x}^2\lesssim\|q_{tt}+\w_t\cdot v\|_{L^{\infty}_tL^2_x}^2+\|\w_t\cdot v\|_{L^{\infty}_tL^2_x}^2\lesssim \F_{\kappa}$.
The last equality follows from the third line on the definition~(\ref{eq:energy}) of $\F_{\kappa}$ and a simple bound on the $L^{\infty}_tL^{\infty}_x$-norm 
of $v$, which follows from Sobolev embedding.
Finally, choose any $j,k\in\{1,2\}$. Applying $\g_{x^j}\g_{x^k}$ to~(\ref{eq:heatreg}), we arrive at the elliptic equation
\[
\Delta_{\Psi_\kappa}q,_{jk}=-(\A^m_i\A^n_i),_jq,_{mnk}-(\A^m_i\A^n_i),_kq,_{mnj}-(\A^m_i\A^n_i),_{jk}q,_{mn}+(q_t+v\cdot \w),_{jk} + \alpha,_{jk}.
\]
By the estimates already derived above,~(\ref{eq:energy}),~\eqref{E:ALPHA}, and~\eqref{E:APRIORI}, the right-hand side is bounded by $\F_{\kappa}\left(1+P(\mathcal A_\kappa)\right)\lesssim \mathcal F_\kappa$
in $L^{\infty}_tL^2_x$-norm. Thus, by elliptic regularity, we finally conclude $\|q\|_{L^{\infty}H^4}\lesssim \F_{\kappa}$, concluding the proof
of~\eqref{E:ELLIPTIC1}.

To prove~\eqref{E:ELLIPTIC2} we start with the easiest case $b=2$.
For $j=1,2$, we have
\[
q,_{jtt}=(\Psi_{\kappa,j}\cdot v)_{tt}=\Psi_{\kappa,jtt}\cdot v+2\Psi_{\kappa,jt}\cdot v_t+\Psi_{\kappa,j}\cdot v_{tt}.
\]
From the above we easily infer that
\begin{align}
\int_0^t\|\nabla q_{tt}\|_{0}^2\,d\tau &\lesssim
\int_0^t\left(\|\nabla \Psi_{\kappa,tt}\|_{0}^2\|v\|_{\infty}^2+\|\Psi_{\kappa,jt}\|_{\infty}^2\|v_t\|_0^2+\|\Psi_{\kappa,j}\|_{\infty}^2\|v_{tt}\|_0^2\right)\,d\tau \notag \\
& \lesssim P(\F_{\kappa}).\label{E:EQUIV1}
\end{align}
We thereby used the trace estimate~\eqref{E:DMKAPPA} to obtain
\[
\|\nabla \Psi_{\kappa,tt}\|_{0}^2 \lesssim |h_{\kappa,tt}|_{0.5}^2 \lesssim \F_\kappa,
\]
where we have used the definition~\eqref{eq:coercive} in the last bound above. On the other hand, using the a priori bound~\eqref{E:APRIORI} and the Sobolev embedding we conclude that
$\|v\|_{\infty}\lesssim \|v\|_{1}\lesssim \|q\|_2\le \E_\kappa(0)+1\lesssim 1.$ The remaining two terms on the right-hand side of~\eqref{E:EQUIV1} are estimated in a similar fashion.
If $b=1$ we apply the same ideas using~\eqref{E:ELLIPTIC1},~\eqref{E:APRIORI}, and the Sobolev embeddings.
To prove $\|q\|_{L^2_tH^{4.5}_x}^2 \lesssim  \F_\kappa$ we need to use an interpolation estimate. The strategy consists of estimating
$\|q\|_{L^2_tH^{5}_x}^2$ and $\|q\|_{L^2_tH^{4}_x}^2$ separately and then interpolating between the two estimates. The reader may consult~\cite{HaSh3} for the details.
\end{proof}


\begin{remark}
The regularity  of $q\in L^2_tH^{4.5}_x$ can in fact be improved.
\end{remark}


\begin{lemma}[Optimal regularity for $\Psi_ \kappa $ and $q$]\label{optimal_q}  Suppose that the pair $(q,h)$ is  a solution of the 
$\kappa $-problem~\eqref{eq:reg} given by Theorem~\ref{kthm}, and 
 that the basic assumption~\eqref{E:APRIORI} holds on $[0,T_\kappa].$ Then 
$
\int_0^{T_ \kappa } \left( \| \Psi_\kappa\|_5^2  +  \| q\|_5^2\right) dt  \le C \mathcal{E}_\kappa (t) \,.
$
\end{lemma} 


\begin{proof} 
{\it Step 1.}  We will first prove that $\int_0^{T_ \kappa }  \| \Psi_\kappa\|_5^2 dt  \le C \mathcal{E} (t) $.  Since $q=0$ on $\Gamma$ it follows that
\begin{equation}\label{ms10}
v(x,t) \cdot \tau_ \kappa(x,t)  =0 \text { on } \Gamma\,,
\end{equation} 
where we recall that $\tau_ \kappa $ is the unit tangent vector to $\Psi_ \kappa (t, \Gamma)$.  Applying the horizontal derivative $\bar \partial$ to \eqref{ms10},
and using the fact that $ \bar \partial \tau_\kappa  = g_\kappa^{-1}  \bar\partial^2\Psi_\kappa \cdot n_ \kappa \ n_ \kappa $ and that 
$\bar\partial^2 \Psi_\kappa \cdot n_ \kappa = -g_ \kappa ^{-1}  \bar \partial^2 h_ \kappa $, 
we find that
\be\label{E:HIMPROVED}
\bar \partial^2  h_ \kappa = \frac{ g_\kappa ^2 \bar \partial v\cdot \tau_ \kappa }{ v \cdot n_\kappa } \,.
\ee
The dominator in~\eqref{E:HIMPROVED} is strictly positive for $T_ \kappa $ small enough by  the Taylor sign condition
\eqref{E:TAYLORKAPPA}.
For any $W:\Omega\to\mathbb R^2$ we define 
\be\label{E:divcurldef}
\curl_{\Psi} W = \varepsilon_{ji}A^s_jW^i,_s
\ee
where $\varepsilon_{21}=-\varepsilon_{12}=1,$ $\varepsilon_{11}=\varepsilon_{22}=0$.
By the tangential trace inequality (see~\cite{ChSh2014}), 
\be\label{E:KEYTRACE}
\big| \bar \partial^4 v\cdot \tau_ \kappa\big|_{H^{-\frac12}(\Gamma)} \lesssim \|\curl_\Psi \t^3v\|_{L^2(\Omega)} + \|\t^4 v\|_{L^2(\Omega)},
\ee
We observe that
\begin{align*}
\curl_\Psi \t^3 v & = \t^3\left(\curl_\Psi v\right)  - \varepsilon_{ji}\sum_{m=1}^3c_m\t^m A^s_j \t^{3-m}v^j,_s \\
& =  - \varepsilon_{ji}\sum_{m=1}^3 c_m\t^m A^s_j \t^{3-m}v^j,_s,
\end{align*}
where we have used the identity $\curl_\Psi v = (\curl \nabla p)\circ \Psi=0.$
Using the Cauchy-Schwarz inequality and the definition of $\E_\kappa$ we obtain
\[
\|\curl_\Psi \t^3 v \|_0 \lesssim \sqrt{\E_\kappa}.
\]
From~\eqref{E:KEYTRACE} and the definition~\eqref{eq:energy} of $\E_\kappa$ we obtain
\be\label{E:HIMPROVED}
 \int_0^{t}\big|\bar \partial^4 v\cdot \tau_ \kappa\big|_{H^{-\frac12}(\Gamma)}^2\,d\sigma \,  \lesssim \,  \E_\kappa(t), \ \ 0\le t\le T_\kappa.
\ee
Using~\eqref{E:HIMPROVED} and~\eqref{E:KEYTRACE} it follows easily that $\int_0^{t}|h_\kappa|_{4.5}^2\,d\sigma \le \E_\kappa(t),$ \ \ $0\le t\le T_\kappa.$ Recalling that
$\Psi_\kappa$ is the harmonic extension of $(x',h_\kappa(x')),$ $x'\in\Gamma$ the optimal trace inequality~\eqref{DM} implies that 
$\int_0^{t}\|\Psi_\kappa\|_5^2\,d\sigma\le C\E_\kappa(t)$ for any $t\in[0, T_\kappa].$
\\
{\it Step 2.}  
The fact that $ \int_0^{T_\kappa} \| q\|_5^2 dt  \le C \mathcal{E} (t)$ follows from Step 1, and the elliptic regularity result in Theorem 3.6 in 
\cite{ChSh2014}.   
\end{proof}



As a consequence of Lemmas~\ref{L:TEMPERATURE} and~\ref{optimal_q} we obtain the following key bound between the norm $\E_\kappa$ and the energy $\F_\kappa$.


\begin{prop}[Norm-energy equivalence]\label{pr:basic}
Let $(q,h)$ be a solution of the $\kappa $-problem~\eqref{eq:reg} given by Theorem~\ref{kthm}.
Assume that the a priori assumption~\eqref{E:APRIORI} holds on $[0,T_\kappa].$ Then 
$
\E_{\kappa}\lesssim \F_\kappa
$
on $[0,T_\kappa]$.
\end{prop}



\begin{proof}
Due to the Taylor sign condition~(\ref{eq:assumptions}),  the boundary integrals  in 
$\E_{\kappa}$and  $\F_{\kappa}$
satisfy
\beas
&&\sum_{b=0}^2|\Lambda_{\kappa}h|_{L^{\infty}_tH^{4-2b}_x}^2+\sum_{b=0}^1|\Lambda_{\kappa}h_t|_{L^2_tH^{3-2b}_x}^2\\
&&{}{}\lesssim\sum_{a+2b\leq4;}|\sqrt{-q,_2}\t^a\partial_t^b\Lambda_{\kappa}h|_{L^{\infty}_tL^2_x}^2
+\sum_{a+2b\leq3}|\sqrt{-q,_2}\t^a\partial_t^b\Lambda_{\kappa}h_{t}|_{L^2_tL^2_x}^2.
\eeas
The remaining estimates now follow directly from Lemmas~\ref{L:TEMPERATURE} and~\ref{optimal_q}.
\end{proof}



%
%
%
\subsection{Proof of Lemma~\ref{lm:i1}}\label{su:proofs}
{\em Proof of part (i) of Lemma~\ref{lm:i1}.}
Applying the tangential differential operator $\t^4$ 
to the equation~(\ref{eq:veq}), multiplying it by $\t^4v^i$ and
integrating over $\Omega$, we obtain
\be\label{eq:enid}
\big(\t^4v^i+\t^4\A^k_iq_{,k}+\A^k_i\t^4q_{,k},\,\t^4v^i\big)_{L^2}=
\sum_{l=1}^3c_l\big(\t^l\A^k_i\t^{4-l}q_{,k},\,\t^4v^i\big)_{L^2},
\ee
where $c_{l}={4\choose l}$.
Recalling~(\ref{eq:comm}), we write
\be\label{eq:iden2}
\t^4\A^k_i=-\A^s_i\t^4\Psi^r_{\kappa,s}\A^k_r+\{\t^4,\A^k_i\},
\ee
where $\{\t^4,\A^k_i\}$ stands for the lower order commutator defined in~(\ref{eq:comm}).
With this identity, we obtain
\be\label{eq:t1}
\begin{array}{l}
\displaystyle
\big(\t^4\A^k_iq_{,k},\t^4v^i\big)_{L^2(\Omega)}
=-\big(\A^s_i\t^4\Psi^r_{\kappa,s}\A^k_rq_{,k},\t^4v^i\big)_{L^2(\Omega)}
+\big(\{\t^4,\A^k_i\}q_{,k},\t^4v^i\big)_{L^2(\Omega)}\\
\displaystyle
=-\int_{\Gamma}q_{,k}\A^s_i\t^4\Psi_{\kappa}^r\A^k_r\t^4v^iN^s
+\int_{\Omega}\A^s_i\t^4\Psi_{\kappa}^r\A^k_rq,_k\t^4v^i_{,s}
+\mathcal{T}\\
\displaystyle
=-\int_{\Gamma}q_{,k}\A^s_i\t^4\Psi_{\kappa}^r\A^k_r\t^4v^iN^s
-\int_{\Omega}\A^s_i\t^4\Psi_{\kappa}^rv^r\t^4v^i_{,s}
+\mathcal{T}.
\end{array}
\ee
where we have used $(\A^s_i),_s=0$ and the identity $v^r=-\A^k_rq_{,k}$ to write the last
line more concisely. Furthermore, integrating by parts with respect to $x^k$
\be\label{eq:est2}
\big(\A^k_i\t^4q_{,k},\,\t^4v^i\big)_{L^2}=\int_{\Omega}\A^k_i\t^4q_{,k}\t^4v^i
=\int_{\Gamma}\A^k_i\t^4q\t^4v^i\cdot N^k-\int_{\Omega}\A^k_i\t^4q\t^4v^i_{,k}.
\ee
Note that the the boundary contribution coming from the fixed boundary $\g\Omega_{\text{top}}$ vanishes 
due to the boundary condition~(\ref{eq:topreg}), which further reduces to $v^2=0$ on $\g\Omega_{\text{top}}$.
Summing~(\ref{eq:t1}) and~(\ref{eq:est2}), we obtain
\be\label{eq:new1}
\begin{array}{l}
\displaystyle
\big(\t^4\A^k_iq_{,k}+\A^k_i\t^4q_{,k},\t^4v^i\big)_{L^2(\Omega)}
=-\int_{\Gamma}q_{,k}\A^s_i\t^4\Psi_{\kappa}^r\A^k_r\t^4v^iN^s
+\int_{\Gamma}\A^k_i\g^4q\g^4v^i\cdot N^k\\
\displaystyle
\quad-\int_{\Omega}\A^k_i\t^4v^i,_k\big(\t^4q+\t^4\Psi_{\kappa}\cdot v\big)
+\mathcal{T}.
\end{array}
\ee
The first three terms on the right-hand side of~(\ref{eq:new1}) will
be the source of positive definite quadratic contributions to the energy. 
To extract the quadratic coercive contribution from the first integral on the right-hand side 
of~(\ref{eq:new1}), we simplify it to 
\beas
-\int_{\Gamma}q_{,k}\A^s_i\t^4\Psi_{\kappa}^r\A^k_r\t^4v^iN^s
&=&\int_{\Gamma}q,_2\A^2_i\t^4\Psi_{\kappa}^r\A^2_r\t^4v^i+\int_{\Gamma}q,_1\A^2_i\t^4\Psi_{\kappa}^{r}\A^1_r\t^4v^i\\
&=&\int_{\Gamma}q,_2\t^4\Psi_{\kappa}\cdot \A^2_{\bullet}\t^4v\cdot \A^2_{\bullet}
+\int_{\Gamma}q,_1\t^4\Psi_{\kappa}^{r}\A^1_r\t^4v\cdot\A^2_{\bullet}.
\eeas
We rewrite the expression $\t^4v\cdot\A^2_{\bullet}$ and thereby use the boundary condition~(\ref{eq:evolreg}):
\beas
\t^4v\cdot \A^2_{\bullet}&=&\t^4w\cdot \A^2_{\bullet}+\t^4(v+w)\cdot\A^2_{\bullet}\\
&=&\t^4w\cdot\A^2_{\bullet}+\t^4\big(\underbrace{(v+w)\cdot\A^2_{\bullet}}_{=0}\big)-\sum_{l=1}^4a_l\t^{4-l}(v+w)\cdot\t^l\A^2_{\bullet}\\
&=&\t^4w\cdot\A^2_{\bullet}-\sum_{l=1}^4a_l\t^{4-l}(v+w)\cdot\t^l\A^2_{\bullet}.
\eeas
Due to the above identity and recalling $\Psi_{\kappa}=\Lambda_{\kappa}\Lambda_{\kappa}\Psi$, we obtain
\be\label{eq:iden4}
\int_{\Gamma}q,_2\t^4\Psi_{\kappa}\cdot n_{\kappa}\t^4v\cdot\A^2_{\kappa}
=\int_{\Gamma}q,_2\t^4\Lambda_{\kappa}\Lambda_{\kappa}\Psi\cdot\A^2_{\bullet}\t^4w\cdot\A^2_{\bullet}
-\sum_{l=1}^4a_l\int_{\Gamma}q,_2\t^4\Psi_{\kappa}\cdot\A^2_{\bullet}\t^{4-l}(v+w)\cdot\t^l\A^2_{\bullet}.
\ee
The first term on the right-hand side of~(\ref{eq:iden4}) is rewritten in the following way
\be\label{eq:iden4.1}
\begin{array}{l}
\displaystyle
\int_{\Gamma}q,_2\t^4\Lambda_{\kappa}\Lambda_{\kappa}\Psi\cdot \A^2_{\bullet}\t^4w\cdot \A^2_{\bullet}
=\int_{\Gamma}q,_2\Lambda_{\kappa}\t^4\Psi\cdot \A^2_{\bullet}
\Lambda_{\kappa}\t^4\Psi_t\cdot \A^2_{\bullet}\\
\displaystyle
\quad+\int_{\Gamma}\t^4\Lambda_{\kappa}\Psi
\Big[\Lambda_{\kappa}\big[q,_2\big(\t^4\Psi_t\cdot \A^2_{\bullet}\big)\A^2_{\bullet}\big]
-q,_2\big(\t^4\Lambda_{\kappa}\Psi_t\cdot \A^2_{\bullet}\big)\A^2_{\bullet}\Big]\\
\displaystyle
=\frac{1}{2}\g_t\int_{\Gamma}q,_2\big|\t^4\Lambda_{\kappa}\Psi\cdot \A^2_{\bullet}\big|^2
-\frac{1}{2}\int_{\Gamma}q_{,2t}\big|\t^4\Lambda_{\kappa}\Psi\cdot \A^2_{\bullet}\big|^2
-\int_{\Gamma}q,_2\t^4\Lambda_{\kappa}\Psi\cdot \A^2_{\bullet}\t^4\Lambda_{\kappa}
\Psi\cdot \A^2_{\bullet t}\\
\displaystyle
\quad+\int_{\Gamma}\Lambda_{\kappa}\t^4\Psi
\Big[\Lambda_{\kappa}\big[q,_2\big(\t^4\Psi_t\cdot \A^2_{\bullet}\big)\A^2_{\bullet}\big]
-q,_2\big(\Lambda_{\kappa}\t^4\Psi_t\cdot \A^2_{\bullet}\big)\A^2_{\bullet}\Big]\\
\displaystyle
=\frac{1}{2}\g_t\int_{\Gamma}q,_2\big|\t^4\Lambda_{\kappa}\Psi\cdot \A^2_{\bullet}\big|^2
+\int_{\Gamma}\Lambda_{\kappa}\t^4\Psi
\Big[\Lambda_{\kappa}\big[q,_2\big(\t^4\Psi_t\cdot \A^2_{\bullet}\big)\A^2_{\bullet}\big]
-q,_2\big(\Lambda_{\kappa}\t^4\Psi_t\cdot \A^2_{\bullet}\big)\A^2_{\bullet}\Big]+\mathcal{T}.
\end{array}
\ee
The second term on the right-hand side of~(\ref{eq:new1}) turns into
\beas
\int_{\Gamma}\A^k_i\t^4q\t^4v^i\cdot N^k
&=&-\int_{\Gamma}\t^4q\t^4v\cdot\A^2_{\bullet}\\
&=&\kappa^2\int_{\Gamma}\t^4h_t\big(\t^4(v\cdot\A^2_{\bullet})
-\sum_{l=0}^3a_l\t^lv\t^{4-l}\A^2_{\bullet}\big)+\kappa^2\int_{\Gamma}\beta(t,x')\t^4v\cdot\A^2_{\bullet}\\
&=&\kappa^2\int_{\Gamma}J_{\kappa}^{-1}|\t^4h_t|^2
-\kappa^2\sum_{l=0}^3a_l\int_{\Gamma}\t^4h_t\t^lv\t^{4-l}\A^2_{\bullet}+\kappa^2\int_{\Gamma}\t^4\beta(t,x')\t^4v\cdot\A^2_{\bullet},
\eeas
where we have used the boundary condition~(\ref{eq:bdryreg}) in the second equality 
above (recall $v\cdot\a^2_{\bullet}=w\cdot\a^2_{\bullet}=h_t$).
As to the third term on the right-hand side of~(\ref{eq:new1}), note that
\[
\A^k_i\t^4v^i,_k=\t^4(\A^k_iv^i,_k)-\sum_{l=1}^4c_l\t^l\A^k_i\t^{4-l}v^i,_k
=-\t^4(q_t+v\cdot \w)-\sum_{l=1}^4c_l\t^l\A^k_i\t^{4-l}v^i,_k,
\]
where $\A^k_iv^i,_k=-\div_{\Psi_{\kappa}}v=-(q_t+v\cdot \w)+\alpha$ by the parabolic equation~(\ref{eq:heatreg}).
Thus
\be\label{eq:new2}
\begin{array}{l}
\displaystyle
-\int_{\Omega}\A^k_i\t^4v^i,_k\big(\t^4q+\t^4\Psi_{\kappa}\cdot v\big)\\
\displaystyle
=\int_{\Omega}\t^4(q_t+\Psi_{\kappa t}\cdot v-\alpha)\big(\t^4q+\t^4\Psi_{\kappa}\cdot v\big)
+\sum_{l=1}^4c_l\int_{\Omega}\t^l\A^k_i\t^{3-l}v^i,_k\big(\t^4q+\t^4\Psi_{\kappa}\cdot v\big)\\
\displaystyle
=\frac{1}{2}\partial_t\int_{\Omega}\big(\t^4q+\t^4\Psi_{\kappa}\cdot v\big)^2
+\int_{\Omega}(\sum_{l=1}^4d_l\t^{4-l}\Psi_{\kappa t}\cdot\t^lv-\t^4\Psi_{\kappa}\cdot v_t)
\big(\t^4q+\t^4\Psi_{\kappa}\cdot v\big)\\
\displaystyle
\quad+\sum_{l=1}^4c_l\int_{\Omega}\t^l\A^k_i\t^{3-l}v^i,_k\big(\t^4q+\t^4\Psi_{\kappa}\cdot v\big) + \int_{\Omega}\t^4\alpha \big(\t^4q+\t^4\Psi_{\kappa}\cdot v\big).
\end{array}
\ee
Combining~(\ref{eq:new1}), ~(\ref{eq:iden4}), ~(\ref{eq:iden4.1}) and~(\ref{eq:new2}) we obtain
the identity~(\ref{eq:enidentity})
with the error terms $\mathcal{R}_1$ and $\mathcal{R}_2$ given by: 
\begin{align}
\mathcal{R}_1:= &
\sum_{l=1}^3c_l\t^{l}\A^k_i\t^{4-l}q_{,k}\t^4v^i
-\big(\sum_{l=1}^4c_l\t^l\A^k_i\t^{4-l}v^i,_k
+\sum_{l=1}^4d_l\t^{4-l}\w\cdot\t^lv-\t^4\Psi_{\kappa}\cdot v_t\big)\big(\t^4q+\t^4\Psi_{\kappa}\cdot v\big) \notag \\
&  +\t^4\alpha \big(\t^4q+\t^4\Psi_{\kappa}\cdot v\big);\label{eq:remainder}
\end{align}
\be
\label{eq:remaindergamma}
\begin{array}{l}
\displaystyle
\mathcal{R}_2:=
-\Lambda_{\kappa}\t^4\Psi
\Big[\Lambda_{\kappa}\big[(-q,_2)\big(\t^4\Psi_t\cdot \A^2_{\bullet}\big)\A^2_{\bullet}\big]
-(-q,_2)\big(\t^4\Lambda_{\kappa}\Psi_t\cdot \A^2_{\bullet}\big)\A^2_{\bullet}\Big]\\
\displaystyle
+\sum_{l=1}^4a_l(-q,_2)\t^4\Psi_{\kappa}\cdot \A^2_{\bullet}\t^{4-l}(v+w)\cdot\t^l \A^2_{\bullet}
+q,_1\t^4\Psi_{\kappa}^{r}\A^1_r\t^4v\cdot\A^2_{\bullet}
+\kappa^2\sum_{l=0}^3a_l\t^4h_t\t^lv\t^{4-l}\A^2_{\bullet}\\
\displaystyle
-\kappa^2\t^4\beta\t^4v\cdot\A^2_{\bullet}.
\end{array}
\ee
Applying the tangential differential operator $\t^2\partial_t$ 
to the equation~(\ref{eq:veq}), multiplying it by $\t^2\partial_tv^i$ and
integrating over $\Omega$, we obtain in a completely analogous fashion
identity~(\ref{eq:en34}) claimed in Lemma~\ref{lm:i1}
with error terms $\mathcal{R}_3$ and $\mathcal{R}_4$ given by:
\be\label{eq:r3}
\begin{array}{l}
\displaystyle
\mathcal{R}_3:=
\sum_{1\leq m+n\leq 2}c_{mn}\t^m\partial_t^n\A^k_i\t^{2-m}\partial_t^{1-n}q,_k\t^2\partial_tv^i
-\big(\sum_{1\leq m+n\leq2}c_{mn}\t^m\partial_t^n\A^k_i\t^{2-m}\partial_t^{1-n}v^i,_k\\
\displaystyle
\qquad+\sum_{0\leq m+n\leq 2}d_{mn}\t^m\partial_t^n\Psi_{\kappa t}\cdot\t^{2-m}\partial_t^{1-n}v-
\t^2\Psi_{\kappa t}\cdot v_t\big)
\times\big(\t^2q_t+\t^2\Psi_{\kappa t}\cdot v\big);
\end{array}
\ee
\be\label{eq:r4}
\begin{array}{l}
\displaystyle
\mathcal{R}_4:=
-\Lambda_{\kappa}\t^2\Psi_t
\Big[\Lambda_{\kappa}\big[q,_2\big(\t^2\g_t\Psi_{t}\cdot \A^2_{\bullet}\big)\A^2_{\bullet}\big]
-q,_2\big(\t^2\g_t\Lambda_{\kappa}\Psi_t\cdot \A^2_{\bullet}\big)\A^2_{\bullet}\Big]\\
\displaystyle
\quad+\sum_{l+l'\geq1}a_{l,l'}\int_{\Gamma}q,_2\t^l\g_t^{l'}\Psi_{\kappa}\cdot \A^2_{\bullet}\t^{2-l}\g_t^{1-l'}(v+w)
\cdot\g_t^l \A^2_{\bullet}\\
\displaystyle
\quad-q,_1\t^2\g_{t}\Psi_{\kappa}^{r}\A^1_r\t^2\g_{t}v\cdot\A^2_{\bullet}
+\kappa^2\sum_{0\leq l+l'<3}\t^2\g_th_t\t^l\g_t^{l'}v\cdot\t^{2-l}\g_t^{1-l'}\A^2_{\bullet}
-\kappa^2\t^2\g_t\beta\t^2\g_tv\cdot\A^2_{\bullet}.
\end{array}
\ee
Finally, 
applying $\partial_{tt}$ to the equation~(\ref{eq:veq}), multiplying it by $\partial_{tt}v^i$ and
integrating over $\Omega$, the last identity~(\ref{eq:en56}) of Lemma~\ref{lm:i1} follows with error 
terms $\mathcal{R}_5$ and $\mathcal{R}_6$ given by
\be\label{eq:r5}
\begin{array}{l}
\displaystyle
\mathcal{R}_5:=
-\big(\A^k_i,_{tt}v^i,_k+2\A^k_i,_tv^i,_{kt}+2v_t\cdot \w_t+v_{tt}\cdot \w-\Psi_{\kappa tt}\cdot v_t\big)
(q_{tt}+\Psi_{\kappa tt}\cdot v);
\end{array}
\ee
\be\label{eq:r6}
\begin{array}{l}
\displaystyle
\mathcal{R}_6:=
-\Lambda_{\kappa}\Psi_{tt}
\Big[\Lambda_{\kappa}\big[q,_2\big(\Psi_{ttt}\cdot \A^2_{\bullet}\big)\A^2_{\bullet}\big]
-q,_2\big(\g_{tt}\Lambda_{\kappa}\Psi_t\cdot \A^2_{\bullet}\big)\A^2_{\bullet}\Big]\\
\displaystyle
\quad+\sum_{l=1}^2a^2_l\int_{\Gamma}q,_2\g_t^2\Psi_{\kappa}\cdot \A^2_{\bullet}\g_t^{2-l}(v+w)\cdot\g_t^l \A^2_{\bullet}
-q,_1\g_{tt}\Psi_{\kappa}^{r}\A^1_r\g_{tt}v\cdot\A^2_{\bullet}
+\kappa^2\sum_{l'=0}^1\g_{tt}h_t\g_t^{l'}v\g_t^{2-l'}\A^2_{\bullet}\\
\displaystyle
\quad-\kappa^2\g_{tt}\beta\g_{tt}v\cdot\A^2_{\bullet}.
\end{array}
\ee
{\em Proof of part (ii) of Lemma~\ref{lm:i1}.}
\par
Applying the tangential operator $\t^3\partial_t$ to the equation~(\ref{eq:veq}), multiplying by $\t^3v^i$
and integrating over $\Omega$ we obtain
\[
\big(\t^3\partial_tv^i,\,\t^3v^i\big)_{L^2}+\big(\t^3\partial_t(\A^k_iq,_k),\,\t^3v^i\big)_{L^2}=0,
\]
implying
\be\label{eq:B1}
\frac{1}{2}\partial_t\int_{\Omega}|\t^3v^i|^2+\big(\t^3\partial_t\A^k_iq,_k+\A^k_i\t^3\partial_tq,_k,\,\t^3v^i\big)_{L^2}
=\sum_{l+\bar{l}=3,\, k+\bar{k}=1\atop 0<l+k<4} c_{l,k,\bar{l},\bar{k}}
\big(\t^l\partial_t^k\A^k_i\t^{\bar{l}}\partial_t^{\bar{k}}q,_k,\,\t^3v^i\big)_{L^2}.
\ee
Recalling~(\ref{eq:comm}), we write
\be\label{eq:B2}
\t^3\partial_t\A^k_i=-\A^k_r\t^3\w,_s^r\A^s_i+\{\t^3\partial_t,\,\A^k_i\}.
\ee
Using this decomposition we have
\be\label{eq:B3}
\big(\t^3\partial_t\A^k_iq,_k+\A^k_i\t^3\partial_tq,_k,\,\t^3v^i\big)_{L^2}
=-\int_{\Omega}\A^k_r\t^3\w,_s^r\A^s_iq,_k\t^3v^i+\int_{\Omega}\A^k_i\t^3\partial_tq,_k\t^3v^i
+\mathcal{T},
\ee
where the commutator term has been absorbed in the error $\mathcal{T}$.
Integrating by parts with respect to $s$ and $k$ in the first two integrals on the right-hand side above respectively, we obtain
analogously to the proof of Lemma~\ref{lm:i1}:
\be\label{eq:B4}
\begin{array}{l}
\displaystyle
-\int_{\Omega}\A^k_r\t^3\w,^r_s\A^s_iq,_k\t^3v^i+\int_{\Omega}\A^k_i\t^3\partial_tq,_k\t^3v^i\\
\displaystyle
=\int_{\Gamma}q,_2\t^3\w\cdot \A^2_{\bullet}\t^3w\cdot\A^2_{\bullet}
-\int_{\Gamma}\t^3\g_tq\t^3v\cdot\A^2_{\bullet}
+\int_{\Gamma}q,_1\A^1_r\t^3\w^r\t^3v\cdot\A^2_{\bullet}\\
\displaystyle
+\int_{\Omega}(\t^3q_t+\t^3\w\cdot v)^2
+\int_{\Omega}\big(\sum_{l=1}^3d_l\t^{3-l}\w\cdot\t^lv+\sum_{l=1}^3e_l\t^lA^s_i\t^{3-l}v^i,_s\big)\big(\t^3q_t+\t^3\w\cdot v\big)
+\mathcal{T}.
\end{array}
\ee
Note further that the first term on the right-hand side above can be, similarly to~(\ref{eq:iden4.1}), further written as
\beas
\int_{\Gamma}q,_2\t^3\w\cdot \A^2_{\bullet}\t^3w\cdot\A^2_{\bullet}
&=&\int_{\Gamma}q,_2|\t^3\Lambda_{\kappa}w\cdot \A^2_{\bullet}|^2
+\int_{\Gamma}\Lambda_{\kappa}\t^3w\cdot\big[\Lambda_{\kappa}[q,_2(\t^3w\cdot \A^2_{\bullet})\A^2_{\bullet}]\\
&&-q,_2(\t^3\Lambda_{\kappa}w\cdot \A^2_{\bullet})\A^2_{\bullet}\big].
\eeas
The second term on the right-hand side of~(\ref{eq:B4}) reads, using the boundary condition~(\ref{eq:bdryreg})
\beas
\int_{\Gamma}\t^3\g_tq\t^3v\cdot\A^2_{\bullet}
&=&\kappa^2\int_{\Gamma}\t^3\g_th_t\big(\t^3(v\cdot\A^2_{\bullet})-\sum_{l=0}^2c_l\t^lv\cdot\t^{3-l}\A^2_{\bullet}\big)
+\kappa^2\int_{\Gamma}\t^3\g_t\beta\t^3v\cdot\A^2_{\bullet}\\
&=&\frac{\kappa^2}{2}\frac{d}{dt}\int_{\Gamma}J_{\kappa}^{-1}|\t^3h_t|^2-\kappa^2\sum_{l=0}^2c_l\int_{\Gamma}\t^3\g_th_t\t^lv\cdot\t^{3-l}\A^2_{\bullet}
+\kappa^2\int_{\Gamma}\t^3\g_t\beta\t^3v\cdot\A^2_{\bullet}+\mathcal{T},
\eeas
where the error term $\mathcal{T}$ denotes the lower order terms containing the time derivative of $J_{\kappa}$.
We also used the regularized boundary condition~(\ref{eq:bdryreg}) in the first equality above.
Combining~(\ref{eq:B1})--(\ref{eq:B4}) and the last identity we obtain the identity~(\ref{eq:s12})
with error terms $\mathcal{S}_1$ and $\mathcal{S}_2$ given by
\begin{align}
\mathcal{S}_1:= &
\sum_{l+\bar{l}=3,\, k+\bar{k}=1\atop 0<l+k<4} c_{l,k,\bar{l},\bar{k}}
\t^l\partial_t^k\A^k_i\t^{\bar{l}}\partial_t^{\bar{k}}q,_k \notag \\
&-\big(\sum_{l=1}^3d_l\t^{3-l}\w\cdot\t^lv+\sum_{l=1}^3e_l\t^l\A^s_i\t^{3-l}v^i,_s\big)\big(\t^3q_t+\t^3\w\cdot v\big); \label{eq:sremainder}
\end{align}
\begin{align}
\mathcal{S}_2:=&
-\Lambda_{\kappa}\t^3w\cdot\big[\Lambda_{\kappa}[q,_2(\t^3w\cdot \A^2_{\bullet})\A^2_{\bullet}]
-q,_2(\t^3\Lambda_{\kappa}w\cdot \A^2_{\bullet})\A^2_{\bullet}\big]\notag \\
&+\sum_{l=1}^3c_l(-q,_2)\t^3\w\cdot \A^2_{\bullet}(\t^{3-l}(v+w)\cdot\t^l\A^2_{\bullet})
-q,_1\A^1_r\t^3\w^r\t^3v\cdot\A^2_{\bullet} \notag \\
&+\kappa^2\sum_{l=0}^2c_l\t^3\g_th_t\t^lv\cdot\t^{3-l}\A^2_{\bullet}
-\kappa^2\t^3\g_t\beta\t^3v\cdot\A^2_{\bullet}.\label{eq:sremaindergamma}
\end{align}
Applying the tangential operator $\t\partial_t^2$ to the equation~(\ref{eq:veq}), multiplying by $\t\partial_tv^i$
and integrating over $\Omega$ we obtain the identity~(\ref{eq:s34}) in an analogous way, with 
error terms $\mathcal{S}_3$ and $\mathcal{S}_4$ given by
\be\label{eq:s3}
\begin{array}{l}
\displaystyle
\mathcal{S}_3:=\big(v,_s\cdot\t\Psi_{\kappa tt}\A^s_i-\{\t\partial_t^2,\A^k_i\}q,_k\big)
\t v^i_t
+\sum_{1\leq m+n\leq2}c_{mn}\t^a\partial_t^b\A^k_i\t^{1-a}\partial_t^{2-b}q,_k\t v^i_{tt}\\
\displaystyle
+\big(\sum_{1\leq m+n\leq2}d_{mn}\t^a\partial_t^b\A^s_i\t^{1-a}\partial_t^{1-b}v^i,_s
-(\Psi_{tt}\cdot\t v+\t\Psi_t\cdot v_t+\Psi_t\t v_t)\big)\big(\t q_{tt}+\t\Psi_{\kappa,tt}\cdot v\big);
\end{array}
\ee
\begin{align}
\mathcal{S}_4:=&q,_2\t \w_t\cdot \A^2_{\bullet}\big[(\t(v+\w)\cdot \A^2_{\bullet t})
+(\w_t+v_t)\cdot\t\A^2_{\bullet}+(\w+v)\cdot\t \A^2_{\bullet t}\big] \notag \\
&-\t\g_t\Lambda_{\kappa}w\cdot\A^2_{\bullet}
\Big[\Lambda_{\kappa}\big((-q,_2)A^2_{\bullet}\t\g_tw\cdot\A^2_{\bullet}\big)
-q,_2\t\g_t\Lambda_{\kappa}w\A^2_{\bullet}\Big]\notag \\
&-q,_1\A^1_r\t\g_t\w^r\t\g_tv\cdot\A^2_{\bullet}+\kappa^2\t\g_th_t(\t\g_t(v\cdot\A^2_{\bullet})-\t\g_tv\cdot\A^2_{\bullet})
-\kappa^2\t\g_{tt}\beta\t\g_tv\cdot\A^2_{\bullet}. \label{eq:s4}
\end{align}

\subsection{Nonlinear energy estimates}\label{se:est1}
The following proposition states the desired energy bound for the classical Stefan problem (with $\sigma=0$), 
will subsequently lead to a uniform-in-$\kappa$ time of existence
 for our family solutions  to the regularized $ \kappa$-problems~(\ref{eq:reg}).
\begin{prop}[Main energy inequality]\label{pr:basic1}
There exists a constant $C$ independent of $\kappa$ and a generic polynomial function $P$  
such that for any $t\in [0,T^{\kappa}]$ we have the following bound:
\be\label{E:POLYNOMIALBOUND}
\E_{\kappa}(t)\leq C\E_{\kappa}(0)+C(t+\sqrt{t})P(\E_{\kappa}).
\ee
\end{prop}
%
%
The proof of the proposition proceeds by systematically estimating error terms in the energy identities 
from Section~\ref{su:identities}. We shall implicitly use the a priori bound~\eqref{E:APRIORI} freely throughout the proof
without explicitly making a reference to it.

\subsubsection*{ Step 1. Estimates for $\int_0^t\int_{\Omega}\mathcal{R}_1$ defined by~\eqref{eq:remainder}} 
%
%

\medskip

\noindent
We start by estimating the integral $\sum_{l=1}^3\int_0^t\int_{\Omega}c_l\t^{l}\A^k_i\t^{4-l}q_{,k}\t^4v^i$ (the first term appearing in~\eqref{eq:remainder}.)
If $l=1$, we have
\beas
\Big|\int_0^t\int_{\Omega}\t \A^k_i\t^3q_{,k}\t^4v^i\Big|
&\leq&\|\t \A^k_i\|_{L^{\infty}_tL^{\infty}_x}\int_0^t\|\t^3q,_k\|_{L^2}\|\t^4v^i\|_{L^2}\\
&\lesssim&\|\int_0^t\t\partial_t(\A^k_i)\|_{H^{1.5}}\|\t^3q,_k\|_{L^2_tL^2_x}\|\t^4v^i\|_{L^2_tL^2_x}\\
&\leq&\sqrt{t}\|\w\|_{L^2_tH^{3.5}_x}\|\t^3q,_k\|_{L^2_tL^2_x}\|\t^4v^i\|_{L^2_tL^2_x}\lesssim \sqrt{t}P(\E_{\kappa}).
\eeas
For $l=2,3$ we have
\beas
\Big|\int_0^t\int_{\Omega}\t^{l}\A^k_i\t^{4-l}q_{,k}\t^4v^i\Big|
&\lesssim&\|\t^l\A^k_i\|_{L^{\infty}_tH^{0.5}_x}\|\t^{4-l}q,_k\|_{H^{0.5}}\int_0^t\|\t^4v^i\|_{L^2}\\
&\lesssim&\sqrt{t}\|\nabla(\Psi_{\kappa}-\text{Id})\|_{L^{\infty}_tH^{2.5}_x}\|q\|_{L^{\infty}_tH^{3.5}_x}\|\t^4v^i\|_{L^2_tL^2_x}
\lesssim \sqrt{t}P(\E_{\kappa}).
\eeas
We proceed to estimate the integral $\sum_{l=1}^4c_l\int_0^t\int_{\Omega}\t^l\A^k_i\t^{4-l}v^i,_k\big(\t^4q+\t^4\Psi_{\kappa}\cdot v\big)$ (the second term appearing in~\eqref{eq:remainder}).
Only cases $l=1$ and $l=4$ deserve special attention, while the cases $l=2$ and $l=3$ are estimated
by a routine application of the Cauchy-Schwarz inequality and the Sobolev embedding.
When $l=1,$ we can use Lemma~\ref{L:TECHNICAL}
to conclude that 
\[
\begin{array}{l}
\displaystyle
\Big|\int_0^t\int_{\Omega}\t \A^k_i\t^3v^i,_k\big(\t^4q+\t^4\Psi_{\kappa}\cdot v\big)\Big|
\leq \int_0^t\|\t\A^k_i\|_{0.5}\|v^i,_k\|_{2.5}\|\t^4q+\t^4\Psi_{\kappa}\cdot v\|_{0}\\
\displaystyle
+\|\t \A^k_i\|_{L^{\infty}_tL^{\infty}_x}\int_0^t\|v^i,_k\|_{2.5}\big(\|\t^4q\|_{0.5}
+\|(\t^4\Psi_{\kappa}\cdot v)\|_{0.5}\big)\\
\displaystyle
\lesssim\|\int_0^t\nabla^2\partial_t(\A^k_i)\|_{L^{\infty}_tL^2_x}\int_0^t\|q\|_{4.5}
\big(\|\t^4q\|_{0.5}+\|\t^4\Psi_{\kappa}\cdot v\|_{0.5}\big)
+\|\int_0^t\t\partial_t(\A^k_i)\|_{H^{1.5}}\int_0^t\|q\|_{4.5}^2\\
\displaystyle
+\|\nabla^2\Psi_{\kappa}\|_{L^{\infty}_tH^{1.5}_x}\|v\|_{L^{\infty}_tH^{0.5}_x}\|\nabla^2\Psi_{\kappa}\|_{L^{\infty}_tH^{2.5}_x}\int_0^t\|q\|_{4.5}\\
\displaystyle
\lesssim\sqrt{t}\|\w\|_{L^2_tH^3_x}\|q\|_{L^2_tH^{4.5}_x}^2+\sqrt{t}\|\w\|_{L^2_tH^3_x}
\|\nabla^2\Psi_{\kappa}\|_{L^{\infty}_tH^{2.5}_x}\int_0^t\|q\|_{4.5}
+C\sqrt{t}\|\w\|_{L^2_tH^{3.5}_x}\int_0^t\|q\|_{4.5}^2\\
\displaystyle
+C\sqrt{t}\|\nabla^2\Psi_{\kappa}\|_{L^{\infty}_tH^{1.5}_x}\|\nabla^2\Psi_{\kappa}\|_{L^{\infty}_tH^{2.5}_x}\|q\|_{L^2_tH^{4.5}_x}
\lesssim \sqrt{t}P(\E_{\kappa}).
\end{array}
\]
As for the case $l=4$, we use Lemma~\ref{L:TECHNICAL} again 
and obtain 
\begin{align*}
& \Big|\int_0^t\int_{\Omega}\t^4\A^k_iv^i,_k\big(\t^4q+\t^4\Psi_{\kappa}\cdot v\big)\Big|
\leq\int_0^t\int_{\Omega}\|\A^k_i\|_{3.5}
\|v^i,_k\big(\t^4q+\t^4\Psi_{\kappa}\cdot v\big)\|_{0.5}\\
& \leq\int_0^t\|\nabla^2\Psi_{\kappa}\|_{2.5}\|v^i,_k\|_{W^{0.5,\infty}}\|\t^4q+\t^4\Psi_{\kappa}\cdot v\|_{0.5} \\
& \leq\|\nabla^2\Psi_{\kappa}\|_{L^{\infty}_tH^{2.5}_x}\|q\|_{L^{\infty}_tH^{3.5}_x}\int_0^t(\|q\|_{4.5}+\|\nabla^2\Psi_{\kappa}\|_{2.5})\\
&\leq\sqrt{t}\|\nabla^2\Psi_{\kappa}\|_{L^{\infty}_tH^{2.5}_x}\|q\|_{L^{\infty}_tH^{3.5}_x}\|q\|_{L^2_tH^{4.5}_x}
+t\|q\|_{L^{\infty}_tH^{3.5}_x}\|\nabla^2\Psi_{\kappa}\|_{L^{\infty}_tH^{2.5}_x}^2
\lesssim  (\sqrt{t}+t)P(\E_{\kappa}).
\end{align*}
%
The next error term to estimate is $\sum_{l=1}^4d_l\int_0^t\int_{\Omega}\t^{4-l}\w\cdot\t^lv\big(\t^4q+\t^4\Psi_{\kappa}\cdot v\big)$ (the third term appearing in~\eqref{eq:remainder}).
If $l=4$, we estimate
\begin{align*}
& \Big|\int_0^t\int_{\Omega}\w\cdot\t^4v\big(\t^4q+\t^4\Psi_{\kappa}\cdot v\big)\Big|
\leq\int_0^t\|\w\|_{\infty}\|\t^4v\|_0\|\t^4q+\t^4\Psi_{\kappa}\cdot v\|_0\\
&\leq\sqrt{t}\|\w\|_{L^{\infty}_tH^{1.5}_x}\|\t^4q+\t^4\Psi_{\kappa}\cdot v\|_{L^{\infty}_tL^2_x}\|\t^4v\|_{L^2_tL^2_x}
\lesssim  \sqrt{t}P(\E_{\kappa});
\end{align*}
and analogously for $l=3$
\[
\Big|\int_0^t\int_{\Omega}\t \w\cdot\t^3v\big(\t^4q+\t^4\Psi_{\kappa}\cdot v\big)\Big|
\leq\sqrt{t}\|\w\|_{L^{\infty}_tH^{2.5}_x}\|\t^4q+\t^4\Psi_{\kappa}\cdot v\|_{L^{\infty}_tL^2_x}\|\t^3v\|_{L^2_tL^2_x}
\lesssim  \sqrt{t}P(\E_{\kappa}).
\]
For $l=1,2$, we have
\beas
\Big|\int_0^t\int_{\Omega}\t^l\w\cdot\t^{4-l}v\big(\t^4q+\t^4\Psi_{\kappa}\cdot v\big)\Big|
&\leq&\int_0^t\|\t^l\w\|_0\|\t^{4-l}v\|_{\infty}\|\t^4q+\t^4\Psi_{\kappa}\cdot v\|_0\\
&\leq&\sqrt{t}\|\w\|_{L^{\infty}_tH^2_x}\|\t^4q+\t^4\Psi_{\kappa}\cdot v\|_{L^{\infty}_tL^2_x}\|q\|_{L^2_tH^{4.5}_x}
\lesssim  \sqrt{t}P(\E_{\kappa}).
\eeas
The next-to-last term on the right-hand side of~(\ref{eq:remainder}) is estimated as follows:
\[
\Big|\t^4\Psi_{\kappa}\cdot v_t\big(\t^4q+\t^4\Psi_{\kappa}\cdot v\big)\Big|
\leq\sqrt{t}\|\t^4\Psi_{\kappa}\|_{L^{\infty}_tL^2_x}\|\t^4q+\t^4\Psi_{\kappa}\cdot v\|_{L^{\infty}_tL^2_x}
\|v_t\|_{L^2_tH^{1.5}_x}\lesssim \sqrt{t}P(\E_{\kappa}).
\]
\noindent
Finally, to bound $\int_0^t\int_\Omega \t^4\alpha \big(\t^4q+\t^4\Psi_{\kappa}\cdot v\big)$ (the last term appearing in~\eqref{eq:remainder})
we Integrate by parts and use the Cauchy-Schwarz inequality to obtain
\begin{align*}
\Big|\int_0^t\int_\Omega \t^4\alpha \big(\t^4q+\t^4\Psi_{\kappa}\cdot v\big)\Big| \le \|\t^3\alpha\|_{L^\infty_tL^2_x} \sqrt t \|\t^5q + \t(\t^4\Psi_\kappa\cdot v)\|_{L^2_tL^2_x} 
\lesssim \E_\kappa(0) + t \E_\kappa(t).
\end{align*}
We used Lemma~\ref{optimal_q} and a priori bound~\eqref{E:APRIORI}.
\subsubsection*{Estimates for the error term $\int_0^t\int_{\Gamma}\mathcal{R}_2$ defined by~\eqref{eq:remaindergamma}.}
%
%
%
For any $i,j\in\{1,2\}$ set
$F=q,_2\A^2_i\A^2_j$, $G=\t^3\Psi^i_t$ and apply Lemma~\ref{lm:comm}
to conclude
\[
\begin{array}{l}
\displaystyle
\int_0^t\Big|\Lambda_{\kappa}\big[q,_2\big(\t^4\Psi_t\cdot \A^2_{\bullet}\big)\A^2_{\bullet}\big]
-q,_2\big(\t^4\Lambda_{\kappa}\Psi_t\cdot \A^2_{\bullet}\big)\A^2_{\bullet}\Big|^2
\lesssim\int_0^t|q,_2\A^2_{\bullet}:\A^2_{\bullet}|_{W^{1,\infty}(\Gamma)}^2|\Lambda_\kappa w|_3^2\\
\displaystyle
\quad\lesssim\sup_{0\leq s\leq t}|q,_2\A^2_{\bullet}:\A^2_{\bullet}|_2^2\int_0^t|\Lambda_\kappa w|_3^2\lesssim P(\E_{\kappa}),
\end{array}
\]
where we  
estimate $|\Lambda_\kappa w|_{L^2_tH^3_x}$ using~(\ref{E:HEVOLUTION}):
\be\label{eq:ht}
|h_t|_{L^2_tH^3_x}^2\lesssim\int_0^t\big|\t^3(\sqrt{1+|\t h_{\kappa}|^2}(v\cdot \A^2_{\bullet}))\big|^2
\lesssim P(\E_{\kappa})\int_0^t\|q\|_{4.5}^2\lesssim P(\E_{\kappa}).
\ee 
Note that we bounded $|v|_{3}$ by relating it to its norm over $\Omega$ via the trace estimate
\be\label{eq:htaux}
|v|_{L^2_tH^3_x}^2\lesssim\|v\|_{L^2_tH^{3.5}_x}^2\lesssim P(\E_{\kappa})\int_0^t\|q\|_{4.5}^2\lesssim P(\E_{\kappa}).
\ee
Thus,
\[
\begin{array}{l}
\displaystyle
\Big|\int_0^t\int_{\Gamma}\Lambda_{\kappa}\t^4\Psi
\Big[\Lambda_{\kappa}\big[q,_2\big(\t^4\Psi_t\cdot \A^2_{\bullet}\big)\A^2_{\bullet}\big]
-q,_2\big(\t^4\Lambda_{\kappa}\Psi_t\cdot \A^2_{\bullet}\big)\A^2_{\bullet}\Big]\Big|\\
\displaystyle
\qquad
\lesssim P(\E_{\kappa})^{1/2}\big(\int_0^t\int_{\Gamma}\big|\Lambda_{\kappa}\t^4\Psi\big|^2\big)^{1/2}
\lesssim tP(\E_{\kappa}).
\end{array}
\]
Finally, we treat the last term on the right-hand side of~(\ref{eq:remaindergamma}).
For $1\leq l\leq2$, we have
\[
\begin{array}{l}
\displaystyle 
\int_0^t\int_{\Gamma}q,_2\t^4\Psi_{\kappa}\cdot \A^2_{\bullet}\t^{4-l}(v+\w)\cdot\t^l \A^2_{\bullet}
\leq|q,_2\t^l \A^2_{\bullet}|_{L^{\infty}_tL^{\infty}_x}\int_0^t|\t^4\Psi_{\kappa}\cdot \A^2_{\bullet}|_0|\t^{4-l}(v+\w)|_0\\
\displaystyle
\lesssim|q,_2|_{L^{\infty}_tH^1_x}|\t^l \A^2_{\bullet}|_{L^{\infty}_tH^1_x}
|\t^4\Psi_{\kappa}\cdot \A^2_{\bullet}|_{L^{\infty}_tL^2_x}\sqrt{t}\big(|v|_{L^2_tH^3_x}+|\w|_{L^2_tH^3_x}\big)
\lesssim \sqrt{t}P(\E_{\kappa}),
\end{array}
\]
where estimates~(\ref{eq:ht}) and~(\ref{eq:htaux}) were used in the last inequality.
If $l=3$, we apply a similar estimate, bounding the term $\t^l\A^2_{\bullet}=\t^3\A^2_{\bullet}$ in $L^2$-norm and 
$\t^{4-l}(v+\w)=\t(v+\w)$ via 
$L^{\infty}$ norm and Sobolev embedding leading to:
\[
\Big|\int_0^t\int_{\Gamma}q,_2\t^4\Psi_{\kappa}\cdot \A^2_{\bullet}\t(v+\w)\cdot\t^3 \A^2_{\bullet}\Big|
\lesssim \sqrt{t}P(\E_{\kappa})
\]
Case $l=4$ is the trickiest error term as four derivatives fall on $\A^2_{\bullet}$, thus creating a term 
that at highest order contains five derivatives of $\Psi$, which is more than the number of derivatives allowed by our energy $\E_\kappa.$ 
However, we have the following identity:
%
\be\label{eq:id}
\begin{array}{l}
\displaystyle
\int_{\Gamma}q,_2\t^4\Psi_{\kappa}\cdot \A^2_{\bullet}(v+\w)\cdot\t^4\A^2_{\bullet}
=\frac{1}{2}\int_{\Gamma}\t[q,_2\frac{\w\cdot\tau_{\kappa}}{|\t\tau_{\kappa}|}]
\big(\t^4\Psi_{\kappa}\cdot \A^2_{\bullet}\big)^2\\
\displaystyle
\qquad
+\int_{\Gamma}q,_2\frac{\w\cdot\tau_{\kappa}}{|\t\tau_{\kappa}|}
\t^4\Psi_{\kappa}\cdot \A^2_{\bullet}\t^4\Psi_{\kappa}\cdot \t \A^2_{\bullet}
+\int_{\Gamma}q,_2\t^4\Psi_{\kappa}\cdot \A^2_{\bullet}E,
\end{array}
\ee
where $E$ is the lower order error term given by 
\be\label{eq:E}
E=\w\cdot\tau_{\kappa}\sum_{l=1}^3e_l\t^{l+1}\Psi_{\kappa}\t^{4-l}(|\t\Psi_{\kappa}|^{-1})\cdot \A^2_{\bullet}
+\sum_{l=1}^3e_l\w\cdot\tau_{\kappa}\t^l\tau_{\kappa}\t^{4-l}\A^2_{\bullet}.
\ee
%
%
To prove~\eqref{eq:id} we first note that
\[
v+\w=(v+\w)\cdot n_{\kappa}n_{\kappa}+(v+\w)\cdot\tau_{\kappa}\tau_{\kappa}
=(v+\w)\cdot\tau_{\kappa}\tau_{\kappa}
\]
where we have used the boundary condition~(\ref{eq:evolreg}).
Therefore, we have the equality
\beas
(v+\w)\cdot\t^4\A^2_{\bullet}&=&(v+\w)\cdot\tau_{\kappa}\tau_{\kappa}\cdot\t^4\A^2_{\bullet}
=\w\cdot\tau_{\kappa}\tau_{\kappa}\cdot\t^4\A^2_{\bullet}\\
&=&-\w\cdot\tau_{\kappa}\t^4\tau_{\kappa}\cdot \A^2_{\bullet}
+\sum_{l=1}^3e_l\w\cdot\tau_{\kappa}\t^l\tau_{\kappa}\t^{4-l}\A^2_{\bullet},
\eeas
where we first used the identity $v\cdot\tau_{\kappa}=0$ and in the last line
we used the product rule expansion of the the identity $0=\t^4(\tau_{\kappa}\cdot \A^2_{\bullet})$ with $e_l$ the corresponding binomial coefficients.
Since $\tau_{\kappa}=\frac{\t\Psi_{\kappa}}{|\t\Psi_{\kappa}|}$, we have
\beas
\t^4\tau_{\kappa}\cdot \A^2_{\bullet}&=&\frac{\t^5\Psi_{\kappa}}{|\t\Psi_{\kappa}|}\cdot \A^2_{\bullet}
+\sum_{l=1}^3e_l\t^{l+1}\Psi_{\kappa}\t^{4-l}(|\t \Psi_{\kappa}|^{-1})\cdot \A^2_{\bullet}+\t\Psi_{\kappa}\t^4
(|\t\Psi_{\kappa}|^{-1})\cdot \A^2_{\bullet}\\
&=&\frac{\t^5\Psi_{\kappa}}{|\t\Psi_{\kappa}|}\cdot \A^2_{\bullet}
+\sum_{l=1}^3e_l\t^{l+1}\Psi_{\kappa}\t^{4-l}(|\t\Psi_{\kappa}|^{-1})\cdot \A^2_{\bullet},
\eeas
where we simply used the product rule to expand $\t^4(\frac{\t\Psi}{|\t\Psi|})$ and the orthogonality
of $\t\Psi_{\kappa}$ and $\A^2_{\bullet}$ in the last line.
Combining the previous two identities, we may write
\[
(v+\w)\cdot\t^4\A^2_{\bullet}=-\frac{\t^5\Psi_{\kappa}}{|\t\Psi_{\kappa}|}\cdot \A^2_{\bullet}\w\cdot\tau_{\kappa}+E,
\]
where the error term $E$ is given by~\eqref{eq:E}.
We thus obtain
\[
\int_{\Gamma}q,_2\t^4\Psi_{\kappa}\cdot \A^2_{\bullet}(v+\w)\cdot\t^4\A^2_{\bullet}
=-\int_{\Gamma}q,_2\frac{\w\cdot\tau_{\kappa}}{|\t\Psi_{\kappa}|}
\t^4\Psi_{\kappa}\cdot \A^2_{\bullet}\t^5\Psi_{\kappa}\cdot \A^2_{\bullet}
+\int_{\Gamma}q,_2\t^4\Psi_{\kappa}\cdot \A^2_{\bullet}\,E.
\]
Note that first integral on the right-hand side has a symmetry allowing us to extract a full tangential derivative
at the level of highest order terms: 
\[
\begin{array}{l}
\displaystyle
-\int_{\Gamma}q,_2\frac{\w\cdot\tau_{\kappa}}{|\t\tau_{\kappa}|}
\t^4\Psi_{\kappa}\cdot \A^2_{\bullet}\t^5\Psi_{\kappa}\cdot \A^2_{\bullet}\\
\displaystyle
\quad=-\frac{1}{2}\int_{\Gamma}q,_2\frac{\w\cdot\tau_{\kappa}}{|\t\tau_{\kappa}|}
\t[\big(\t^4\Psi_{\kappa}\cdot \A^2_{\bullet}\big)^2]
+\int_{\Gamma}q,_2\frac{\w\cdot\tau_{\kappa}}{|\t\tau_{\kappa}|}
\t^4\Psi_{\kappa}\cdot \A^2_{\bullet}\t^4\Psi_{\kappa}\cdot \t \A^2_{\bullet}\\
\displaystyle
\quad=\frac{1}{2}\int_{\Gamma}\t[q,_2\frac{\w\cdot\tau_{\kappa}}{|\t\tau_{\kappa}|}]
\big(\t^4\Psi_{\kappa}\cdot \A^2_{\bullet}\big)^2
+\int_{\Gamma}q,_2\frac{\w\cdot\tau_{\kappa}}{|\t\tau_{\kappa}|}
\t^4\Psi_{\kappa}\cdot \A^2_{\bullet}\t^4\Psi_{\kappa}\cdot \t \A^2_{\bullet},
\end{array}
\]
where we have used integration by parts in the second equation. Finally, summing the previous two identities
we arrive at~(\ref{eq:id}). 

Note that $\Psi_\kappa$ enters the right-hand side of the above identity at most with $4$ derivatives.
By standard $L^{\infty}-L^2-L^2$ type estimates and identity~\eqref{eq:id}, we finally arrive at 
\be
\Big|\int_0^t\int_{\Gamma}q,_2\t^4\Psi_{\kappa}\cdot \A^2_{\bullet}(v+\w)\cdot\t^4\A^2_{\bullet}\Big|
\lesssim\sqrt{t}P(\E_{\kappa}).
\ee
Before we estimate the third term on the right-hand side of~(\ref{eq:remaindergamma}), we first rewrite:
\[
\t^4v\cdot\A^2_{\bullet}=-\t^4(\w\cdot\A^2_{\bullet})-\sum_{l=0}^3a_l\t^lv\t^{4-l}\A^2_{\bullet}
=-J_\kappa^{-1}\t^4h_{\kappa,t}-\sum_{l=0}^3a_l\t^lv\t^{4-l}\A^2_{\bullet},
\]
where $a_l,$ $l=0,\dots,3$ are the corresponding binomial coefficients.
As a consequence, we have
\be\label{eq:kappa1}
\begin{array}{l}
\displaystyle
\Big|\int_0^t\int_{\Gamma}q,_1\t^4\Psi_{\kappa}^{r}\A^1_r\t^4v\cdot\A^2_{\bullet}\Big|
\leq\kappa^2\Big|\int_0^t\int_{\Gamma}(v\cdot\a^2_{\bullet}+\beta),_1\t^4\Psi_{\kappa}^{r}\A^1_rJ_{\kappa}^{-1}\t^4h_t\Big|\\
\displaystyle
\quad+\kappa^2\sum_{l=0}^3a_l\Big|\int_0^t\int_{\Gamma}(v\cdot\a^2_{\bullet}+\beta),_1\t^4\Psi_{\kappa}^{r}\A^1_r\t^lv\t^{4-l}\A^2_{\bullet}\Big|.
\end{array}
\ee
The first term on the right-hand side above is easily bounded as follows:
\beas
\kappa^2\Big|\int_0^t\int_{\Gamma}(v\cdot\a^2_{\bullet}+\beta),_1\t^4\Psi_{\kappa}^{r}\A^1_rJ_{\kappa}^{-1}\t^4h_t\Big|
&\leq&\sqrt{t}\kappa|(v\cdot\a^2_{\bullet}+\beta),_1\A^1_rJ_{\kappa}^{-1}|_{L^{\infty}_tL^{\infty}_x}|\t^4\Psi_{\kappa}|_{L^{\infty}_tL^2_x}
\kappa|\t^4h_t|_{L^2_tL^2_x}\\
&\lesssim&\sqrt{t}\kappa P(\E_{\kappa}).
\eeas
The second term on the right-hand side of~(\ref{eq:kappa1}) is a sum, and the hardest summand to bound is created when $l=0$.
In this case, roughly speaking we bound $|\t^4\A^2_{\bullet}|_0$ by $\kappa^{-1}|\Lambda_{\kappa}\Psi|_4$ trading one tangential derivative on 
$\t^4\Lambda_{\kappa}\Lambda_{\kappa}\nabla\Psi$ for a bound on $\Lambda_{\kappa}\nabla\Psi$ in $H^3$, at the expense of a factor of $\kappa^{-1}$.
Using this observation we obtain 
\begin{align*}
\kappa^2\Big|\int_0^t\int_{\Gamma}(v\cdot\a^2_{\bullet}+\beta),_1\t^4\Psi_{\kappa}^{r}\A^1_rv\t^4\A^2_{\bullet}\Big|
& \lesssim\kappa^2\sqrt{t}|(v\cdot\A^2_{\bullet}+\beta),_1\A^1_r|_{L^{\infty}_tL^{\infty}_x}|\kappa^{-1}|\t^4\Lambda_{\kappa}|_{L^{\infty}_tL^2_x} \\
& \lesssim\sqrt{t}\kappa P(\E_{\kappa}).
\end{align*}
The next-to-last term on the right-hand side of~(\ref{eq:remaindergamma}) is again a sum and the hardest
term to estimate is created again when $l=0$. We use the same idea as in the previous estimate to obtain
\[
\Big|\kappa^2\int_0^t\int_{\Gamma}\t^4h_tv\cdot\t^4\A^2_{\bullet}\Big|
\lesssim\sqrt{t}\kappa|\t^4h_t|_{L^2_tL^2_x}|v|_{L^{\infty}_tL^{\infty}_x}\kappa\kappa^{-1}|\Lambda_{\kappa}\Psi|_{L^{\infty}_tL^2_x}
\lesssim\sqrt{t}P(\E_{\kappa}).
\]
Note that we exploited the presence of the $\kappa$-dependent energy term
in our energy $\E_{\kappa}$, using the bound $\kappa|\t^4h_t|_{L^2_tL^2_x}\leq\sqrt{\E_{\kappa}}$.
In analogous manner, we conclude
\[
\Big|\int_{\Omega}\mathcal{R}_3+\mathcal{R}_5\,dx\Big|+\Big|\int_{\Gamma}\mathcal{R}_4+\mathcal{R}_6\,dx'\Big|
\lesssim(t+\sqrt{t})P(\E_{\kappa}),
\]
where we note that the commutator term, i.e. the first term on the right-hand side of~(\ref{eq:r6}) deserves special attention.
Due to the absence of spatial derivatives in the term $\Psi_{ttt}$ in 
\[
-\Lambda_{\kappa}\Psi_{tt}
\Big[\Lambda_{\kappa}\big[q,_2\big(\Psi_{ttt}\cdot \A^2_{\bullet}\big)\A^2_{\bullet}\big]
-q,_2\big(\g_{tt}\Lambda_{\kappa}\Psi_t\cdot \A^2_{\bullet}\big)\A^2_{\bullet}\Big]
\]
we cannot apply the commutator bound from Lemma~\ref{lm:comm} in the form stated.
Here we crucially exploit the $\kappa$-dependent term in the energy $\E_{\kappa}$. 
Note that
\be\label{eq:whyreg}
\begin{array}{l}
\displaystyle
\Big|\int_0^t\int_{\Gamma}\Lambda_{\kappa}\Psi_{tt}
\Big[\Lambda_{\kappa}\big[q,_2\big(\Psi_{ttt}\cdot \A^2_{\bullet}\big)\A^2_{\bullet}\big]
-q,_2\big(\g_{tt}\Lambda_{\kappa}\Psi_t\cdot \A^2_{\bullet}\big)\A^2_{\bullet}\Big]\Big|\\
\displaystyle
\quad\leq\sqrt{t}|\Lambda_{\kappa}\Psi_{tt}|_{L^{\infty}_tL^2_x}||q,_2\A^2_{\bullet}:\A^2_{\bullet}|_{L^{\infty}_tW^{1,\infty}_x}
\kappa|\Psi_{ttt}|_{L^2_tL^2_x}\\
\displaystyle
\quad
\leq\sqrt{t}P(\E_{\kappa}),
\end{array}
\ee  
where we gain one power of $\kappa$ in the second line above from the commutator estimate and
then absorb it into the energy contribution $\kappa|\Psi_{ttt}|_{L^2_tL^2_x}$.
The last term on the right-hand side of~(\ref{eq:remaindergamma}) contains the $\beta$-contribution from the 
regularized Dirichlet condition~(\ref{eq:bdryreg}). It is easily estimated using the Cauchy-Schwarz inequality
by a term of the form $Ctm_0+Ct\kappa^2|\t^4h_t|_{L^2L^2}^2$ which in turn is smaller than a constant multiple of
$tm_0+t\E_{\kappa}$. Here $m_0$ is a constant, which depends only on the initial data.

\subsubsection*{Estimates for $\int_{\Omega}\mathcal{S}_1$ and $\int_{\Gamma}\mathcal{S}_2$.}
In the first term on the right-hand side of~(\ref{eq:sremainder}), the hardest terms to estimate correspond to the cases
$(\bar{l},\bar{k})=(2,1)$ and $(l,k)=(2,1)$. If $(\bar{l},\bar{k})=(2,1)$, then
\[
\begin{array}{l}
\displaystyle
\Big|\int_0^t\int_{\Omega}\t \A^k_i\t^2\partial_tq,_k\t^3v^i\Big|
\leq\int_0^t\int_{\Omega}|\t(\t \A^k_i\t^3v^i)\t\partial_tq,_k|\\
\displaystyle
\leq\int_0^t\|\t^2\A^k_i\|_{\infty}\|\t^3v^i\|_0\|\t\partial_tq,_k\|_0
+\int_0^t\|\t\partial_t \A^k_i\|_{\infty}\|\t^4v^i\|_0\|\t\partial_tq,_k\|_0\\
\displaystyle
\lesssim\sqrt{t}\|\t^2\A^k_i\|_{L^{\infty}_tH^{1.5}_x}\|\t^3v\|_{L^{\infty}_tL^2_x}\|q_t\|_{L^2_tH^2_x}
+\|\int_0^t\t \A^k_i(s)\,ds\|_{H^{1.5}}\|\t^4v^i\|_{L^2_tL^2_x}\|q_t\|_{L^2_tH^2_x}\\
\displaystyle
\lesssim\sqrt{t}\|\nabla^2\Psi_\kappa\|_{L^{\infty}_tH^{2.5}_x}\|\t^3v\|_{L^{\infty}_tL^2_x}\|q_t\|_{L^2_tH^2_x}
+t\|\nabla w\|_{L^2_tH^1_x}\|\t^4v^i\|_{L^2_tL^2_x}\|q_t\|_{L^2_tH^2_x}\\
\displaystyle
\lesssim (t+\sqrt{t})P(\E_{\kappa}).
\end{array}
\]
Assume now $(l,k)=(2,1)$, then
\[
\begin{array}{l}
\displaystyle
\Big|\int_0^t\int_{\Omega}\t^2\partial_t\A^k_i\t q,_k\t^3v^i\Big|
\lesssim\sqrt{t}\|q\|_{L^{\infty}_tH^2_x}\|\t^3v\|_{L^{\infty}_tL^2_x}\|\w\|_{L^2_tH^3_x}\lesssim\sqrt{t}P(\E_{\kappa}).
\end{array}
\]
The second error term is rather 
straightforward: for any $l=2,3$
\[
\begin{array}{l}
\displaystyle
\Big|\int_0^t\int_{\Omega}\t^{3-l}\w\cdot\t^lv\big(\t^3q_t+\t^3\w\cdot v\big)\Big|
\leq\|\t^{3-l}\w\|_{L^{\infty}_tH^{1.5}_x}\int_0^t\|\t^lv\|_0\|\t^3q_t+\t^3\w\cdot v\|_0\\
\displaystyle
\lesssim\sqrt{t}\|\w\|_{L^{\infty}_tH^{2.5}_x}\|\t^lv\|_{L^{\infty}_tL^2_x}\|\t^3q_t+\t^3\w\cdot v\|_{L^2_tL^2_x}
\lesssim\sqrt{t}P(\E_{\kappa}).
\end{array}
\]
If $l=1$, then
\[
\Big|\int_0^t\int_{\Omega}\t^2\w\cdot\t v\big(\t^3q_t+\t^3\w\cdot v\big)\Big|
\lesssim\|Dv\|_{L^{\infty}_tH^{1.5}_x}\|\w\|_{L^{\infty}_tH^2_x}\|\t^3q_t+\t^3\w\cdot v\|_{L^2_tL^2_x}
\lesssim\sqrt{t}P(\E_{\kappa}).
\]
Similar analysis yields:
\[
\sum_{l=1}^3\Big|\int_0^t\int_{\Omega}\t^l\A^s_i\t^{3-l}v^i,_s\big(\t^3q_t+\t^3\w\cdot v\big)\Big|\lesssim\sqrt{t}P(\E_{\kappa}).
\]
As for the error term~(\ref{eq:sremaindergamma}), we start by applying Lemma~\ref{lm:comm} 
to deal with the commutator term. For any $i,j\in\{1,2\}$ set $F=q,_2\A^2_{i}\A^2_{j}$, $G=\t^2w$ and
apply Lemma~\ref{lm:comm} to obtain
\be\label{eq:comm1}
\begin{array}{l}
\displaystyle
\int_0^t\Big|\Lambda_{\kappa}[q,_2(\t^3w\cdot \A^2_{\bullet})\A^2_{\bullet}]
-q,_2(\t^3\Lambda_{\kappa}w\cdot \A^2_{\bullet})\A^2_{\bullet}\big]\Big|^2
\lesssim\int_0^t|q,_2\A^2_{\bullet}:\A^2_{\bullet}|^2_{W^{1,\infty}}|w|_2^2\\
\displaystyle
\lesssim t\sup_{0\leq s\leq t}|q,_2\A^2_{\bullet}:\A^2_{\bullet}|_2^2\sup_{0\leq s\leq t}|w|_2^2
\lesssim tP(\E_{\kappa}),
\end{array}
\ee
where, in order to bound $|w|_{L^{\infty}_tH^2_x}$, we use the equation~(\ref{eq:evolreg}) 
analogously to the bound~(\ref{eq:ht}).
Upon using the Cauchy-Schwarz inequality we get
\[
\Big|\int_0^t\int_{\Gamma}\Lambda_{\kappa}\t^3w\cdot\big[\Lambda_{\kappa}[q,_2(\t^3w\cdot \A^2_{\bullet})\A^2_{\bullet}]
-q,_2(\t^3\Lambda_{\kappa}w\cdot \A^2_{\bullet})\A^2_{\bullet}\big]\Big|
\leq\sqrt{t}P(\E_{\kappa}).
\]
As for the second term on the right-hand side of~(\ref{eq:sremaindergamma}) we obtain
\[
\begin{array}{l}
\displaystyle
\Big|\int_0^t\int_{\Gamma}q,_2(\t^{3-l}(v+w)\cdot\t^l\A^2_{\bullet})\t^3\w\cdot \A^2_{\bullet}\Big|
\leq\int_0^t|q,_2|_{\infty}\big((|v|_2+|w|_2)|\t^3\A^2_{\bullet}|_0\big)|\t^3\Lambda_{\kappa}w|_0\\
\displaystyle
\quad\lesssim\|q\|_{L^{\infty}_tH^{2.5}_x}(\|v\|_{L^{\infty}_tH^{2.5}_x}+|w|_2)
\left|\t\Psi-\text{Id}\right|_3\int_0^t|\t^3\Lambda_{\kappa}w|_0,
\lesssim\sqrt{t}P(\E).
\end{array}
\]
where, the term $|w|_{L^{\infty}_tH^2_x}^2$ is bounded by $P(\E_{\kappa})$ for the same reason as in~(\ref{eq:comm1}).
The last term on the right-hand side of~(\ref{eq:sremaindergamma}) is a sum, and the hardest term to bound is created when $l=0$. 
We must integrate by parts with respect to the time variable, to obtain
\be\label{eq:kappa2}
\kappa^2\int_0^t\int_{\Gamma}\t^3\g_th_tv\cdot\t^3\A^2_{\bullet}
=\kappa^2\int_{\Gamma}\t^3h_tv\cdot\t^3\A^2_{\bullet}\Big|^t_0-\kappa^2\int_0^t\int_{\Gamma}\t^3h_tv_t\cdot\t^3\A^2_{\bullet}
-\kappa^2\int_0^t\int_{\Gamma}\t^3h_tv\cdot\t^3\g_t\A^2_{\bullet}.
\ee
Now observe that 
\beas
\kappa^2\int_{\Gamma}\t^3h_tv\cdot\t^3\A^2_{\bullet}\Big|^t_0
&\lesssim& \kappa^2m_0+\kappa|\t^3h_t|_{L^{\infty}_tL^2_x}|v|_{L^{\infty}_tL^{\infty}_x}\kappa|\int_0^t\g_t\t^3\A^2_{\bullet}|_{L^{\infty}_tL^2_x}\\
&\lesssim&\sqrt{\E_{\kappa}}\sqrt{\E_{\kappa}}\sqrt{t}\kappa|\t^4h_t|_{L^2_tL^2_x}
\lesssim\sqrt{t}P(\E_{\kappa}),
\eeas
where $m_0$ depends only on the initial conditions. As for the remaining three terms on the right-hand side of~(\ref{eq:kappa2}),
they are straightforward to bound using the standard energy estimates.
We arrive at
\[
\kappa^2\Big|\int_0^t\int_{\Gamma}\t^3\g_th_tv\cdot\t^3\A^2_{\bullet}\Big|
\leq\kappa^2 m_0+\sqrt{t}(1+\kappa)P(\E_{\kappa}).
\]
In analogous manner we conclude
\[
\Big|\int_0^t\int_{\Omega}\mathcal{S}_3\,dx\Big|+\Big|\int_0^t\int_{\Gamma}\mathcal{S}_4\,dx'\Big|\lesssim(t+\sqrt{t})P(\E_{\kappa}).
\]
%
%

\subsection{Proof of Theorem~\ref{th:main}}\label{se:proof1}
The polynomial inequality~\eqref{E:POLYNOMIALBOUND} replaces the typically used Gronwall inequality.
Since the constants appearing in~\eqref{E:POLYNOMIALBOUND} are independent of $\kappa$ a standard continuity argument
(see for instance Section 9 of \cite{CoSh06}) yields the existence of a  $\kappa$-independent time $T$ such
that
\[
\E_{\kappa}(t)\leq C\E_{\kappa}(0)\leq C\E(0)+1
\]
for $\kappa$ small enough.

Since $\E(t)\leq\E_{\kappa}(t),$ $t\in[0,T]$ (recall the definitions~(\ref{eq:ALEenergy}) of $\E$ and~(\ref{eq:energy}) of $\E_{\kappa}$),
we obtain the uniform bound  
\[
\E(q^{\kappa},h^{\kappa})\leq C\E(0)+1,
\]
where $(q^{\kappa},h^{\kappa})_{\kappa}$ is a family of solutions to the $\kappa$-regularized 
problem~(\ref{eq:reg}), $0\leq\kappa\leq1$.
Note that the assumptions~(\ref{eq:assumptions}) remain valid (on a possibly smaller) time interval $[0,T]$, as both
$|\t h|_{L_t^{\infty}L^{\infty}_x}$ and $\delta$ are easily controlled by the energy $\E$.
By the fundamental theorem of calculus, it is clear that on a possibly smaller time interval $[0,T]$ we have 
\[
\sup_{0\le t \le T}\mathcal A(t) \le \E_\kappa(0) + T \sup_{0\le t \le T}\E_\kappa(t) \le  \E_\kappa(0) + \frac12,
\]
thus justifying a posteriori the a priori assumption~\eqref{E:APRIORI}.
Thus, passing to the weak limit as $\kappa\to0$ we obtain a solution on the time interval $[0,T]$ which belongs to the space 
$\mathcal{S}(T)$ defined in~(\ref{eq:solspace}). 
Since 
$\mathcal{S}(T)$ embeds compactly into $C^1_tC^0_x\cap C^0_tC^2_x$ the solution is also classical.

\medskip
\noindent
{\bf Uniqueness.} We only present a brief sketch of the uniqueness argument. 
A simple application of the energy method also implies uniqueness of the solution. Assume that $(\tilde{q},\tilde{h})$
also solves~(\ref{eq:reg}) with the corresponding $\tilde{\Psi},\tilde{v},\tilde{w}$. Then the pair ($r,\rho):=(q-\tilde{q}, h-\tilde{h})$
satisfies the following system of equations:
\begin{subequations}
\label{eq:regu}
\begin{alignat}{2}
r_t-A_i^j\big(A^k_ir,_k\big),_j&=(\Delta_{\Psi}-\Delta_{\tilde{\Psi}})q-(v-\tilde{v})\cdot w+\tilde{v}(w-\tilde{w}) && \ \text{ in } \ \Omega \,; \label{eq:heatu} \\
(v-\tilde{v})^i+A^k_ir,_k+\tilde{q},_k(A^k_i-\tilde{A}^k_i)&=0&& \ \text{ in } \ \Omega \,;\label{eq:vequ}\\
r&=0&& \ \text{ on } \ \Gamma\,;\label{eq:bdryregu}\\
\rho_t&=-r,_2&& \ \text{ on } \ \Gamma\,;\label{eq:evolregu}\\
\g_nr&=0&& \ \text{ on } \ \g\Omega_{\text{top}}.\,
\end{alignat}
\end{subequations}
Furthermore, initially $(r(0,x),\rho(0,x'))=(0,0)$. Applying
$\t$ to the identity~(\ref{eq:vequ}), multiplying by
$(\t(v-\tilde{v}))^i$ and integrating over $\Omega$, we derive the first identity
in analogy to the proof of Lemma~\ref{lm:i1}.
Similarly, applying $\g_t$ to~(\ref{eq:vequ}), multiplying by $(v-\tilde{v})^i$
and integrating, we obtain the second energy identity.
The natural quadratic form that emerges is equivalent to
\[
E:=\|\t(v-\tilde{v})\|_{L^2_tL^2_x}^2+\|v-\tilde{v}\|_{L^{\infty}L^2}^2+
\|r\|_{L^{\infty}H^1_x}^2+\|r_t\|_{L^2_tL^2_x}^2+|\rho|_{L^{\infty}_tL^2_x}^2+|\rho_t|_{L^2_tL^2_x}^2.
\] 
Furthermore, we have an a-priori control of the high-order derivatives of the two solutions, i.e. for 
some $M>0$: $\E(q,h)+\E(\tilde{q},\tilde{h})<M$. From here, we can easily prove the polynomial bound
\[
E(t)\leq tP(E(t)),
\]
which in particular, uses the fact that the initial values for $\rho$ and $r$ are $0$.
We infer that $E=0$ and hence the uniqueness follows.

\medskip
\noindent
{\bf Continuity in time.} 
Since $q\in L^2_tH^5_x$ and $q_t \in L^2_tH^3_x$, it follows that $q\in C^0_tH^4_x$; similarly, since $q_{tt} \in L^2_tH^3_x$, then
 $q_t\in C^0_tH^2_x$.   
Passing to the limit as $\kappa\to0$ in \eqref{E:HIMPROVED},  
\be\label{E:HIMPROVED2}
\bar \partial^2  h = \frac{ g^2 \bar \partial v\cdot \tau }{ v \cdot n } \,,
\ee
where $v \cdot n >0$ by the Taylor sign condition.    By passing to the limit as $\kappa \to 0$ in Lemma \ref{optimal_q}, we have that $\Psi\in L^2_t H^{5}_x$,
and we also have that $\Psi \in L^2_t H^{3.5}_x$, from which it follows that $\Psi\in C^0_t H^{4}_x$.
Since $q\in C^0 _tH^4_x$ and , and since $v=-\nabla_\Psi q\in L^2_tH^{4}_x$,
it follows that $v\in L^2_tH^{4}_x \cap  C^0 _tH^{3}_x$; hence, 
 $\t v\cdot\tau \in L^2_tH^{2.5}(\Gamma)$.  Then, since $g$ and  $n$ are in $ L^ \infty _tH^{3}(\Gamma)$,  and $v \cdot n \in  L^ \infty _tH^{2.5}(\Gamma)$, we see
 from  \eqref{E:HIMPROVED2} that
 $$
 h \in L^2_tH^{4.5}(\Gamma)\,.
 $$
Since $h_t = g v\cdot n$ on $\Gamma$, we then have that
 $$
 h_t \in L^2_tH^{3.5}(\Gamma)\,,
 $$
 from which it follows that 
  $$
 h \in C^0_tH^{4}(\Gamma)\,.
 $$
 Since  $h_t = g v\cdot n$ on $\Gamma$, and since $g$ and  $n$ are in $ C^0 _tH^{3}(\Gamma)$ and $v \in C^0_tH^{2.5}(\Gamma)$,  then
 $h_t\in C^0_tH^{2.5}(\Gamma)$.
Using that $h_{tt} = \partial_t\left[g v\cdot n\right]$  on $\Gamma$ and the fact that $v_t \in C^0_tH^{0.5}(\Gamma)$, we also have that
 $h_{tt}\in C^0_tH^{0.5}(\Gamma)$.

It remains to show that $q_{tt}\in C^0_tL^2_x$. From ~\eqref{eq:ALEheat},
\[
q_{tt} = (\Delta q)_t - (v\cdot w)_t.
\]
Given the regularity  already established for $q$, $q_t$, $\Psi$, and $\Psi_t$, we need to establish the regularity for $w_{t}=\Psi_{tt}$.
Since
$h_{tt}\in C^0_tH^{0.5}(\Gamma)$, then $\Psi_{tt} \in C^0_tH^1(\Omega)$, and we find that $q_{tt} \in C^0_tL^2(\Omega)$.

\section{The vanishing surface tension limit}\label{se:vanish}
Local-in-time existence for the Stefan problem with surface tension has been studied in a variety of papers; see, for example, \cite{HaGu1,Ha1, Ra, EsPrSi}.
For any $(q^\sigma_0,h_0^\sigma)\in H^4(\Omega)\times H^{5.5}(\Omega)$ there exists a local-in-time classical solution $(q,h)$ to the Stefan problem with surface tension in 
the harmonic gauge:
\begin{subequations}
\label{eq:ALESIGMA}
\begin{alignat}{2}
q_t- \Delta _\Psi q&=-v\cdot w&& \ \text{ in } \ \Omega \times (0,T] \,,\\
v^i+A^k_iq,_k&=0&&\ \text{ in } \ \Omega  \times (0,T] \,,\\
q&=-\sigma\frac{\t^2h}{(1+|\t h|^2)^{\frac32}}&& \ \text{ on } \ \Gamma \times [0,T] \,,\label{eq:bdryregsigma}\\
\Delta \Psi&=0&& \ \text{ on } \ \Omega  \times [0,T] \,,\\
\Psi&= \text{Id}+ h\, N&& \ \text{ on } \ \Gamma  \times [0,T] \,,\\
\Psi&= \text{Id}&& \ \text{ on } \ \partial\Omega_{top}  \times [0,T] \,,\\
\Psi_t\cdot n(t)&=-v\cdot n(t)&& \ \text{ on } \ \Gamma  \times (0,T] \,,\label{E:NEUMANNSIGMA}\\
v\cdot N&=0&& \ \text{ on } \ \g\Omega_{\text{\text{top}}} \times [0,T] \,,\\
\Psi(0,\cdot)=\Psi_0& \ \   q(0,\cdot)=q^\sigma_0=p_0\circ\Psi_0\,,&&\,.
\end{alignat}
\end{subequations}

With $\sigma>0$, we can prove the following energy identities in the same way as Lemma~\ref{lm:i1}. 
\begin{lemma}\label{lm:isigma}
Let $(q,h)$ be a local-in-time solution to~\eqref{eq:ALESIGMA} defined on the time interval $[0,T_\sigma].$
Then we have the following energy identity:
\beas
F^{\sigma}(q,\Psi)(t)&=&\int_0^t\int_{\Omega}\{\mathcal{R}_1+\mathcal{R}_3+\mathcal{R}_5+\mathcal{S}_1+\mathcal{S}_3\}+
\int_0^t\int_{\Gamma}\{\mathcal{R}_2+\mathcal{R}_4+\mathcal{R}_6+\mathcal{S}_2+\mathcal{S}_4\}\\
&&+\int_0^t\int_{\Gamma}\{\mathcal{R}^{\sigma}_2+\mathcal{R}^{\sigma}_4+\mathcal{R}^{\sigma}_6
+\mathcal{S}^{\sigma}_2+\mathcal{S}^{\sigma}_4\},
\eeas
where 
\be
F^{\sigma}:=\F+
\frac{\sigma}{2}\sum_{a+2b\leq4}\Big||\t\Psi|^{-3/2}J^{-1/2}\t^{a+1}\partial_t^bh\Big|_{L^{\infty}_tL^2_x}^2
+\sigma\sum_{a+2b\leq3}\Big|\t\Psi|^{-3/2}J^{-1/2}\t^{a+1}\partial_t^b h_{t}\Big|_{L^2_tL^2_x}^2, \label{E:FSIGMA}
\ee
with the energy $\mathcal F$ 
and error terms $\mathcal{R}_i$, $i=1,\dots, 6$, $\mathcal{S}_i$, $i=1,\dots, 4$ given by~\eqref{eq:coercive} and Lemma~\ref{lm:i1} respectively, wherein we drop the $\kappa$-dependent terms. 
Furthermore, 
\be\label{eq:r2sigma}
\begin{array}{l}
\displaystyle
\mathcal{R}_2^{\sigma}:=-\sigma\t\left(\frac{\t^2 h}{|\t\Psi |^3}\right)
\t^4\Psi \cdot A^1_{\bullet}\t^4v\cdot  A^2_{\bullet}\\
\displaystyle
\qquad +\sigma\big\{-\t^5 h\cdot \t\big(-\t^4 h_t|\t\Psi |^{-3}\big)
+\t^5 h h_t|\t\Psi |^{-3}\big\}\\
\displaystyle
\qquad+\frac{\sigma}{2}|\t^5 h|^2\g_t(|\t\Psi |^{-3}J^{-1})
+\sigma\t^4 w\cdot A^2_{\bullet}\big[\t^4( \frac{\t^2 h}{|\t\Psi|^3})
-\t^6 h|\t\Psi|^{-3}\big]\\
\displaystyle
\qquad
+\sigma\sum_{l=1}^4a_l\t^{4-l}(w+v)\cdot\t^l A^2_{\bullet}
\t^4 \big(\frac{\t^2 h}{|\t\Psi|^{3}}\big)
\end{array}
\ee
\be\label{eq:r4sigma}
\begin{array}{l}
\displaystyle
\mathcal{R}_4^{\sigma}:=-\sigma\big(\frac{\t^2 h}{|\t\Psi|^3}\big),_1
\t^2\g_t\Psi\cdot A^1_{\bullet}\t^2\g_tv\cdot A^2_{\bullet}\\
\displaystyle
\qquad+\frac{\sigma}{2}|\t^3\g_t h|^2\g_t(|\t\Psi|^{-3}J^{-1})
+\sigma\t^3\g_t w\cdot A^2_{\bullet}\big[\t\g_t(\frac{\t^2 h}{|\t\Psi|^3})-\frac{\t^3 h_t}{|\t\Psi|^3}\big]\\
\displaystyle
\qquad+\sigma\sum_{l+l'\geq1}a_{l,l'}\t^2\g_t \big(\frac{\t^2 h}{|\t\Psi|^2}\big)
\t^l\g_t^{l'}(w+v)\cdot\t^{2-l}\g_t^{1-l'} A^2_{\bullet};
\end{array}
\ee
\be\label{eq:r6sigma}
\begin{array}{l}
\displaystyle
\mathcal{R}_6^{\sigma}:=-\sigma\big(\frac{\t^2 h}{|\t\Psi|^3}\big),_1\Psi_{tt}\cdot A^1_{\bullet}\g_{tt}v\cdot A^2_{\bullet}
+\frac{\sigma}{2}|\t h_{tt}|^2(|\t\Psi|^{-3}J^{-1})_t
-\sigma\t h_{tt} h_{ttt}\t(|\t\Psi|^{-3})\\
\displaystyle
\qquad+\sigma w_{tt}\big(\g_{tt}(\frac{\t^2 h}{|\t\Psi|^3})
- A^2_{\bullet}\frac{\t^2 h_{tt}}{|\t\Psi|^3}\big)
+\sigma\sum_{l=0}^1a_l\g_{tt} \frac{\t^2 h}{|\t\Psi|}\g_t^l(w+v)\cdot\g_t^{2-l} A^2_{\bullet};
\end{array}
\ee
\be\label{E:S2SIGMA}
\begin{array}{l}
\displaystyle
\mathcal{S}_2^{\sigma}
:=\sigma\sum_{l=1}^3a_l\t^{3-l}(w+v)\cdot\t^l A^2_{\bullet}\t^3\g_t\t \big(\frac{\t h}{|\t\Psi|}\big)\\
\displaystyle
\qquad +\sigma\sum_{a+b<4\atop a\leq 3,\,b\leq1}\t^4 w\cdot A^2_{\bullet}\t^{a+1}\g_t^b h\t^{3-a}\g_t^{1-b}\big(|\t\Psi|^{-1}\big)
\end{array}
\ee
\be\label{E:S4SIGMA}
\begin{array}{l}
\displaystyle
\mathcal{S}_4^{\sigma}:=\sigma\sum_{l+l'<2\atop l,l'\leq1}a_{l,l'}\t^{1-l}\g_t^{1-l'}(w+v)\cdot\t^l\g_t^{l'} A^2_{\bullet}\t\g_{tt}
\t \big(\frac{\t h}{|\t\Psi|}\big)\\
\displaystyle
\qquad+\sigma\sum_{l+l'<3\atop l\leq1,\,l'\leq2}b_{l,l'}\t\g_t w\cdot A^2_{\bullet}\t^{l+1}\g_t^{l'} h\t^{1-l}\g_t^{2-l'}\big(|\t\Psi|^{-1}\big).
\end{array}
\ee
\end{lemma}
%
%
%



\begin{remark}
The higher-order energy function $F^\sigma$ is obtained by proceeding in the same way as in the derivation of the energy function $\F_\kappa$ in Section~\ref{su:proofs}.
The essential difference is the nontrivial trace of the term $\t^4 q$ on the boundary $\Gamma$. Since $q=\sigma\mathcal H$ on $\Gamma$
an integration by parts with respect to $x_k$ in the integral 
\[
\int_\Omega A^k_i\t^4q,_k\t^4 v^i
\]
leads to an additional $\sigma$-dependent energy term in~\eqref{E:FSIGMA}. 
\end{remark}

%
\subsection{Nonlinear energy estimates}

In the following proposition we prove the basic energy estimate in analogy to 
Proposition~\ref{pr:basic1}. 
Most importantly, we establish a nonlinear polynomial inequality for the energy $\E_\sigma$ with $\sigma$-independent coefficients.
As a consequence, we show that under the assumptions of Theorem~\ref{th:main2} the time interval $T_\sigma$ is independent of $\sigma.$
%
%
\begin{prop}\label{pr:basic2}
Let $(q_0^{\sigma}, h_0^{\sigma})_{\sigma\geq0}$ be a given family of well-prepared initial conditions
in the sense of Definition~\ref{def:wellprep}. There exists a constant $C$ independent of $\sigma$ and a universal polynomial $P$  
such that for any $t\in [0,T^{\sigma}]$ the following bound holds
\be\label{E:PINEQUALITY}
\E^{\sigma}(t)\leq C\E^{\sigma}(0)+C(t+\sqrt{t})P(\E^{\sigma}).
\ee
In particular, there exists a time $T>0$ {\em independent} of $\sigma$, a constant $C^*>0$
and the solution $(q^{\sigma},\Psi^{\sigma})$ to the Stefan problem with surface tension defined on $[0,T]$
satisfying the bound
\[
\E^{\sigma}(q^{\sigma},\Psi^{\sigma})(t)\leq C^*,\quad0\leq\sigma\leq1, \ \ t\in[0,T].
\] 
\end{prop}
\begin{proof}
In comparison to the estimates for the classical Stefan problem carried over in Section~\ref{se:est1}
the only new error terms to estimate are the terms $\mathcal{R}^{\sigma}_2$, $\mathcal{R}^{\sigma}_4$,
$\mathcal{R}^{\sigma}_6$, $\mathcal{S}^{\sigma}_2$, $\mathcal{S}^{\sigma}_4$
given in the statement of Lemma~\ref{lm:isigma}.
%
%
\subsubsection*{Estimating $\int_0^t\int_\Gamma\mathcal{R}_2^{\sigma}$ defined by~\eqref{eq:r2sigma}}
We start by bounding the first term on the right-hand side of~(\ref{eq:r2sigma}).
\[
\begin{array}{l}
\displaystyle
\sigma\Big|\int_0^t\int_{\Gamma}\t^2\big(\frac{\t h}{|\t\Psi\t h|}\big)
\t^4\Psi\cdot A^1_{\bullet}\t^4v\cdot A^2_{\bullet}\Big|
\lesssim \int_0^tP(|\sqrt{\sigma}\t h|_4)|\sqrt{\sigma}\Psi|_5|v|_3\\
\displaystyle
\lesssim  P(|\sqrt{\sigma}\t h|_{L^{\infty}H^4})|\sqrt{\sigma}\Psi|_{L^{\infty}H^5}\sqrt{t}|v|_{L^2H^3}
\lesssim \sqrt{t}P(\E^{\sigma}\t h).
\end{array}
\]
The second and the third term on the right-hand side of~(\ref{eq:r2sigma}) are estimated analogously and rely on the standard
$L^{\infty}-L^2-L^2$ estimates. 
As for the fourth term on the right-hand side of~(\ref{eq:r2sigma}), note that due to~(\ref{eq:bulletk}) 
\beas
\sigma\Big|\int_0^t\int_{\Gamma}J^{-1}\t^4h_t\big(\t^4(\frac{\t^2h}{|\t\Psi|^3})
-\frac{\t^6h}{|\t\Psi|^3})\big)\Big|
&\lesssim&\int_0^t|\sqrt{\sigma}\t^3h_t|_0\sqrt{\sigma}\big|\t^4(\frac{\t^2 h}{|\t\Psi|^3})
-\frac{\t^6h}{|\t\Psi|^3})\big|_1\\
&\lesssim&\sqrt{t}P(\E^{\sigma}),
\eeas
where the last estimate follows in the standard way: terms with less derivatives are bounded in the $L^{\infty}$-norm and
then by the Sobolev embedding theorem.
In the last term on the right-hand side of~(\ref{eq:r2sigma}), the hardest case to deal with is $l=4$.
Note that
\[
\begin{array}{l}
\displaystyle
\sigma\int_0^t\int_{\Gamma}(v+w)\cdot\t^4 A^2_{\bullet}\t^4\big(\frac{\t^2h}{|\t\Psi|}\big)
=\sigma\int_0^t\int_{\Gamma}(v+w)\cdot\t^4 A^2_{\bullet}\big(\t^6h|\t\Psi|^{-3}\big)\\
\displaystyle
\qquad+\sigma\int_0^t\int_{\Gamma}(v+w)\cdot\t^4 A^2_{\bullet}\sum_{l'=1}^4a_{l'}\t^{6-l'}
h\t^{l'}(|\t\Psi|^{-2})
=:I+II.
\end{array}
\]
The more challenging term to estimate is term $I$. Since $A^2_\bullet = J^{-1}(\t h, -1)$ we have the identity
\begin{align*}
I & = \sigma \int_0^t\int_{\Gamma}(v+w)\cdot\t^4 \left(J^{-1}(\t h, -1)\right)\big(\t^6h|\t\Psi|^{-3}\big)  \\
& = \sigma \int_0^t\int_{\Gamma}J^{-1}(v+w)\cdot (\t^5 h, 0)\left(\t^6h|\t\Psi|^{-3}\right)  +  \sigma\sum_{m=1}^3c_m \int_0^t\int_{\Gamma}\t^m \left(J^{-1}\right)(v+w)\cdot\t^{4-m}(\t h, -1)\big(\t^6h|\t\Psi|^{-3}\big)\\
& = : I_A + I_B
\end{align*}
for some universal constants $c_m\in\mathbb R.$ Note that when $m=4$ term $(v+w)\cdot\t^{4-m}(\t h, -1)$ vanishes since $(\t h,-1)$ is parallel to $\vec n$ and $v\cdot n =- w\cdot n$ by~\eqref{E:NEUMANNSIGMA}.
\begin{align*}
\big|I_A\big|& =\sigma\Big|\int_0^t\int_{\Gamma}J^{-1}(v^1-w^1)\t^5h\t^6h|\t\Psi|^{-3}\Big| =\frac{1}{2}\sigma\Big|\int_0^t\int_{\Gamma}J^{-1}\t(|\t^5h|^2)(v^1-w^1)|\t\Psi|^{-3}\Big|\\
& =\frac{1}{2}\sigma\Big| \int_0^t\int_{\Gamma}|\t^5h|^2\t\left(J^{-1}(v^1-w^1)|\t\Psi|^{-3}\right)\Big| \lesssim tP(\E^{\sigma}),
\end{align*}
where 
we have used the parametric representation of $\Psi$ in terms of $h$ and integrated by parts.
The last inequality is rather standard and follows by estimating
$\t\big((v^1-w^1)|\t\Psi|^{-3}\big)$ in $L^{\infty}$ norm and further via Sobolev inequality, 
where we also use $\sigma|\t^5h|_{L^{\infty}_tL^2_x}\lesssim\E^{\sigma}$.
Terms $I_B$ and $II$
are easily estimated via the standard energy $L^{\infty}-L^2-L^2$ bounds and Sobolev imbedding,
and the same applies to the 
remaining cases $l=1,2,3$. 
When estimating the fourth term on the right-hand side of~(\ref{eq:r2sigma})
first integrate by parts so to remove one $\t$-derivative from $\t^6\Psi$ term and then apply the standard energy
estimates.  
\subsubsection*{Estimating $\int_0^t\int_\Gamma\mathcal{R}_4^{\sigma}$ defined by~\eqref{eq:r4sigma}} The estimates are completely analogous to the ones for 
$\mathcal{R}_2^{\sigma}$.
\subsubsection*{Estimating $\int_0^t\int_\Gamma\mathcal{R}_6^{\sigma}$ defined by~\eqref{eq:r6sigma}}
The first term on the right-hand side of~(\ref{eq:r6sigma}) is estimated analogously to the first term on the right-hand side of~(\ref{eq:r2sigma}).
Note that
\[
\Big|\frac{\sigma}{2}\int_0^t\int_{\Gamma}|\t h_{tt}|^2(|\t\Psi|^{-3}J^{-1})_t\Big|
\lesssim(|\t\Psi|^{-3}J^{-1})_t|_{\infty}\sigma\int_0^t|\t h_{tt}|_2^2
\lesssim t|\t h|_{\infty}|\t h_{\kappa t}|_{\infty}\E^{\sigma}
\lesssim tP(\E^{\sigma}),
\]
where we use Sobolev inequality and the definition of $\E^{\sigma}$
to infer
\[
|\t h_{\kappa t}|_{\infty}^2\lesssim|\t h_{\kappa t}|_1^2\lesssim\int_{\Gamma}(-q,_2)|\t^2 h_{\kappa t}|^2\lesssim\E^{\sigma}.
\]
and similarly
\be\label{eq:1}
|\t h|_{\infty}^2\lesssim\int_{\Gamma}(-q,_2)|\t^2 h|^2\lesssim\E^{\sigma}.
\ee
Space-time integrals of the third and fourth term on the right-hand side of~(\ref{eq:r6sigma}) are bounded in the usual way by $P(\E^{\sigma})$.
To bound the last term on the right-hand side of~(\ref{eq:r6sigma}) we distinguish the cases $l=0$ and $l=1$. 
If $l=1$, by Leibniz rule expand
\[
\g_{tt}\frac{\t^2h}{|\t\Psi|}=\t^2h_{tt}|\t\Psi|^{-1}+2\t^2h_t(|\t\Psi|^{-1})_t
+\t^2h(|\t\Psi|^{-1})_{tt}.
\]
For the first two terms above integrate by parts to move one $\t$ derivative away from $\t^2h_{tt}$ and
$\t^2h_t$.
Then use the 
standard $L^{\infty}-L^2-L^2$ type estimates 
as well as the bound $|\t v_t|_{L^{\infty}_tL^2_x}\lesssim\|q_t\|_{L^{\infty}_tH^{2.5}_x}$ to get
the desired estimate. For the third term on right-hand side above we have
\[
\begin{array}{l}
\displaystyle
\sigma\Big|\int_0^t\int_{\Gamma}\big(\t^2h(|\t\Psi|^{-1})_{tt}\big)\g_t(w+v)\cdot\g_t A^2_{\bullet}\Big|\\
\displaystyle
\qquad\leq|\t^2h|_{\infty}\sqrt{t}|\sqrt{\sigma}(|\t\Psi|^{-1})_{tt}|_{L^2_tL^2_x}|\g_t(v+w)|_{L^{\infty}_tL^2_x}
|\g_t A^2_{\bullet}|_{L^{\infty}_tL^{\infty}_x}
\lesssim\sqrt{t}P(\E^{\sigma}).
\end{array}
\]

\subsubsection*{Estimating $\mathcal{S}^{\sigma}_2$ and $\mathcal{S}^{\sigma}_4$ defined by~\eqref{E:S2SIGMA} and~\eqref{E:S4SIGMA} respectively}
The estimates are straightforward and follow the same principle: terms with least amount of derivatives are bounded via Sobolev
embedding by the $\sigma$-independent energy $\E(q^{\sigma},h^{\sigma})$. 
\par
Summing up the above estimates we prove the first inequality in the proposition. 
The existence of a
$\sigma$-independent time $T$ follows from the standard continuity argument and the fact that constant $C$ in~\eqref{E:PINEQUALITY}
is $\sigma$-independent. Since $\E^{\sigma}(0)\to\E(0)$ as $\sigma\to0$ due to our assumption on initial data, the last statement of the proposition follows. 
\end{proof}
\subsection*{Proof of Theorem~\ref{th:main2}}
Recall the definition~\eqref{E:QHNORM} of $\|(q,h)\|_{C^1_tC^0_x\cap C^0_tC^2_x}$.
Assume that $\|(q^{\sigma},h^{\sigma})-(q^0,h^0)\|_{C^1_tC^0_x\cap C^0_tC^2_x}$ does {\em not} 
converge to $0$ as $\sigma\to0$.
Then there exists an $\epsilon>0$ and a subsequence 
$(\sigma_n)_{n\in\N}$, $\sigma_n\to0$ as $n\to\infty$, such that
\be\label{eq:contr1}
\|(q^{\sigma_n},h^{\sigma_n})-(q^0,h^0)\|_{C^1_tC^0_x\cap C^0_tC^2_x}\geq\epsilon\quad\forall n\in\N.
\ee
Since $\E(q^{\sigma_n},h^{\sigma_n})\leq C$, there exists a subsequence 
of $(q^{\sigma_n},h^{\sigma_n})_n$ 
(without loss of generality indexed again by $(\sigma_n)$) and 
$(\bar{q},\bar{h})\in\mathcal{S}$ such that
\[
(q^{\sigma_n},h^{\sigma_n})\rightharpoonup(\bar{q},\bar{h}),
\quad\text{weakly in}\,\,\,\,\mathcal{S},
\]
where we recall that $\mathcal S$ is defined in~\eqref{eq:solspace}.
Note that the injection operator $I:\mathcal{S}\to C^1_tC^0_x\cap C^0_tC^2_x$ is compact. 
Hence $(q^{\sigma_n},h^{\sigma_n})\to(\bar{q},\bar{h})$
in $C^1_tC^0_x\cap C^0_tC^2_x$ where  $(\bar{q},\bar{h})\Big|_{t=0}=(q_0,h_0)$ due to the
property $3)$ in the Definition~\ref{def:wellprep} of the well-prepared initial data. 
Since $\sigma_n\to0$ as $n\to\infty$, $(\bar{q},\bar{h})$ 
solves the classical Stefan problem
with those initial conditions. From the uniqueness statement of Theorem~\ref{th:main},
we conclude that $(\bar{q},\bar{h})=(q,h)$. Thus 
$(q^{\sigma_n},h^{\sigma_n})\to(q,h)$ in $C^1_tC^0_x\cap C^0_tC^2_x$
contradicting~(\ref{eq:contr1}).

\section{The three-dimensional case}\label{se:3D}

In this section, we briefly sketch how to adapt the analysis of the previous sections to prove theorems
analogous to~Theorem~\ref{th:main} and~\ref{th:main2} in the {\em three} dimensional setting.
We assume now that $\Omega(t)$ is an evolving phase inside the reference domain 
\[
\Omega:=\T^2\times (0,1),
\] 
where $\T^2$ is the $2$-torus. 
Initially at $t=0$ the moving boundary 
\[
\Gamma_0=\T^2\times\{x^3=h_0(x)\} 
\] 
is parametrized as a graph over $\Gamma = \T^2\times\{x^3=0\}$ by the height function $h_0.$ 
The {\it top} boundary $\partial\Omega_{\text{top}}=\T^{2}\times\{x^3=1\}$ 
is fixed and the temperature $p$ satisfies the homogeneous 
Neumann boundary condition on $\partial\Omega$ just like in~(\ref{eq:fixedbdry}).
We parametrize boundary as a graph over $\Gamma$ with the height function $h(t,x')$, 
where $x':=(x^1,x^2)$. Using the harmonic coordinates we can change of variables as in~(\ref{eq:ALE0}) to obtain a fixed 
boundary problem given by~(\ref{eq:ALE}).
The associated energy is given by
\be
\label{eq:ALEenergy3D}
\begin{array}{l}
\displaystyle
\E^{3D}(t)=\E^{3D}(q,h)(t):=\sum_{|\alpha|+2b\leq5}\|\bar{\nabla}^{\alpha}\partial_t^bv\|_{L^2_tL^2_x}^2
+\frac{1}{2}\sum_{|\alpha|+2b\leq4}\|\bar{\nabla}^{\alpha}\partial_t^bv\|_{L^{\infty}_tL^2_x}^2\\
\displaystyle
\quad+\frac{1}{2}\sum_{|\alpha|+2b\leq5}|\sqrt{-q,_2}\bar{\nabla}^{\alpha}\partial_t^bh|_{L^{\infty}_tL^2_x}^2
+\sum_{|\alpha|+2b\leq4}|\sqrt{-q,_2}\bar{\nabla}^{\alpha}\partial_t^bh_{t}|_{L^2_tL^2_x}^2\\
\displaystyle
\quad+\frac{1}{2}\sum_{|\alpha|+2b\leq5}\|\bar{\nabla}^{\alpha}\partial_t^bq+\bar{\nabla}^{\alpha}\partial_t^b\Psi\cdot v\|_{L^{\infty}_tL^2_x}^2
+\sum_{|\alpha|+2b\leq4;}\|\bar{\nabla}^{\alpha}\partial_t^bq_t+\bar{\nabla}^{\alpha}\partial_t^b\Psi_t\cdot v\|_{L^2_tL^2_x}^2.
\end{array}
\ee
In the above definition, $\alpha=(\alpha_1,\alpha_2)$ is a multi-index of order $|\alpha|=\alpha_1+\alpha_2$, whereby $\alpha_1,\alpha_2$ are
non-negative integers. Symbol $\bar{\nabla}$ refers to differentiation in tangential directions, i.e.
$\bar{\nabla}^{\alpha}:=\g_{x^1}^{\alpha_1}\g_{x^2}^{\alpha_2}$. 
The three-dimensional Taylor sign condition for a function $q$ reads:
\be\label{eq:taylor3D}
\min_{x'\in\Gamma}(q,_3)(t,x',0)>0.
\ee
The following theorem holds:  
\begin{theorem}
Let the initial conditions $(q_0,h_0)$ be such that $\E^{3D}(q_0,h_0)<\infty$  and let $q_0$ satisfy the Taylor sign condition~(\ref{eq:taylor3D}).
Then the three-dimensional one-phase classical Stefan problem
is locally-in-time well-posed, i.e. there is a $T>0$ such that 
there exists a unique solution $(q,h)$ with the initial data $(q_0,h_0)$ on  
the time interval $[0,T]$. In addition it satisfies the bound:
\[
\E^{3D}(q,h)\leq 2\E^{3D}(q_0,h_0).
\] 
Furthermore, let $(q_0^{\sigma}, \Psi_0^{\sigma})_{\sigma\geq0}$ be a given family of well-prepared initial conditions
in the sense of Definition~\ref{def:wellprep}. Assume that it satisfies 
the Taylor sign condition~(\ref{eq:taylor3D}) and the corresponding compatibility conditions. By 
$(q^{\sigma},h^{\sigma})_{\sigma\geq0}$ we denote the associated family of solutions to the 
problem~(\ref{eq:ALE}).
There exists a $\sigma$-independent time $T>0$ and a constant $C$ depending only on $(q_0,h_0)$ such that
\[
\E^{3D,\sigma}(q^{\sigma},h^{\sigma})(T)\leq C\quad\sigma\geq0.
\]
for all $\sigma\geq0$. As a consequence, sequence $(q^{\sigma},h^{\sigma})$ converges to 
the unique solution $(q,h)$ of the classical Stefan problem~(\ref{eq:ALE}) with 
$\sigma=0$ in $C_t^1C^0_x\cap C^0_tC^2_x$-norm. 
\end{theorem}
\begin{remark}
Note that the definition of $\E^{3D}$ contains time derivatives. Thus, to make sense out of the
assumption $\E^{3D}(q_0,h_0)<\infty$, we express the time derivatives $\g_tq_0$ and $\g_th_0$ in terms of the spatial derivatives 
as explained in Remark~\ref{re:mainth}.
\end{remark}
\subsection*{Acknowledgements}
MH was supported by the National Science Foundation under grant CMG-0530862
and the European Research Council under grant ERC-240385.
SS was supported by the National Science Foundation under grants DMS-1001850 and DMS-1301380,
and by the Royal Society Wolfson Merit Award.

%

\appendix
\renewcommand{\theequation}{\Alph{section}.\arabic{equation}}
\renewcommand{\thetheorem}{\Alph{section}.\arabic{theorem}}

\section{Modifications of our analysis for a more general initial domain}\label{A:SMOOTHREFERENCE}

In this section we explain how to construct a smooth reference interface for a general graph $\Gamma_0=\{{\bf x} \ | {\bf x} = (x,h_0(x))\}\subset \T^1\times[0,1),$ 
where the size of $|h|_{4.5}$ is not necessarily small.
For any $\varepsilon>0$ we define 
\[
h_0^\varepsilon(x) = \int_{\T^1} h(y) \rho_\varepsilon(x-y), \ x\in\T^1, \ \ \Gamma_0^\varepsilon=\{{\bf x} \ | {\bf x} = (x,h_0^\varepsilon(x))\}\subset \T^1\times[0,1),
\]
and set 
\[
\Omega_0^\varepsilon: = \{(x,y)\in \T^1\times[0,1) \ | \, x\in \T^1, \ h_0^\varepsilon(x) < y <1\}.
\]
Here $\rho_\varepsilon$ is the the standard mollifier defined in Definition~\ref{D:CONVOLUTION} and
the domain $\Omega_0^\varepsilon$ will be our reference domain.
Clearly $h_0^\varepsilon\in C^\infty(\Gamma)$ and for $\varepsilon$ sufficiently small we can parametrize the evolving surface $\Gamma(t)$ as a graph over $\Gamma_0^\varepsilon$
using the outward-pointing unit normal vector field $N^\varepsilon$ to $\Gamma_0^\varepsilon:$
\[
\Gamma(t) = \{{\bf x} \ | {\bf x} = (x,h_0^\varepsilon (x)) + h(t,x) N^\varepsilon(x)\}, \ \ N^\varepsilon(x) = \frac{(\t h_0^\varepsilon,-1)}{\sqrt{1+|\t h_0^\varepsilon |^2}}.
\]
Note that $|h_\varepsilon - h_0|_{4.5} \to 0$ as $\varepsilon\to0.$ 
The construction of the harmonic diffeormorphic extension $\Psi:\Omega_0^\varepsilon \to \Omega(t)$ of the boundary data
\[
\Psi(t,x, h_0^\varepsilon(x)) = (x,h_0^\varepsilon (x)) + h(t,x) N^\varepsilon(x) , \ \ \Psi(t,x,1) = (x,1)
\]
is a simple consequence of the existence theory for the Dirichlet boundary value problems for systems of elliptic partial differential equations, since
for small $\varepsilon$ and small times $t\ge0$ we have 
\[
|\Psi - \text{Id}|_{4.5} \lesssim \varepsilon \ll1.
\]
Using the argument in~\eqref{eq:gauge2} 
the trace estimate~\eqref{DM} is true.
Fixing an $\varepsilon>0$ sufficiently small we drop the $\varepsilon$-notation and refer to the reference curve $\Gamma_0^\varepsilon$ as $\Gamma,$ 
the reference domain $\Omega_0^\varepsilon$ as $\Omega,$ the reference unit normal $N^\varepsilon$ as $N,$ and the reference height $h_0^\varepsilon$ as $\tilde h.$ 
In the harmonic gauge, the Stefan problem takes nearly the same form~\eqref{eq:ALE}:
\begin{subequations}
\label{eq:ALE2}
\begin{alignat}{2}
q_t-A^j_i(A^k_iq,_k),_j&=-v\cdot w&& \ \text{ in } \ \Omega\,,\label{eq:ALEheat2}\\
v^i+A^k_iq,_k&=0&&\ \text{ in } \ \Omega\,,\\
q&=0&& \ \text{ on } \ \Gamma\,,\label{eq:ALEdirichlet2}\\
\Psi_t\cdot n(t)&=-v\cdot n(t)&& \ \text{ on } \ \Gamma\,,\label{eq:ALEneumann2}\\
v\cdot N&=0&& \ \text{ on } \ \g\Omega_{\text{top}}\,,\label{eq:ALEtop2}\\
q(0,\cdot)=q_0=p_0\circ\Psi;& \ \Psi(0,\cdot)=\Psi_0\,,&&\,\label{eq:ALEinitial2},
\end{alignat}
\end{subequations}
where 
\[
\|\Psi_0 - \text{Id}\|_{H^5} \lesssim \varepsilon
\]
and the local coordinate realization of the unit normal $n(t,x)$ takes the more general form:
\[
n(t,x) = \frac{(1-h \mathcal H)\sqrt{1+(\t \tilde h)^2} \, N - \t h T}{\sqrt{(1+(\t \tilde h)^2)(1-h \mathcal H_0)^2 + (\t h)^2}}, \ \ x\in \T^1,
\]
where 
\[
\mathcal H = -\frac{\t^2 \tilde h}{(1+(\t \tilde h)^2)^{3/2}}, \ \ T= \frac{(1, \t \tilde h)}{\sqrt{1+(\t \tilde h)^2}}
\]
stand for the mean curvature and the unit tangent to the reference surface $\Gamma$ respectively.
The proof of Theorem~\ref{th:main} applies to~\eqref{eq:ALE2} in an analogous manner, it is simply more technical.
The main technical novelty is that the tangential vector-fields to the reference surface $\Gamma$ are not given by $\t=\partial_x,$ as 
$\Gamma$ may have a nontrivial curvature in general.
Therefore, in the neighborhood of $\Gamma$ for any $C^1$ function $f:\Omega\to\R$ we define the tangental derivative
\[
\t f = \nabla f\cdot T,
\] 
where $T$ is a local extension of the unit tangent vector field $T$ into the domain $\Omega.$ Choosing a smooth cut-off function $\mu:\Omega\to[0,1]$ 
defined to be $1$ in a neighborhood of $\Gamma$ and $0$ in a neighborhood of $\partial\Omega_{\text{top}},$ we can  replace the operator $\t$ in Lemma~\ref{lm:i1}
by the operator
\[
\mu \t + (1-\mu ) \partial_i,  \ \ i = 1,2.
\]
The ensuing energy identities, energy estimates, and the proof of Theorem~\ref{th:main} follow in an analogous way.

\section{Auxiliary lemmas}

%

\renewcommand{\thelemma}{\Alph{section}.\arabic{lemma}}

We collect some auxiliary estimates in this section that have been used in the proof of the energy estimates.
The following commutator estimate is used in the proof of Proposition~\ref{pr:basic1}.
\begin{lemma}[Lemma 5.1 in~\cite{CoSh10}]\label{lm:comm}
For $F\in W^{1,\infty}(\Gamma)$ and $G,\t G\in L^2(\Gamma)$, there is 
a generic constant $C$ independent of $\kappa$ such that
\[
\big|\Lambda_{\kappa}(F\t G)-f\Lambda_{\kappa}\t G\big|
\leq C|F|_{W^{1,\infty}(\Gamma)}|G|_0,
\]
where $W^{1,\infty}(\Gamma)$ denotes the Sobolev space of functions
$h\in L^{\infty}(\Gamma)$ with weak derivative $\t h\in L^{\infty}(\Gamma)$.
\end{lemma}

Similarly, the following bound is used in estimating some top-order terms in the energy estimates. 
%
\begin{lemma}[Lemma 8.5 in~\cite{CoSh10}]\label{L:TECHNICAL}
Let $H^{\frac12}(\Omega)'$ denote the dual space of $H^{\frac12}(\Omega).$ Then there exists a positive constant $C>0$ such that 
\[
\|\t F\|_{H^{\frac12}(\Omega)'} \le C\|F\|_{H^{\frac12}(\Omega)}.
\]
\end{lemma}

\begin{proof}
The proof is a simple consequence of an interpolation estimate between $L^2(\Omega)$ and $H^1(\Omega)'$ - spaces. The details are given in~\cite{CoSh10}.
\end{proof}



\begin{thebibliography}{99}
\bibitem{Lu}
\textsc{Luckhaus, S.}:
Solutions for the two-phase Stefan problem with the Gibbs-Thomson law for the
melting temperature.
{\em Europ. J. Appl. Math.}
{\bf 1}, 101-111 (1990)
\end{thebibliography}
\end{document}